\pdfoutput=1  

\documentclass[a4paper]{scrartcl}

\usepackage[utf8]{inputenc}
\usepackage[numbers]{natbib}
\usepackage{amsmath}
\usepackage{amssymb}
\usepackage{amsthm}
\usepackage{mathtools}
\usepackage{stmaryrd}
\usepackage{csquotes}
\usepackage{enumitem}
\usepackage{tikz-cd}
\usepackage{tikz}
\usepackage{hyperref}
\usepackage{booktabs}

\usetikzlibrary{calc}
\tikzset{string diagram/.style={%
    scale=0.33
  , baseline={($(current bounding box.center) + (0, -.8ex)$)}
  , string/.style={
      draw
    , preaction={draw, line width=6mm, opacity=0}
    }
  , top label/.style={above, text depth=.25ex}
  , bottom label/.style={below, text height=1.5ex}
  }}
\newcommand{\dotnode}[1]{%
  node[circle, fill=black, minimum size=1mm, inner sep=0mm,%
  label={#1}]{}}

\newcommand{\citestacks}[1]
  {\cite[\href{https://stacks.math.columbia.edu/tag/#1}
              {\texttt{Tag #1}}]
        {stacks-project}}

\theoremstyle{definition}
\newtheorem{dummy}{Dummy Name}[subsection]
\newtheorem{definition}[dummy]{Definition}
\newtheorem{remark}[dummy]{Remark}
\newtheorem{example}[dummy]{Example}
\newtheorem{conjecture}[dummy]{Conjecture}
\theoremstyle{plain}
\newtheorem{theorem}[dummy]{Theorem}
\newtheorem{lemma}[dummy]{Lemma}
\newtheorem{proposition}[dummy]{Proposition}
\newtheorem{corollary}[dummy]{Corollary}

\newcommand{\defeq}{\vcentcolon=}
\newcommand{\turnstile}[1]{\;\vdash_{#1}\;}
\newcommand{\doubleturnstile}[1]{\;\dashv\vdash_{#1}\;}
\newcommand{\ex}[2]{\exists #1 {\,:\,} #2 . \;}
\newcommand{\exx}[1]{\exists #1. \;}

\newcommand{\oftype}{\,{:}\,}

\newcommand{\Zar}{\mathrm{Zar}}
\newcommand{\Zarfp}{{\mathrm{Zar}_\mathrm{fp}}}
\newcommand{\Cris}{\mathrm{Cris}}
\newcommand{\Crisfp}{{\mathrm{Cris}_\mathrm{fp}}}
\newcommand{\fp}{\mathrm{fp}}
\newcommand{\lofp}{\mathrm{lofp}}
\newcommand{\Sch}{\mathrm{Sch}}
\DeclareMathOperator{\Spec}{Spec}
\DeclareMathOperator{\Sh}{Sh}
\DeclareMathOperator{\PSh}{PSh}
\DeclareMathOperator{\Ob}{Ob}
\DeclareMathOperator{\Hom}{Hom}
\DeclareMathOperator*{\colim}{colim}
\newcommand{\hole}{-}
\DeclarePairedDelimiter{\angles}{\langle}{\rangle}

\newcommand{\Set}{\mathrm{Set}}
\newcommand{\id}{\mathrm{id}}
\newcommand{\Id}{\mathrm{Id}}
\DeclarePairedDelimiter{\interpretation}{\llbracket}{\rrbracket}
\DeclarePairedDelimiter{\abs}{\lvert}{\rvert}
\renewcommand{\mod}[2]{{#1\operatorname{-mod}(#2)}}
\newcommand{\dashmod}{\text{-}\allowbreak\mathrm{mod}}
\newcommand{\op}{\mathrm{op}}
\newcommand{\Geom}{\mathrm{Geom}}
\DeclareMathOperator{\scn}{scn}
\DeclareMathOperator{\Gl}{Gl}

\newcommand{\sortsset}[1]{{#1\operatorname{-Sort}}}
\newcommand{\relsset}[1]{{#1\operatorname{-Rel}}}
\newcommand{\funsset}[1]{{#1\operatorname{-Fun}}}

\newcommand{\E}{\mathcal{E}}
\newcommand{\F}{\mathcal{F}}

\newcommand{\TT}{\mathbb{T}}
\newcommand{\EE}{\mathbb{E}}

\newcommand{\NN}{\mathbb{N}}
\newcommand{\ZZ}{\mathbb{Z}}
\newcommand{\QQ}{\mathbb{Q}}
\newcommand{\CC}{\mathbb{C}}
\newcommand{\PP}{\mathbb{P}}

\newcommand{\Ring}{\mathrm{Ring}}
\newcommand{\loc}{\mathrm{loc}}
\newcommand{\nil}{\mathrm{nil}}

\DeclareMathOperator{\inv}{inv}
\newcommand{\AlgStr}[1]{{#1\text{-}\mathrm{AlgStr}}}
\newcommand{\Alg}[1]{{#1\text{-}\mathrm{Alg}}}
\newcommand{\AlgQuot}[2]
  {#1\text{-}\mathrm{Alg}\text{-}#2\text{-}\mathrm{Quot}}
\newcommand{\AlgAlg}[2]
  {#1\text{-}\mathrm{Alg}\text{-}#2\text{-}\mathrm{Alg}}
\newcommand{\surj}{\mathrm{surj}}
\newcommand{\PD}{\mathrm{PD}}
\newcommand{\Ideal}{\mathrm{Ideal}}
\newcommand{\PDIdeal}{\mathrm{PD}\text{-}\mathrm{Ideal}}

\newcommand{\downset}[1]{#1{\downarrow}}
\newcommand{\downdownset}[1]{#1{\downarrow}{\downarrow}}

\newcommand{\divides}{\mid}

\begin{document}

\begin{center}
  \thispagestyle{empty}

  \vspace*{25mm}
  {\huge\bfseries
  Syntactic presentations for glued toposes \\[2mm]
  and for crystalline toposes
  }
  \vspace{30mm}

  {\bfseries Dissertation}
  \vspace{3mm}

  zur Erlangung des akademischen Grades
  \vspace{3mm}

  Dr.~rer.~nat.
  \vspace{10mm}

  eingereicht an der
  \vspace{3mm}

  Mathematisch-Naturwissenschaftlich-Technischen Fakultät
  \vspace{3mm}

  der Universität Augsburg
  \vspace{10mm}

  von
  \vspace{3mm}

  {\bfseries Matthias Hutzler}
  \vspace{20mm}

  \includegraphics[width=0.4\textwidth]{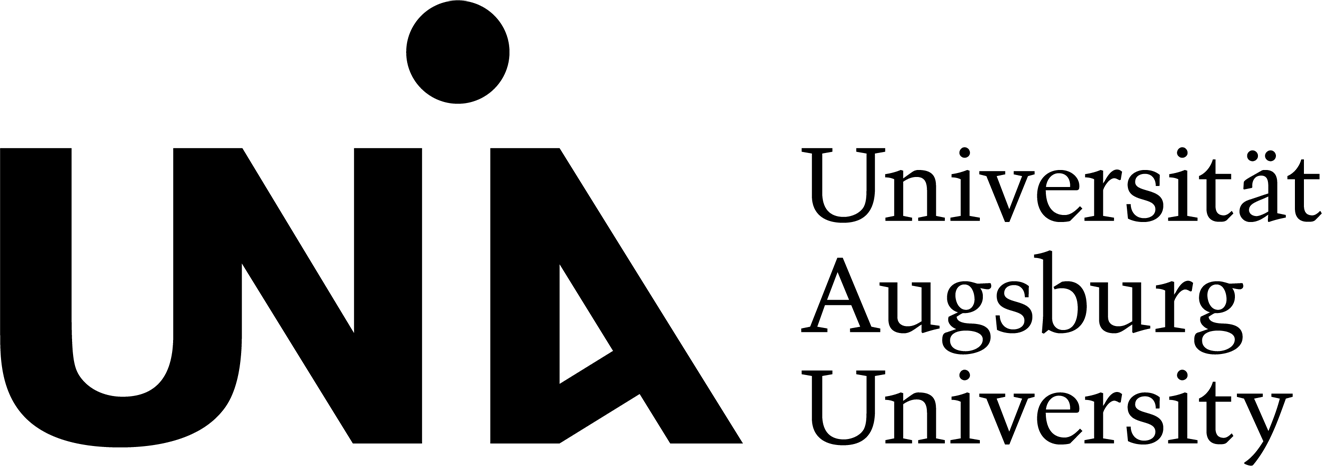}
  \vspace{10mm}

  Augsburg, Oktober 2021
\end{center}

\newpage
\vspace*{170mm}
\noindent
Gutachter: \\
Marc Nieper-Wißkirchen, Universität Augsburg \\
Thierry Coquand, University of Gothenburg
\vspace{5mm}

\noindent
Datum der mündlichen Prüfung: \\
03.12.2021

\newpage
\vspace*{60mm}
\begin{center}
  \textit{
  Im Gedenken an meinen Papa, \\
  der mich auf seine Schultern hob.
  }
\end{center}

\newpage

\section*{Abstract}

We regard a geometric theory classified by a topos
as a syntactic presentation for the topos
and develop tools for finding such presentations.
Extensions (or expansions) of geometric theories,
which can not only add axioms
but also symbols and sorts,
are treated as objects in their own right,
to be able to build up complex theories from parts.
The role of equivalence extensions,
which leave the theory the same up to Morita equivalence,
is investigated.

Motivated by the question
what the big Zariski topos of a non-affine scheme classifies,
we show how to construct a syntactic presentation
for a topos if syntactic presentations for
a covering family of open subtoposes are given.
For this,
we introduce a transformation
of theory extensions
such that
when the result,
dubbed a conditional extension,
is added to a theory,
it requires part of the data a model is made of
only under some condition
given in the form of a closed geometric formula.
We also give a general definition for
systems of interdependent theory extensions,
to be able to talk about compatible syntactic presentations
not only for the open subtoposes in a given cover
but also for their finite intersections.

An important concept
for finding classified theories of toposes
in concrete situations
is that of theories of presheaf type.
We develop several techniques for extending a theory
while preserving the presheaf type property,
and give a list of examples of simple extensions
which can destroy it.

Finally,
we determine a syntactic presentation of
the big crystalline topos of a scheme.
In the case of an affine scheme,
this is accomplished by showing that
the biggest part of the classified theory
is of presheaf type
and transforming the site defining the crystalline topos
into the canonical presheaf site for this theory,
while the remaining axioms induce the Zariski topology.
Then we can apply our results on gluing classifying toposes
to obtain a classified theory even in the non-affine case.

\newpage
\tableofcontents

\newpage

\section{Introduction}

This thesis elaborates in various ways
on the theme in topos theory
that a Grothendieck topos can be viewed as
\emph{the essence of a geometric theory}.
Formally,
one says that a Grothendieck topos $\E$
\emph{classifies} the geometric theory $\TT$
if the models of $\TT$ in any Grothendieck topos $\E'$
correspond to the geometric morphisms from $\E'$ to $\E$.
Since the topos $\E$ is then uniquely determined
by the theory $\TT$,
we can take the dual standpoint
that the theory $\TT$ is a \emph{presentation}
for the topos $\E$.
(All theories that we will meet
will be geometric theories,
and the term topos will always mean Grothendieck topos.)

One reason why we might seek such a presentation for a given topos
is that it can be much more concise
and, we would argue, even more intuitive
than a definition of the same topos by a site.
The prime example for this,
originating in \cite{hakim},
is the big Zariski topos $\E = (\Spec \ZZ)_\Zar$
of the affine scheme $\Spec \ZZ$.
A site of definition for $\E$
is given by the opposite category of
the category of all finitely presentable rings,
equipped with a certain Grothendieck topology
called the Zariski topology,
which involves localizations $A_a$
of a ring $A$ at an element $a \in A$
and the condition that some elements
generate the unit ideal $(1) = A$.
On the other hand,
one can also define $\E$ as
the classifying topos of the theory of local rings.
Here,
a full definition consists simply in
writing down the usual algebraic operations and axioms
defining a (commutative, unitary) ring
and the extra axiom that the ring be local,
stated in the elementary form that if a sum of elements is invertible,
then one of these elements is invertible too.
Of course, to get an actual Grothendieck topos out of this,
one needs the whole machinery of classifying toposes,
but the presentation itself is quite short and very approachable.
To a certain degree,
it is even possible to judge
manipulations of the syntactic presentation correctly,
based on nothing but intuition from elementary algebra.
For example,
the classifying topos stays the same if we add to the list of axioms
a redundant one like $(xy)z = (zy)x$,
but not if we add an axiom like $x = -x$.

Another reason is simply that
a classified geometric theory for a topos $\E$
is a description of the representable functor $\Hom(\hole, \E)$,
that is,
it is a definition for $\E$ by a universal property.
We would like to stress that
since Grothendieck toposes form a 2-category,
the representable functor $\Hom(\hole, \E)$
is in fact a pseudofunctor,
from the 2-category of Grothendieck toposes
to the 2-category of categories.
Such a pseudofunctor comprises
a huge amount of data,
and it is notoriously difficult to keep track of
the coherence conditions that this data must satisfy.
In contrast,
it is simple to check whether a geometric theory is well-defined,
and while it can in general contain arbitrarily big sets
(of relation and function symbols, say) as well,
the examples showing up in practice
are often more or less finitary.

It should be mentioned here that
syntactic presentations of toposes do always exist,
and there is a clear procedure
for constructing a classified geometric theory
out of a given site presentation of a topos.
But a presentation constructed in this way
will of course generally not tell us anything more about the topos
then the site itself does.
Whenever we speak of searching for syntactic presentations,
we intend to find a \emph{concise} presentation,
or one that is interesting for some other reason.

The first of our two main goals,
which will occupy us in Sections
\ref{section-extensions}
and
\ref{section-gluing},
will be to give
a construction on the level of geometric theories
for an operation which is very natural
when viewing toposes as generalized topological spaces,
namely the operation of gluing toposes
along open subtoposes.
More precisely,
our setup will be that $\E$ is a topos
covered by open subtoposes $\E_i = \E_{o(U_i)}$,
and syntactic presentations for the $\E_i$ are given.
Then we ask
how to construct a syntactic presentation for $\E$,
and what additional data might be needed for this.
The appropriate gluing data will consist, unsurprisingly,
of syntactic presentations of the intersections of the $\E_i$,
but not given independently of those for the $\E_i$,
but rather compatible with them, or, really, \emph{extending} them.
Here, the notion of extensions of geometric theories
will be crucial,
which will therefore be investigated first.
The formula we give for the theory classified by $\E$
(see Theorem \ref{theorem-main-gluing})
will then be quite elegant,
it simply adds up all the given theory extensions,
after transforming them into theory extensions
for \enquote{partial models} over the respective open subtoposes.
This gluing technique
is then applied to deduce a syntactic presentation
for the big Zariski topos of a non-affine scheme
(see Theorem \ref{theorem-gluing-Zar})
from the well-known result in the affine case.

Our second objective,
in Sections
\ref{section-presheaf-type}
and
\ref{section-crystalline},
is to give syntactic presentations
for another family of toposes from algebraic geometry,
namely the crystalline toposes of schemes.
These toposes were introduced around 1970
to study crystalline cohomology,
a tool for extracting geometric information from schemes,
similar to de Rham cohomology,
but specifically adapted to schemes
over ground fields of positive characteristic.
While more and more classified theories
for other toposes from algebraic geometry
were found over the years,
a syntactic presentation of the crystalline topos
was up to now missing.
The construction of a crystalline topos depends in fact
not on a single scheme,
as for the Zariski topos,
but on two schemes with some additional structure.
This is reflected in the more involved classified theory we give
(see Theorem \ref{theorem-affine-crystalline-classifies}),
but it is still very much related to the theory of local rings,
and the universal model living in the crystalline topos
consists precisely of its structure sheaf
and some additional data associated to it.
The case of affine base schemes is treated first,
and relies heavily on techniques
for recognizing theories of presheaf type
which we develop for this purpose.
It is then simply another application of the gluing theorem
to generalize to the non-affine case
(see Theorem \ref{theorem-gluing-Cris}),
although some extra care is needed
in constructing an open cover
and a system of syntactic presentations for it.

\section*{Acknowledgements}

I thank Marc Nieper-Wißkirchen
for guiding me towards the topics of this thesis
and for his trusting supervision.
I thank Ingo Blechschmidt,
who played a very special role in my mathemetical socialization,
for tons of encouragement
and for letting so much knowledge
diffuse from his mind to mine.
And I thank Theresa Ritter
for being infinitely patient with me
and for supporting me in every way possible!

\newpage

\section{Extensions of geometric theories}
\label{section-extensions}

\subsection{Background on geometric theories}

For a full definition of geometric theories,
we refer to \cite[Chapter D1.1]{elephant}.
But we want to mention that
a geometric theory $\TT$ can be thought of
as consisting of three \enquote{layers},
first come the \emph{sorts},
then the \emph{relation} and \emph{function symbols},
and finally the \emph{axioms}.
The first two layers are often called the \emph{signature}
of the theory.
What is allowed in each layer
depends on the data in the previous layers:
The set of sorts of $\TT$
is just a set without any additional structure.
The relation and function symbols
have their own signatures,
which in this case just means a list of sorts,
and which we denote
\[ R \subseteq A_1 \times \dots \times A_n
   \qquad\text{and}\qquad
   f : A_1 \times \dots \times A_n \to B
   \rlap{,} \]
where $A_1, \dots, A_n$ and $B$ are sorts of $\TT$.
And the axioms are sequents of the form
\[ \phi \turnstile{x_1 \oftype A_1, \dots, x_n \oftype A_n} \psi
   \rlap{,} \]
to be read as
\enquote{$\phi$ implies $\psi$
in the context $x_1 \oftype A_1, \dots, x_n \oftype A_n$},
where the geometric formulas $\phi$ and $\psi$
can use the relation and function symbols of $\TT$.
The context will sometimes be abbreviated
$\vec{x} \oftype \vec{A}$.

A relation symbol with the empty signature, $n = 0$,
which we might denote $R \subseteq 1$,
is called a \emph{proposition symbol},
and we will rather use the letter $p$ for it.
Similarly,
a function symbol with empty domain, $n = 0$,
is called a \emph{constant symbol},
and instead of $f : 1 \to B$,
we simply denote it as $f : B$,
or rather $c : B$.
A geometric theory is \emph{propositional}
if it has no sorts,
and therefore also no function symbols
and no relation symbols except proposition symbols.

A \emph{model} $M$ of a geometric theory $\TT$
in a topos $\E$
consists of objects
\[ \interpretation{A}_M \in \E \]
for all sorts $A$ of $\TT$,
subobjects
\[ \interpretation{R}_M \subseteq
   \interpretation{A_1}_M \times \dots
   \times \interpretation{A_n}_M \]
for all relation symbols $R$ of $\TT$
and morphisms
\[ \interpretation{f}_M : 
   \interpretation{A_1}_M \times \dots
   \times \interpretation{A_n}_M
   \to \interpretation{B}_M \]
for all function symbols $f$ of $\TT$,
such that the axioms of $\TT$ are fulfilled.
Given a model $M$,
we can not only \emph{interpret}
individual sorts and symbols in $M$,
but also any geometric formula $\phi$
of $\TT$
in a context $x_1 \oftype A_1, \dots, x_n \oftype A_n$,
yielding a subobject
\[ \interpretation{\phi}_M \subseteq
   \interpretation{A_1}_M \times \dots
   \times \interpretation{A_n}_M
   \rlap{.} \]
The requirement that an axiom
\[ \phi \turnstile{\vec{x} \oftype \vec{A}} \psi \]
is fulfilled in $M$
means that there is an inclusion of subobjects
\[ \interpretation{\phi}_M \leq \interpretation{\psi}_M \rlap{.} \]

There is also a notion of \emph{morphism}
between models of the same theory in the same topos,
and the resulting category of $\TT$-models in $\E$
will be denoted
\[ \TT\dashmod(\E) \rlap{.} \]
Furthermore,
the requirement that the axioms of $\TT$ are geometric sequents
ensures that pulling back
the individual parts of a model $M \in \TT\dashmod(\E)$
along a geometric morphism $f : \E' \to \E$
yields a model of $\TT$ in $\E'$.
For a fixed model $M \in \TT\dashmod(\E)$
and varying $f$,
this constitutes a functor
\[ \Geom(\E', \E) \to \TT\dashmod(\E') \rlap{,} \]
and the model $M$ is called a \emph{universal model} of $\TT$
if this functor is an equivalence of categories
for all (Grothendieck) toposes $\E'$.
The topos $\E$ is then called a \emph{classifying topos}
for the geometric theory $\TT$,
and we will call the pair $(\TT, M)$
a \emph{syntactic presentation} of the topos $\E$.

It is a theorem that every (Grothendieck) topos
classifies some geometric theory
and every geometric theory admits a classifying topos,
that is, a universal model in some topos.
The classifying topos of a theory
is also unique up to equivalence,
which justifies writing
\[ \Set[\TT] \]
for a classifying topos of a theory $\TT$.
But it is not at all true that
the classified theory of a topos is unique.
Instead,
two theories admitting universal models in the same topos
are called \emph{Morita equivalent} theories.

There is of course also a notion of provability
for geometric theories,
which we will not define here.
A geometric sequent which is provable in a geometric theory $\TT$
is fulfilled in any model of $\TT$ in any topos,
so it could just as well be added as an axiom of $\TT$.
Two theories $\TT_1$ and $\TT_2$
over the same signature
(same sorts and symbols)
are called \emph{syntactically equivalent}
if every axiom of $\TT_1$ is provable in $\TT_2$
and vice versa.
Note
that this is a much stronger condition on $\TT_1$ and $\TT_2$
than being Morita equivalent,
since Morita equivalent theories
can have different signatures.
We will simply write syntactic equivalence as equality,
\[ \TT_1 = \TT_2 \rlap{.} \]

The universal model of a theory $\TT$
is unique in the sense that
for any two universal models in toposes $\E_1$ and $\E_2$,
there is an equivalence $\E_1 \simeq \E_2$
sending one to the other.
With respect to provability,
it has the following strong property.
A geometric sequent of a theory $\TT$
is fulfilled in the universal model of $\TT$
if and only if
it is provable in $\TT$.

Finally,
we would like to make the point that
when manipulating geometric theories,
one has to think intuitionistically.
This is not the case when the matter is only about axioms;
it is a well-known theorem
that any geometric sequent
which is provable from the axioms of a geometric theory
using full classical first-order logic
is also provable from these axioms in geometric logic.
But if we are also interested in adding sorts and symbols,
the intuitionistic nature shows clearly.
For example,
the theory consisting of a single proposition symbol $p$
has two models (up to isomorphism) in $\Set$,
as the interpretation of $p$ can be either true or false.
The same is true for the theory with
two proposition symbols $p$ and $q$
and the axioms
\[ \top \turnstile{[]} p \lor q
   \quad\text{and}\quad
   p \land q \turnstile{[]} \bot
   \rlap{.}
   \]
But these are two completely different theories,
as a model of the first theory in a topos $\E$
is just an open subtopos of $\E$,
while a model of the second theory
is a decomposition of $\E$ into two subtoposes
which are both open and closed.
And even the categories of models in $\Set$
are not equivalent,
since proposition symbols are allowed to become true
but not to become false
under model homomorphisms.

\subsection{Theory extensions as presentations of geometric morphisms}

\begin{definition}
  A (geometric) \emph{extension} $\EE$ of a geometric theory $\TT$
  consists of a set $\sortsset{\EE}$ of sorts,
  sets $\relsset{\EE}$ and $\funsset{\EE}$
  of relation and function symbols
  over the sorts $\sortsset{\TT} \sqcup \sortsset{\EE}$
  and a set of geometric axioms
  over the sorts $\sortsset{\TT} \sqcup \sortsset{\EE}$
  and the symbols $\relsset{\TT} \sqcup \relsset{\EE}$
  and $\funsset{\TT} \sqcup \funsset{\EE}$.
  We denote $\TT + \EE$ the theory obtained by
  adding these sorts, symbols and axioms to~$\TT$.
  The extension $\EE$ is \emph{localic}
  if $\sortsset{\EE} = \varnothing$;
  it is a \emph{quotient extension}
  if additionally $\relsset{\EE} = \varnothing$
  and $\funsset{\EE} = \varnothing$.
\end{definition}

If $\EE$ is an extension of $\TT$,
we have a forgetful functor
\[ U_\EE : \mod{(\TT + \EE)}{\E} \to \mod{\TT}{\E} \]
for every Grothendieck topos $\E$.
Note that this functor is an isofibration.
Also,
after fixing classifying toposes $\Set[\TT]$ and $\Set[\TT + \EE]$,
the $\TT$-model part of the universal $(\TT + \EE)$-model
induces a canonical geometric morphism
\[ \pi_\EE : \Set[\TT + \EE] \to \Set[\TT] \rlap{,} \]
which in turn acts on generalized points
by the functors $U_\EE$
(up to natural isomorphism).
This is the geometric morphism
\emph{presented} by the extension $\EE$.

The following theorem
says that every geometric morphism
can be presented in this way,
thus generalizing the result that
every Grothendieck topos classifies a geometric theory
to the relative situation
over some base topos $\Set[\TT]$
with an already chosen syntactic presentation.

\begin{theorem}[{\cite[Theorem 7.1.5]{caramello:tst}}]
  \label{theorem-extension-from-geometric-morphism}
  Let a geometric morphism
  \[ p : \E \to \Set[\TT] \]
  to the classifying topos of a geometric theory $\TT$
  be given.
  Then $p$ is, up to isomorphism,
  of the form $\pi_\EE : \Set[\TT + \EE] \to \Set[\TT]$
  for some extension $\EE$ of $\TT$.
  If $p$ is localic (respectively an embedding),
  then we can take $\EE$ to be a localic extension
  (respectively a quotient extension).
\end{theorem}

\begin{proof}
  See \cite[Theorem 7.1.5]{caramello:tst}.
  The case of an embedding,
  which is not mentioned there,
  is instead part of the duality
  between subtoposes and quotient theories
  \cite[Theorem 3.2.5]{caramello:tst}.
\end{proof}

Given two extensions $\EE_1$ and $\EE_2$ of a theory $\TT$,
it is clear that we can also regard $\EE_2$
as an extension of $\TT + \EE_1$ and vice versa.
Then we have
a strict pullback diagram of categories
\[ \begin{tikzcd}
  (\TT + \EE_1 + \EE_2)\dashmod(\E) \ar[r] \ar[d]
  \ar[dr, phantom, very near start, "\lrcorner"] &
  (\TT + \EE_2)\dashmod(\E) \ar[d] \\
  (\TT + \EE_1)\dashmod(\E) \ar[r] &
  \TT\dashmod(\E)
\end{tikzcd} \]
for any topos $\E$,
which is also a weak pullback,
since the the forgetful functors are isofibrations.
This means that $\Set[\TT + \EE_1 + \EE_2]$
is the pullback topos
of $\Set[\TT + \EE_1]$ and $\Set[\TT + \EE_2]$ over $\Set[\TT]$.
In particular,
the product
(as generalized spaces, not as categories)
of two classifying toposes $\Set[\TT_1]$ and $\Set[\TT_2]$
classifies the theory $\TT_1 + \TT_2$.

In the same way in which
we prefer to have a symbol for an extension $\EE$
instead of only for the extended theory $\TT' = \TT + \EE$,
we will also want to regard the data
that is missing in a $\TT$-model
compared to a $(\TT + \EE)$-model
as an object in its own right.

\begin{definition}
  A \emph{model extension} $E$ of a model $M \in \mod{\TT}{\E}$
  along a theory extension $\EE$ of $\TT$
  consists of interpretations for the sorts and symbols of $\EE$
  that extend $M$ to a model $M + E \in \mod{(\TT + \EE)}{\E}$.
  A \emph{homomorphism} of model extensions of $M$
  is a family of maps,
  one for each sort of $\EE$,
  that constitutes a $(\TT + \EE)$-model homomorphism
  when combined with the identity homomorphism of $M$.
  That is,
  the category of model extensions of $M$ along $\EE$
  is isomorphic to the strict preimage of $M$ under $U_\EE$;
  it will be denoted $\mod{\EE}{\E, M}$
  or simply $\mod{\EE}{M}$.
\end{definition}

\begin{remark}
  Note that our terminology here
  is somewhat in conflict with
  the usage of for example \enquote{elementary extension}
  in set theory,
  which means a bigger model of the same theory.
\end{remark}

Using the notion of model extensions,
Theorem \ref{theorem-extension-from-geometric-morphism}
can be formulated as follows.
Given any model $M \in \TT\dashmod(\E)$
of a geometric theory in some topos,
there is always an extension $\EE$ of $\TT$
and a model extension $E$ of $M$ along $\EE$
such that the model $(M + E) \in (\TT + \EE)\dashmod(\E)$
is universal.

\subsection{Equivalence extensions}

\begin{definition}
  An \emph{equivalence extension}
  is an extension $\EE$ of $\TT$
  such that the forgetful functor
  $U_\EE : \mod{(\TT + \EE)}{\E} \to \mod{\TT}{\E}$
  is an equivalence of categories
  for every Grothendieck topos $\E$.
  Equivalently,
  the geometric morphism $\pi_\EE : \Set[\TT + \EE] \to \Set[\TT]$
  presented by $\EE$ is an equivalence.
\end{definition}

If $\EE$ is an equivalence extension of a theory $\TT$,
then it is also an equivalence extension of $\TT + \EE'$,
for any other extension $\EE'$ of $\TT$.
This is clear from the (weak) pullback property
of the category $(\TT + \EE_1 + \EE_2)\dashmod(\E)$.

\begin{lemma}
  \label{lemma-equivalence-extension-if-unique-model-extensions}
  An extension $\EE$ of a theory $\TT$
  is an equivalence extension
  if and only if
  every model $M \in \mod{\TT}{\E}$
  in every Grothendieck topos $\E$
  admits exactly one model extension along $\EE$
  up to isomorphism.
\end{lemma}

\begin{proof}
  The given condition means
  that $U_\EE : \mod{(\TT + \EE)}{\E} \to \mod{\TT}{\E}$
  is bijective on isomorphism classes
  for every topos $\E$,
  which at first sight seems weaker than being an equivalence.
  But it means that the functors
  $\Geom(\E, \Set[\TT + \EE]) \to \Geom(\E, \Set[\TT])$
  induced by the geometric morphism
  $\pi_\EE : \Set[\TT + \EE] \to \Set[\TT]$
  are bijective on isomorphism classes,
  in other words,
  $\pi_\EE$ is an isomorphism in the 1-category
  of toposes and geometric morphisms up to isomorphism,
  which is the same as an equivalence of toposes.
\end{proof}

\begin{remark}
  Lemma \ref{lemma-equivalence-extension-if-unique-model-extensions}
  is a version of the slogan that
  if we know the models of a geometric theory,
  we also know the morphisms between them,
  see \cite[below Lemma B4.2.3]{elephant}.
  However, here it does not suffice to simply say that
  a morphism of $\TT$-models in $\E$
  is the same as a $\TT$-model in the topos $\E^{\to}$
  (the arrow category of $\E$),
  as $U_\EE$ can be bijective on isomorphism classes
  for both $\E$ and $\E^{\to}$
  without being an equivalence for $\E$.
  For example,
  let $\TT = \varnothing$,
  let $\EE$ be the theory of $G$-torsors for a group $G$
  and let $\E = \Set$.
  Denote $\underline{G}$
  the one object groupoid associated to $G$.
  Then both of the functors
  \begin{gather*}
    \underline{G} \simeq \mod{(\TT + \EE)}{\Set} \to
    \mod{\TT}{\Set} \simeq \{*\} \rlap{,}
    \\
    \underline{G} \simeq \underline{G}^{\to}
    \simeq (\TT + \EE)\dashmod(\Set^{\to})
    \to
    \TT\dashmod(\Set^{\to}) \simeq
    \{*\}^{\to} \simeq \{*\} \rlap{,}
  \end{gather*}
  are bijective on isomorphism classes,
  but the first is not an equivalence.
\end{remark}

\begin{lemma}
  \label{lemma-equivalence-extension-of-universal-model}
  Let $M \in \mod{\TT}{\E}$ be a universal model,
  $\EE$ an extension of $\TT$
  and $E$ a model extension of $M$ along $\EE$.
  Then $\EE$ is an equivalence extension
  if and only if
  $M + E \in \mod{(\TT + \EE)}{\E}$ is again universal.
\end{lemma}

\begin{proof}
  This is immediate from the two-out-of-three property
  of equivalences of categories
  in the diagram
  \[ \begin{tikzcd}[row sep=small, column sep=large,
      baseline=(current bounding box.south)]
    & \mod{(\TT + \EE)}{\E'} \ar[dd, "U_\EE"] \\
    \mathrm{Geom}(\E', \E)
    \ar[ru, "\hole^*(M + E)"] \ar[rd, "\hole^*M"'] & \\
    & \mod{\TT}{\E'}
    \rlap{.}
  \end{tikzcd} \qedhere \]
\end{proof}

Apart from syntactic equivalences,
one way to produce an equivalence extension of a theory $\TT$
is to add new symbols
for relations or functions that were already definable in $\TT$.
We now want to show that
all localic equivalence extensions are of this form
up to syntactic equivalence,
meaning that we have a simple syntactic characterization
of equivalence extensions among the localic extensions.

\begin{definition}
  \label{definition-extension-by-definitions}
  Let $\TT$ be a geometric theory,
  let $R_i \subseteq \vec{A}^i$
  and $f_j : \vec{B}^j \to C^j$
  be families of new relation and function symbols
  (with signatures consisting of sorts of $\TT$)
  and let $\{ \vec{x} \oftype \vec{A}^i \,.\; \phi_i(\vec{x}) \}$
  and $\{ \vec{y} \oftype \vec{B}^j, z \oftype C^j \,.\;
  \psi_j(\vec{y}, z) \}$
  be formulas of $\TT$ in corresponding contexts,
  where every $\psi_j$ is provably functional.
  Then the \emph{extension by definitions} for this data
  is the localic extension of $\TT$
  consisting of the symbols $R_i$, $f_j$
  and the axioms
  \[ R_i(\vec{x}) \doubleturnstile{\vec{x} \oftype \vec{A}^i}
  \phi_i(\vec{x})
  \qquad \text{and} \qquad
  f_j(\vec{y}) = z \doubleturnstile{\vec{y} \oftype \vec{B}^j, z \oftype C^j}
  \psi_j(\vec{y}, z) . \]
\end{definition}

\begin{remark}
  \label{remark-recognizing-extensions-by-definitions}
  To show that a given localic extension $\EE$
  of a theory $\TT$ is (syntactically equivalent to)
  an extension by definitions,
  one has to find formulas $\phi_i$, $\psi_j$ of $\TT$,
  the latter provably functional,
  for the relation and function symbols of $\EE$,
  such that the axioms in
  Definition \ref{definition-extension-by-definitions}
  are provable in $\TT + \EE$,
  and furthermore,
  one has to check that all axioms of $\EE$
  are already provable in $\TT$
  if the symbols $R_i$, $f_j$
  are replaced by the formulas $\phi_i$, $\psi_j$.
  However,
  this last check can be omitted
  if a model extension $E$ along $\EE$
  of a universal $\TT$-model $M$ is available,
  since then any $\TT$-sequent which is true in $M + E$
  is trivially also true in $M$
  and therefore provable in $\TT$.
  For the same reason,
  it is then automatic that
  the $\psi_j$ are provably functional.
  This will be relevant in applications of
  Corollary \ref{corollary-localic-gluing}.
\end{remark}

\begin{proposition}
  \label{proposition-extensions-by-definitions}
  A localic extension
  is an equivalence extension
  if and only if it is
  (syntactically equivalent to)
  an extension by definitions.
\end{proposition}

\begin{proof}
  If $\EE$ is an extension by definitions of $\TT$
  then any model $M \in \mod{\TT}{\E}$
  admits exactly one model extension $E$ along $\EE$,
  namely $\interpretation{R_i}_{M + E} \defeq
  \interpretation{\phi_i}_M$
  and $\interpretation{f_i}_{M + E}$ is the morphism with graph
  $\interpretation{\psi_j}_M$.
  So $\EE$ is an equivalence extension by Lemma
  \ref{lemma-equivalence-extension-if-unique-model-extensions}.

  On the other hand,
  if $\EE$ is an equivalence extension of $\TT$,
  then consider a universal model $M \in \mod{\TT}{\Set[\TT]}$
  and the unique (up to isomorphism) extension $E$ of $M$ along $\EE$.
  By Lemma \ref{lemma-equivalence-extension-of-universal-model},
  $M + E$ is a universal model of $\TT + \EE$.
  Now, for every relation symbol $R \subseteq \vec{A}$
  introduced in $\EE$,
  $\interpretation{R}_{M + E}$ is a subobject of
  $\interpretation{A_1}_M \times \dots \times \interpretation{A_n}_M$
  and since $M$ is universal,
  \cite[Theorem 6.1.3]{caramello:tst} tells us that
  there is a formula $\phi_R$ of $\TT$
  with $\interpretation{\phi_R}_M = \interpretation{R}_{M + E}$.
  Similarly, for every new function symbol $f$
  we find a provably functional formula $\psi_f$ of $\TT$
  such that $\interpretation{\psi_f}_M$
  is the graph of $\interpretation{f}_{M + E}$.
  That is, the axioms for defining $R$ by $\phi_R$ and $f$ by $\psi_f$
  are fulfilled for $M + E$
  and therefore provable in $\TT + \EE$.
  Thus,
  we have $\EE = \EE_1 + \EE_2$
  where $\EE_1$ is an extension by definitions
  and $\EE_2$ is a quotient.
  But since both $\EE_1$ and $\EE_1 + \EE_2$ are equivalence extensions,
  $\EE_2$ must be one too,
  and we have (up to syntactic equivalence) $\EE_2 = \varnothing$
  and $\EE = \EE_1$.
\end{proof}

The following lemma is about \enquote{reverting}
an extension $\EE$ of $\TT$ by applying another extension $\EE'$,
subject to extensibility of the universal $\TT$-model along $\EE$.

\begin{lemma}
  \label{lemma-revert-an-extension}
  Let $M \in \mod{\TT}{\E}$ be a universal model
  and $E$ an extension of $M$ along some theory extension $\EE$.
  Then there is a localic extension $\EE'$ of $\TT + \EE$
  and a model extension $E'$ of $M + E$ along $\EE'$
  such that $(M + E + E') \in \mod{(\TT + \EE + \EE')}{\E}$
  is again universal.
  (In particular, $\EE + \EE'$ is an equivalence extension.)
  If $\EE$ is localic,
  $\EE'$ can be chosen as a quotient;
  if $\EE$ is a quotient,
  $\EE' = \varnothing$ fits the bill.
  \[ \begin{tikzcd}[row sep=small]
    \TT + \EE + \EE' \\
    \TT + \EE \ar[u, no head] \\
    \TT \ar[u, no head]
    \ar[uu, bend left=60, no head, "\sim"] 
  \end{tikzcd} \]
\end{lemma}

\begin{proof}
  Fix a classifying topos $\E_{\TT + \EE}$ for $\TT + \EE$.
  We have the forgetful geometric morphism $u : \E_{\TT + \EE} \to \E$,
  and $M + E \in \mod{(\TT + \EE)}{\E}$ corresponds to
  a section (up to isomorphism) $s : \E \to \E_{\TT + \EE}$ of $u$.
  As such, $s$ is localic (since $u \circ s$ is).
  By Theorem \ref{theorem-extension-from-geometric-morphism},
  it is therefore presented by some localic extension $\EE'$
  of $\TT + \EE$,
  that is,
  we have a universal model of $\TT + \EE + \EE'$ in $\E$
  extending $M + E$, as required.

  The two special cases follow since
  a section of a localic geometric morphism is an embedding
  and a section of an embedding is an equivalence,
  but we can also show them more directly.
  If $\EE = \QQ$ is a quotient,
  then the existence of $E$ just means that
  the new axioms are fulfilled in $M$.
  But this means that they were already provable in $\TT$
  and $M = M + E$ is indeed a universal model of $\TT + \QQ$.
  For $\EE$ localic,
  we have to deal with the new relation symbols $R$
  and function symbols $f$.
  From \cite[Theorem 6.1.3]{caramello:tst}
  we know that
  $\interpretation{R(\vec{x})}_{M + E} =
  \interpretation{\phi_R(\vec{x})}_M$
  and
  $\interpretation{f(\vec{x}) = y}_{M + E} =
  \interpretation{\phi_f(\vec{x}, y)}_M$
  for some formulas $\phi_R$, $\phi_f$,
  where $\phi_f$ is provably functional.
  So taking for $\EE'$ the axioms
  $R(\vec{x}) \doubleturnstile{\vec{x}} \phi_R(\vec{x})$
  and
  $f(\vec{x}) = y \doubleturnstile{\vec{x}, y} \phi_f(\vec{x}, y)$,
  we obtain as $\EE + \EE'$ an extension by definitions.
\end{proof}

As a corollary,
we find that an equivalence between two theories,
and more generally an equivalence between
two extensions of some base theory
(meaning that they present the same geometric morphism),
can always be captured in a syntactic way.

\begin{corollary}
  \label{corollary-diagonal-extension}
  Let $\EE_1$, $\EE_2$ be two equivalent extensions of a theory $\TT$,
  that is,
  there exists a model $M \in \mod{\TT}{\E}$
  with extensions to universal models
  $M + E_1$ and $M + E_2$ of $\TT + \EE_1$ respectively $\TT + \EE_2$
  in the same topos $\E$.
  Then there is a localic extension $\EE_{1, 2}$
  of $\TT + \EE_1 + \EE_2$
  such that both $\EE_2 + \EE_{1, 2}$ and $\EE_1 + \EE_{1, 2}$
  are equivalence extensions
  (of $\TT + \EE_1$ respectively $\TT + \EE_2$)
  \[ \begin{tikzcd}[column sep=tiny]
    & \TT + \EE_1 + \EE_2 + \EE_{1, 2} & \\
    \TT + \EE_1 \ar[ur, no head, "\sim"] & &
    \TT + \EE_2 \ar[ul, no head, "\sim"'] \\
    & \TT \ar[ul, no head] \ar[ur, no head] &
  \end{tikzcd} \]
  and there is a model extension $E_{1, 2}$ along $\EE_{1, 2}$
  such that $M_0 + E_1 + E_2 + E_{1, 2}$ is universal.
  If $\EE_1$ or $\EE_2$ is localic,
  then $\EE_{1, 2} = \QQ_{1, 2}$ can be chosen as
  a quotient extension.
\end{corollary}

\begin{proof}
  Regard $\EE_2$ as an extension of $\TT + \EE_1$ for the moment,
  then we have an extension $E_2$ of the universal model $M + E_1$,
  so by Lemma \ref{lemma-revert-an-extension}
  there exists a localic extension $\EE_{1, 2}$
  of $\TT + \EE_1 + \EE_2$
  together with a model extension $E_{1, 2}$
  of $M + E_1 + E_2$
  such that $M + E_1 + E_2 + E_{1, 2}$ is universal.
  Then, by Lemma \ref{lemma-equivalence-extension-of-universal-model},
  both $\EE_2 + \EE_{1, 2}$ and $\EE_1 + \EE_{1, 2}$
  are equivalence extensions, as required.
  If $\EE_1$ is localic,
  we get a quotient extension $\EE_{1, 2} = \QQ_{1, 2}$ from
  Lemma \ref{lemma-revert-an-extension},
  and if $\EE_2$ is localic, we swap the two.
\end{proof}

A completely different proof of
Corollary \ref{corollary-diagonal-extension}
in the absolute case $\TT = \varnothing$
can be found in \cite[Theorem 5.1]{syntactic-morita}.

\begin{definition}
  \label{definition-diagonal-extension}
  We call an extension $\EE_{1, 2}$ as in
  Corollary \ref{corollary-diagonal-extension}
  a \emph{diagonal extension} for $\EE_1$ and $\EE_2$
  over $\TT$,
  because it presents the diagonal geometric morphism
  of the pullback topos
  \[ \Set[\TT + \EE_1 + \EE_2] =
     \Set[\TT + \EE_1] \times_{\Set[\TT]} \Set[\TT + \EE_2] =
     \E \times_{\Set[\TT]} \E
     \rlap{.} \]
\end{definition}

Diagonal extensions are not unique.
For example,
if $\TT_1 = \angles{p_1} + \angles{p_2}$
and $\TT_2 = \angles{p_3} + \angles{p_4}$
both consist of two proposition symbols,
there are clearly two different diagonal quotient extensions
for $\TT_1$ and $\TT_2$ over $\varnothing$.
Corollary \ref{corollary-diagonal-extension}
produces a diagonal extension
in accordance with the two universal models
$M_0 + E_1$ and $M_0 + E_2$ living in the same topos.

Another easy consequence of
Lemma \ref{lemma-revert-an-extension}
is that any object of the classifying topos of a theory
can be introduced into the theory as a new sort
by means of an equivalence extension.

\begin{corollary}
  Let $M \in \mod{\TT}{\E}$ be a universal model
  and let $X \in \E$ be any object.
  Then there is an equivalence extension $\EE$ of $\TT$
  containing exactly one new sort $A$,
  such that the unique (up to unique isomorphism)
  model extension $E$ of $M$ along $\EE$
  interprets $A$ as $X$.
\end{corollary}

\begin{proof}
  Apply Lemma \ref{lemma-revert-an-extension}
  to the extension $\EE_1$ adding nothing but the sort $A$
  (and the model extension $E_1$
  given by $\interpretation{A}_{E_1} \defeq X$)
  to obtain a localic extension $\EE_2$,
  and set $\EE \defeq \EE_1 + \EE_2$.
\end{proof}


\subsection{Some examples of equivalence extensions}

Giving examples of localic equivalence extensions
does not seem necessary after
Proposition \ref{proposition-extensions-by-definitions},
but here are some non-localic equivalence extensions
that will be of use later.

\begin{example}
  \label{example-materialize-subobject}
  Let $A$ be a sort of a geometric theory $\TT$
  and let $\phi(x)$ be a geometric formula
  in the context $x \oftype A$.
  Then there is an equivalence extension of $\TT$
  consisting of
  a sort $S_\phi$,
  a function symbol $\iota : S_\phi \to A$
  and the axioms
  \[ \iota(x) = \iota(y) \turnstile{x, y \oftype S_\phi} x = y
     \qquad\text{and}\qquad
     \ex{x}{S_\phi} \iota(x) = y \doubleturnstile{y \oftype A} \phi(y)
     \rlap{,} \]
  the first of which forces the interpretation of $\iota$ in a model
  to be a monomorphism,
  while the second ensures that
  $\interpretation{S_\phi} \hookrightarrow \interpretation{A}$
  is the same subobject as
  $\interpretation{\phi} \hookrightarrow \interpretation{A}$.
\end{example}

\begin{example}
  \label{example-materialize-quotient}
  Let $A$ be a sort of a geometric theory $\TT$
  and let $x \sim x'$ be a geometric formula
  in the context $x, x' \oftype A$
  such that the usual axioms of an equivalence relation
  (which are Horn sequents)
  are provable for $\sim$.
  Then there is an equivalence extension of $\TT$
  consisting of a sort $A/{\sim}$,
  a function symbol $\pi : A \to A/{\sim}$
  and the axioms
  \[ \top \turnstile{y \oftype A/{\sim}} \ex{x}{A} \pi(x) = y
     \qquad\text{and}\qquad
     \pi(x) = \pi(x') \doubleturnstile{x, x' \oftype A} x \sim x'
     \rlap{.} \]
  These force the interpretation
  $\interpretation{A} \to \interpretation{A/{\sim}}$
  of $\pi$
  to be an epimorphism
  with kernel pair
  $\interpretation{\sim} \rightrightarrows \interpretation{A}$.
\end{example}

\begin{remark}
  An intermediate notion between
  syntactic equivalence and Morita equivalence
  is the notion of (geometric) \emph{bi-interpretability},
  see \cite[Definition 2.1.13]{caramello:tst}.
  Two geometric theories are bi-interpretable
  if and only if their syntactic sites are equivalent categories.
  So while this notion is still syntactic in nature,
  it does not assume any previously given relation
  between the signatures of the two theories.
  For example, a sort $A$ of the first theory,
  which is represented in the syntactic site
  by the object $\{x \oftype A.\; \top\}$,
  can correspond to any formula in context
  $\{ \vec{y} \oftype \vec{B}.\; \phi(\vec{y}) \}$
  of the second theory.

  The equivalence extension in
  Example \ref{example-materialize-subobject}
  induces a bi-interpretation,
  interpreting the new sort $S_\phi$
  by the formula in context
  $\{x \oftype A.\; \phi(x) \}$
  of $\TT$.
  But already Example \ref{example-materialize-quotient}
  shows that bi-interpretability is a stronger condition
  than Morita equivalence,
  since the formula in context $\{x \oftype A/{\sim}.\; \top\}$
  can not be expressed as any formula in context of $\TT$.
\end{remark}

\begin{example}
  \label{example-import-a-set}
  Given a set $A$,
  define a theory $\underline{A}$
  (named after its only sort)
  consisting of
  a sort $\underline{A}$,
  constant symbols $c_a : \underline{A}$
  for every element $a \in A$
  and the axioms
  \[ c_a = c_{a'} \turnstile{[]} \bot
     \quad\text{for $a \neq a' \in A$},
     \qquad \qquad
     \top \turnstile{x : \underline{A}} \bigvee_{a \in A} (x = c_a)
     \rlap{.} \]
  One can check that the unique model (up to unique isomorphism)
  in any Grothendieck topos is the constant sheaf
  associated to the set $A$,
  which is also denoted $\underline{A}$.
  The theory $\underline{A}$
  is therefore Morita-equivalent to the empty theory,
  and adding $\underline{A}$ to any given theory
  is an equivalence extension.
  In other words,
  we can always \emph{import} a set $A$
  into our theory
  without changing it up to Morita-equivalence.

  If we have a function like $f : A \to B$
  or a relation like $R \subseteq A$,
  we can import it together with the respective sets.
  Namely,
  after adding a function symbol
  $\underline{f} : \underline{A} \to \underline{B}$
  or a relation symbol
  $\underline{R} \subseteq \underline{A}$,
  the axioms
  \[ \top \turnstile{[]} \underline{f}(c_a) = c_{f(a)}
     \quad\text{for $a \in A$} \]
  respectively
  \[ \top \turnstile{[]} \underline{R}(c_a)
     \quad\text{for $a \in R$},
     \qquad \qquad
     \underline{R}(c_a) \turnstile{[]} \bot
     \quad\text{for $a \in A \setminus R$}
     \]
  produce an extension by definitions
  in presence of the axioms of $\underline{A}$.

  Another perspective on this is
  that if we have any model of a geometric theory $\TT$ in $\Set$,
  then by Lemma \ref{lemma-revert-an-extension}
  there is a localic extension to a universal model in $\Set$,
  which then yields,
  by pulling back along the unique geometric morphism to $\Set$,
  the unique model of the extended theory in any topos.
\end{example}

\newpage

\section{Gluing classifying toposes}
\label{section-gluing}

\subsection{Introduction}

In this section
we explicitly construct a geometric theory
classified by a given topos
from a cover by open subtoposes
with known classified theories.

We already saw in
Section \ref{section-extensions}
that taking the product of two toposes,
regarded as generalized spaces,
corresponds to the simplest possible operation
involving two unrelated theories,
namely forming the sum $\TT_1 + \TT_2$.
The present topic is a generalization of
the dual question
what the coproduct of two toposes classifies.
This is much less obvious,
since a geometric morphism
$\E \to \Set[\TT_1] \amalg \Set[\TT_2]$
will neither define a $\TT_1$-model
nor a $\TT_2$-model in $\E$.
Instead,
it first of all defines a decomposition of $\E$
into two clopen subtoposes,
and then a $\TT_1$-model in one and a $\TT_2$-model in the other
of these subtoposes.

Another good example is
the big Zariski topos of
the projective line $\PP^1_K$
over a ring $K$.
The big Zariski topos $X_\Zar$ of an affine scheme $X = \Spec R$
classifies the geometric theory of local $R$-algebras.
Now,
$\PP^1_K$ is not affine,
but it can be covered by two copies of the affine line,
$\Spec K[X]$ and $\Spec K[Y]$,
such that the intersection is $\Spec K[X, Y]/(XY - 1)$.
We will see
that this induces an open cover
of the big Zariski topos $(\PP^1_K)_\Zar$
by open subtoposes $(\Spec K[X])_\Zar$ and $(\Spec K[Y])_\Zar$,
which both classify local $K$-algebras with one distinguished element.
This suggests that
$(\PP^1_K)_\Zar$ classifies
local $K$-algebras which are equipped with an element $X$
\emph{or} with an element $Y$,
where the two possibilities are not mutually exclusive
but rather,
their intersection is described by the condition $XY = 1$.
A formulation of this idea
as a geometric theory,
which indeed presents $(\PP^1)_\Zar$,
is given in Proposition \ref{proposition-gluing-projective-line}.



\subsection{Conditional extensions}

If a topos $\E$ has an open subtopos $\E_1 \subseteq \E$
which classifies some theory $\TT_1$,
then a geometric morphism $f : \tilde{\E} \to \E$
does not give us a model of $\TT_1$ in $\tilde{\E}$,
so to find a classified theory for $\E$,
we should not look among extensions of $\TT_1$.
However, $f$ does give us a model of $\TT_1$
in some open subtopos of $\tilde{\E}$,
namely in the preimage of $\E_1$.
We now show how to capture syntactically
the requirement of
\enquote{a model in some open subtopos}.

To avoid technical complications,
we exclude function symbols from our discussion.
Recall
that a function symbol $f : \vec{A} \to B$
can be considered an abbreviation for a relation symbol
$R \subseteq \vec{A} \times B$
together with axioms ensuring that $R$ is provably functional,
as long as we don't mind replacing axioms
with nested function applications
like $g(f(x)) = z$
by versions with auxiliary variables
like $\exists y .\, (f(x) = y) \land (g(y) = z)$.
That is,
function symbols can be considered \enquote{syntactic sugar}.

\begin{definition}
  \label{definition-conditional-extension}
  Let $\TT$ be a geometric theory,
  $\phi$ a closed formula of $\TT$
  and $\EE$ an extension
  (without function symbols)
  of $\TT + \phi$.
  We define the \emph{conditional extension} $\EE/\phi$ of $\TT$
  to consist of the following.
  \begin{itemize}
    \item
      For every sort $S \in \sortsset{\EE}$ of $\EE$,
      a sort $S$
      and the axiom $\top \turnstile{x : S} \phi$.
    \item
      For every relation symbol
      $R \subseteq S_1 \times \dots \times S_n$ of $\EE$,
      a relation symbol $R \subseteq S_1 \times \dots \times S_n$
      and the axiom $R(\vec{x}) \turnstile{\vec{x} : \vec{S}} \phi$.
    \item
      For every axiom $\psi_1 \turnstile{\Gamma} \psi_2$ of $\EE$,
      the axiom
      $\psi_1 \land \phi \turnstile{\Gamma} \psi_2$.
  \end{itemize}
\end{definition}


If the theory $\TT + \phi + \EE$
asks us to specify a model of $\TT$ that satisfies $\phi$
and furthermore a model extension along $\EE$,
then the theory $\TT + \EE/\phi$
instead asks us to specify a model of $\TT$
and a model extension along $\EE$
\emph{in case} our $\TT$-model happens to satisfy $\phi$
--- or, more geometrically speaking,
\emph{wherever} it satisfies $\phi$.

\begin{lemma}
  \label{lemma-conditional-extension-first-properties}
  The assignment $\EE \mapsto \EE/\phi$
  is well-defined with respect to syntactic equivalence.
  Furthermore,
  we have the following syntactic equivalences.
  \begin{enumerate}[label=(\roman*)]
    \item
      $\TT + \EE/\phi + \phi = \TT + \phi + \EE$
    \item
      $\TT + \EE_1/\phi + \EE_2/\phi = \TT + (\EE_1 + \EE_2)/\phi$
  \end{enumerate}
\end{lemma}

\begin{proof}
  To show that the construction is well-defined,
  let $\psi \turnstile{\Gamma} \chi$ be a geometric sequent
  which is provable in $\TT + \phi + \EE$.
  If it was added as an axiom to $\EE$,
  the axiom $\psi \land \phi \turnstile{\Gamma} \chi$
  would be added to $\EE / \phi$,
  so we must show that this sequent
  is already provable in $\TT + \EE/\phi$.
  This is equivalent to showing that
  $\psi \turnstile{\Gamma} \chi$
  is provable in $\TT + \EE/\phi + \phi$,
  so we are done if (i) is true.
  For (i),
  we observe that the additional axioms of $\EE / \phi$
  are indeed trivial
  and $\psi_1 \land \phi \turnstile{\Gamma} \psi_2$
  is indeed equivalent to $\psi_1 \turnstile{\Gamma} \psi_2$
  if our theory already contains $\phi$ as an axiom.
  For (ii),
  Definition \ref{definition-conditional-extension}
  produces the exact same axiomatization for both sides.
\end{proof}

\begin{proposition}
  \label{proposition-conditional-model-extensions}
  Let $\EE$ be an extension of $\TT + \phi$,
  $M \in \mod{\TT}{\E}$
  and $\F \defeq \E_{o(\interpretation{\phi}_M)}$
  the open subtopos of $\E$
  corresponding to the subterminal object
  $\interpretation{\phi}_M \subseteq 1_\E$.
  Then there is an equivalence of categories
  \[ \mod{\EE}{\F, M|_\F} \simeq \mod{(\EE/\phi)}{\E, M}
     \rlap{.} \]
\end{proposition}

\begin{proof}
  The open inclusion geometric morphism $i : \F \to \E$
  admits a further (full and faithful) left adjoint $i_! : \F \to \E$
  with essential image $\E / U$,
  where $U \defeq \interpretation{\phi}_M$.
  The composition $i_! \circ i^*$ is $X \mapsto X \times U$.
  \[ \begin{tikzcd}
    \E/U \simeq \F
    \ar[r, hook, shift left=2, "i_!"]
    \ar[r, phantom, "\rotatebox{90}{$\vdash$}"]
    & \E
    \ar[l, shift left=2, "i^*"]
  \end{tikzcd} \]

  For a sort $S$ of $\EE$,
  the axiom $\top \turnstile{x : S} \phi$ of $\EE / \phi$
  means precisely that
  for every $E \in (\EE/\phi)\dashmod(M)$,
  the object $\interpretation{S}_{M + E}$ must lie in
  the full subcategory $\E / U$,
  so we have established the equivalence
  for the case that $\EE$ contains only sorts.
  The quasi-inverse is of course
  the restriction functor $i^* : \E \to \F$,
  applied to the interpretation of every sort.
  For a relation symbol $R \subseteq \vec{S}$ of $\EE$
  (where the list $\vec{S}$ may contain sorts
  from both $\TT$ and $\EE$),
  the axiom $R(\vec{x}) \turnstile{\vec{x} : \vec{S}} \phi$
  similarly means
  $\interpretation{R}_{M + E} \subseteq \interpretation{\vec{S}}_{M + E} \times U$,
  so that $\interpretation{R}_{M + E}$
  corresponds to a subobject of
  $i^*(\interpretation{\vec{S}}_{M + E}) =
  \interpretation{\vec{S}}_{M|_\F + E|_\F}$.
  Finally, the modified axiom
  $\psi_1 \land \phi \turnstile{\Gamma} \psi_2$
  (or equivalently
  $\psi_1 \land \phi \turnstile{\Gamma} \psi_2 \land \phi$)
  is satisfied by $M + E$
  if and only if
  $\interpretation{\psi_1} \times U \subseteq
  \interpretation{\psi_2} \times U$,
  that is, if and only if
  $\psi_1 \turnstile{\Gamma} \psi_2$
  is satisfied by $M|_\F + E|_\F$.
\end{proof}

For a model $M \in \TT\dashmod(\E)$
and an extension
$E \in \EE\dashmod(M|_{\E_{o(\interpretation{\phi}_M)}})$,
the corresponding extension of $M$ along $\EE/\phi$
will be denoted $E/\phi$.
That means that we have
\[ (M + E/\phi)|_{\E_{o(\interpretation{\phi}_M)}} \cong
   M|_{\E_{o(\interpretation{\phi}_M)}} + E \]
as models of $\TT + \EE/\phi + \phi = \TT + \phi + \EE$.

\begin{lemma}
  \label{lemma-conditional-extension-equivalence-extension}
  An extension $\EE$ of $\TT + \phi$
  is an equivalence extension if and only if
  $\EE/\phi$ is an equivalence extension of $\TT$.
\end{lemma}

\begin{proof}
  If $\EE/\phi$ is an equivalence extension of $\TT$,
  then it is also an equivalence extension of $\TT + \phi$.
  But by
  Lemma \ref{lemma-conditional-extension-first-properties} (i),
  $\EE/\phi$ is syntactically equivalent to $\EE$
  as an extension of $\TT + \phi$.

  Conversely,
  if $\EE$ is an equivalence extension of $\TT + \phi$,
  then by
  Lemma \ref{lemma-equivalence-extension-if-unique-model-extensions},
  and
  Proposition \ref{proposition-conditional-model-extensions}
  every $\TT$-model admits a unique (up to isomorphism)
  model extension along $\EE/\phi$,
  so we are done by the other direction of
  Lemma \ref{lemma-equivalence-extension-if-unique-model-extensions}.
\end{proof}

\subsection{Systems of theory extensions}

Given two equivalence extensions $\EE_1, \EE_2$ of a theory $\TT$,
their sum $\EE_1 + \EE_2$ is again
an equivalence extension of $\TT$.
The same is true if $\EE_2$ is not an extension of $\TT$
but of $\TT + \EE_1$,
that is, if $\EE_2$ depends on $\EE_1$.
To formulate such statements in greater generality,
we first need to clarify
in which ways an extension can be built up
from smaller extensions,
possibly with dependencies among them.

\begin{definition}
  Let $\TT$ be a geometric theory.
  A \emph{system of extensions} over $\TT$
  is a family $(\EE_i)_{i \in I}$,
  indexed by some partially ordered set $I$,
  where each $\EE_i$ consists of
  a set of sorts,
  sets of relation and function symbols
  --- whose signatures may contain sorts
  from $\TT$ and all $\EE_j$ with $j \leq i$,
  treated disjointly ---
  and a set of (geometric) axioms
  --- which may again use all sorts and symbols
  from $\TT$ and $\EE_j$ with $j \leq i$,
  treated disjointly.
  Taking the disjoint union of the sorts, symbols and axioms
  of all $\EE_i$
  thus yields an extension of $\TT$,
  which we denote $\sum_{i \in I} \EE_i$.
\end{definition}

For a subset $J \subseteq I$
of the partially ordered set $I$,
we use the notation
\[ \downset{J} \defeq \{ i \in I \mid \exists j \in J : i \leq j\}
   \qquad\text{and}\qquad
   \downdownset{J} \defeq \downset{J} \setminus J
   \rlap{.} \]
Given an initial segment $J \subseteq I$
(that is, $\downset{J} = J$),
we can of course restrict a system $(\EE_i)_{i \in I}$ to $J$,
obtaining a system $(\EE_i)_{i \in J}$ over the same base theory $\TT$.
In particular, we see that for every $i \in I$,
$\EE_i$ is an extension of $\TT + \sum_{j < i} \EE_j$.
More generally,
the index set $I$ can be restricted
to any \emph{inward closed} subset $J \subseteq I$
(if $i \leq j \leq k$ and $i, k \in J$ then $j \in J$),
resulting in a system $(\EE_i)_{i \in J}$
over $\TT + \sum_{i \in \downdownset{J}} \EE_i$ instead of $\TT$.

We can also push a system $(\EE_i)_{i \in I}$ over $\TT$ forward
along an order-preserving map $f : I \to J$
to another partially ordered set $J$
by setting $\tilde{\EE}_j \defeq \sum_{i \in f^{-1}(j)} \EE_i$.
Note that $f^{-1}(j)$ is inward closed
and $\tilde{\EE}_j$ is an extension of
$\TT + \sum_{i \in \downdownset{f^{-1}(j)}} \EE_i$
and therefore also of
$\TT + \sum_{j' < j} \tilde{\EE}_{j'}$.

\begin{definition}
  Let $(\EE_i)_{i \in I}$ be a system of extensions
  over a theory $\TT$.
  A \emph{system of model extensions} $(E_i)_{i \in I}$
  for $(\EE_i)_{i \in I}$
  in a Grothendieck topos $\E$
  over some $M \in \mod{\TT}{\E}$
  is just a model extension $\sum_{i \in I} E_i$
  of $M$ along $\sum_{i \in I} \EE_i$,
  regarded as a family of model extensions $E_i$
  of $M + \sum_{j < i} E_j$ along $\EE_i$.

  For a family of subtoposes $(\E_i)_{i \in I}$ of $\E$
  with $\E_i \subseteq \E_j$ for $j < i$,
  a system of model extensions \emph{in the $\E_i$}
  is similarly a family $(E_i)_{i \in I}$,
  where each $E_i$ is an extension
  of $M|_{\E_i} + \sum_{j < i} E_j|_{\E_i}$
  along $\EE_i$.
  If $M|_{\E_i} + \sum_{j \leq i} E_j|_{\E_i}$
  is a universal model of $\TT + \sum_{j \leq i} \EE_j$
  for every $i \in I$,
  then we call $(\EE_i, E_i)_{i \in I}$
  a \emph{system of presentations}
  of the $\E_i$ over ($\TT$ and) $M$.
\end{definition}

The following lemma
is our general formulation of how equivalence extensions
can be built up from smaller equivalence extensions.
We will also use it
(in Corollary \ref{corollary-system-of-presentations-over-presentation}
below)
to clarify the role of systems of presentations
which all present the same topos.

\begin{lemma}
  \label{lemma-system-of-equivalence-extensions}
  Let $(\EE_i)_{i \in I}$ be a system over $\TT$
  where the partial order $I$ is well-founded.
  Then the assertions
  \begin{enumerate}[label=(\roman*)]
    \item
      For every $i \in I$,
      $\EE_i$ is an equivalence extension
      of $\TT + \sum_{j < i} \EE_j$.
    \item
      For every $i \in I$,
      $\sum_{j \leq i} \EE_j$ is an equivalence extension of $\TT$.
  \end{enumerate}
  are equivalent and they imply:
  \begin{enumerate}[label=(\roman*)]
    \setcounter{enumi}{2}
    \item
      $\sum_{i \in I} \EE_i$ is an equivalence extension of $\TT$.
  \end{enumerate}
\end{lemma}

\begin{proof}
  $(i) \Rightarrow (iii)$.
  We show this implication first,
  because it will be used for the others.
  To see that the functor
  \[ U_{\sum_{i \in I} \EE_i} :
  \mod{(\TT + \sum_{i \in I} \EE_i)}{\E} \to \mod{\TT}{\E} \]
  is an equivalence,
  start by taking two models $M, M'$
  of $\TT + \sum_{i \in I} \EE_i$ in $\E$
  and homomorphisms $f, g : M \to M'$
  with $U_{\sum_{i \in I} \EE_i}(f) = U_{\sum_{i \in I} \EE_i}(g)$.
  Then we can show $f = g$
  by well-founded induction over $i \in I$
  using the faithfulness of each
  \[ U_{\EE_i} : \mod{(\TT + \sum_{j \leq i} \EE_j)}{\E} \to
  \mod{(\TT + \sum_{j < i} \EE_j)}{\E} \text{.}\]
  Given instead only
  $f_0 : U_{\sum_{i \in I} \EE_i}(M) \to
  U_{\sum_{i \in I} \EE_i}(M')$,
  we use the fullness of the $U_{\EE_i}$
  to construct $f : M \to M'$
  with $U_{\sum_{i \in I} \EE_i}(f) = f_0$
  by well-founded recursion.
  Finally,
  let $M_0 \in \mod{\TT}{\E}$
  and notice that the $U_{\EE_i}$ are not only essentially surjective
  but strictly surjective on objects (since they are isofibrations).
  This means we can again use well-founded recursion
  to construct a model extension $\sum_{i \in I} E_i$ of $M_0$
  and therefore a (strict) preimage $M_0 + \sum_{i \in I} E_i$
  of $M_0$ under $U_{\sum_{i \in I} \EE_i}$.

  $(i) \Rightarrow (ii)$.
  Fix $i \in I$.
  Then this follows immediately
  by applying $(i) \Rightarrow (iii)$
  to the restricted system $(\EE_j)_{j \leq i}$.

  $(ii) \Rightarrow (i)$.
  By well-founded induction,
  let $i \in I$
  with $\EE_j$ an equivalence extension for all $j < i$.
  Then by $(i) \Rightarrow (iii)$
  applied to $(\EE_j)_{j < i}$,
  $\sum_{j < i} \EE_j$ is an equivalence extension.
  But since $\sum_{j \leq i} \EE_j = \sum_{j < i} \EE_j + \EE_i$
  is an equivalence extension by assumption,
  $\EE_i$ must be one too.
\end{proof}

\begin{remark}
  The assumption that $I$ is well-founded
  is necessary in Lemma \ref{lemma-system-of-equivalence-extensions}.
  Indeed, we can take $\TT$ to be the empty theory
  and consider the system of theory extensions
  $(\EE_n)_{n \in \ZZ}$
  where every $\EE_n$ is the quotient extension
  consisting of the contradictory axiom $\top \turnstile{[]} \bot$.
  Then none of the $\TT + \sum_{m \leq n} \EE_m$
  is equivalent to $\TT$,
  and neither is $\TT + \sum_{n \in \ZZ} \EE_n$,
  but every $\EE_n$ is an equivalence extension of
  $\TT + \sum_{m < n} \EE_m$.
  If we insist on treating syntactically equivalent theories as equal,
  we can even say that every $\EE_n$ is the \emph{empty} extension
  of its respective base theory.
  In this sense,
  a system of theory extensions $(\EE_i)_{i \in I}$
  is not fully determined by the individual extensions $\EE_i$,
  for $I$ not well-founded.
  But we would rather argue that
  the notion of knowing an extension
  without knowing its base theory
  is ill-defined.
  (In the language of type theory,
  the type of extensions is inherently a dependent type.)
\end{remark}

\begin{corollary}
  \label{corollary-system-of-presentations-over-presentation}
  Let $M \in \mod{\TT}{\E}$ be universal.
  Then a system of extensions $(\EE_i, E_i)_{i \in I}$
  over $(\TT, M)$ in $\E$
  with $I$ well-founded
  is a system of presentations of $\E$
  if and only if
  all $\EE_i$ are equivalence extensions.
  And in this case,
  $M + \sum_{i \in I} E_i \in \mod{(\TT + \sum_{i \in I} \EE_i)}{\E}$
  is also universal.
\end{corollary}

\begin{proof}
  By Lemma \ref{lemma-equivalence-extension-of-universal-model},
  the universality of the models $M + \sum_{j \leq i} E_j$
  for all $i \in I$
  is just condition $(ii)$ of Lemma
  \ref{lemma-system-of-equivalence-extensions}
  and the universality of $M + \sum_{i \in I} E_i$
  is condition $(iii)$ of Lemma
  \ref{lemma-system-of-equivalence-extensions}.
\end{proof}

\subsection{Systems of conditional extensions}

We now briefly describe how to apply
the construction of conditional extensions
to systems of theory extensions.
Given a system $(\EE_i)_{i \in I}$ over $\TT$
and a family of closed geometric formulas $\phi_i$ of $\TT$,
we can of course blindly apply the rules of
Definition \ref{definition-conditional-extension}
to every $\EE_i$,
and we do get a new system of extensions
\[ (\EE_i/\phi_i)_{i \in I} \]
over $\TT$,
simply because the signature of $\EE_i/\phi_i$
is the same as that of $\EE_i$.
But then we can not say that
the extension $\EE_i/\phi_i$
of $\TT + \sum_{j < i} \EE_j/\phi_j$
is an honest conditional extension,
since there is no sensible way
to regard $\EE_i$ as an extension of
$\TT + \sum_{j < i} \EE_j/\phi_j + \phi_i$.
For example,
fixing some $\TT$, $\EE_1$ and $\phi_1$
and using $\phi_2 = \top$,
the operation of mapping an extension $\EE_2$
of $\TT + \phi_1 + \EE_1$
to $\TT + \EE_1/\phi_1 + \EE_2$
is not well-defined with respect to syntactic equivalence.

To remedy this,
we impose the condition that
\[ \phi_i \turnstile{[]} \phi_j \qquad\text{for $j \leq i$} \]
is provable in $\TT$.
Then we can regard each $\EE_i$ as an extension of
\[ \TT + \sum_{j < i} \EE_j/\phi_j + \phi_i =
   \TT + \sum_{j < i} \EE_j + \phi_i
   \rlap{,} \]
and obtain the system $(\EE_i/\phi_i)_{i \in I}$
of conditional extensions.

In Definition \ref{definition-conditional-extension},
we treated $\EE$ as an extension of $\TT + \phi$,
not of $\TT$,
which is justified in view of
Lemma \ref{lemma-conditional-extension-equivalence-extension}.
This refinement is dropped here
in order to be able to use
our definition of a system of theory extensions
without modification.
In our application
(see Theorem \ref{theorem-main-gluing} below),
the extensions $\EE_i$ are such that
$\phi_i$ is provable from $\TT + \sum_{j \leq i} \EE_i$.

We also want to be able to construct model extensions,
to be denoted as
\[ M + \sum_{i \in I} E_i/\phi_i \rlap{,} \]
for systems of conditional extensions,
as we did in
Proposition \ref{proposition-conditional-model-extensions}
for a single conditional extension.
Let $(E_i)_{i \in I}$ be a family of model extensions
for $(\EE_i)_{i \in I}$,
over some $M \in \TT\dashmod(\E)$,
in the subtoposes
\[ \E_i \defeq \E_{o(\interpretation{\phi_i}_M)} \rlap{.} \]
This makes sense because
our requirement on the $\phi_i$
ensures $\E_i \subseteq \E_j$ for $j \leq i$.
Let $\iota_i : \E_i \to \E$
be the open embeddings,
and also $\iota_i^j : \E_i \to \E_j$
for $j \leq i$.
\[ \begin{tikzcd}
  \E_j \ar[r, hook, "\iota_j"] & \E \\
  \E_i \ar[ur, hook, "\iota_i"'] \ar[u, hook, "\iota_i^j"] &
\end{tikzcd} \]

Then,
like in Proposition \ref{proposition-conditional-model-extensions},
applying $(\iota_i)_!$ to the interpretations
of the sorts of $\EE_i$ in $E_i$
produces objects in $\E / \interpretation{\phi_i}_M$,
as required for
a model extension of $M$ along $\sum_{i \in I} \EE_i/\phi_i$.
The signatures of the relation symbols of $\EE_i$
can contain sorts $S$ from $\EE_i$,
sorts $S'$ from various $\EE_j$ with $j \leq i$
and sorts $S''$ from $\TT$,
so the interpretation of such a symbol is a subobject
of an object like
\[ \interpretation{S} \times
   (\iota_i^j)^* \interpretation{S'} \times
   (\iota_i)^* \interpretation{S''}
   \;\in\; \E_i
   \rlap{.} \]
Applying $(\iota_i)_!$
yields precisely a subobject of
$(\iota_i)_! \interpretation{S} \times
(\iota_j)_! \interpretation{S'} \times
\interpretation{S''} \times \interpretation{\phi_i}$,
since
\[ (\iota_i)_! (\iota_i^j)^* \interpretation{S'} =
   (\iota_i)_! (\iota_i)^* (\iota_j)_! \interpretation{S'} =
   (\iota_j)_! \interpretation{S'} \times \interpretation{\phi_i}
   \rlap{.} \]
Also,
the modified axioms of $\EE_i$ in $\EE_i/\phi_i$
hold for the structure in $\E$ defined this way
because they can be tested after pulling back to $\E_i$,
as before.
In summary,
we have constructed a model extension
$\sum_{i \in I} E_i / \phi_i$ of $M$
(in the topos $\E$),
such that for every $i$,
the extension $E_i / \phi_i$
is obtained from the model extension $E_i$
(in the topos $\E_i$)
as in Proposition \ref{proposition-conditional-model-extensions},
in accordance with our previous use of the notation $E/\phi$.


\subsection{Syntactic presentation of a glued topos}

We need a lemma about
testing the property of a geometric morphism to be an equivalence
on an open cover of the target topos.
Or rather,
we need its reformulation in terms of models of theories
(Corollary \ref{corollary-check-universality-locally} below),
saying that the property of a model to be universal
can be tested on an \enquote{open cover of the theory}.

\begin{lemma}
  \label{lemma-check-equivalence-locally}
  Let $f : \E \to \F$
  be a geometric morphism
  and let a cover $\F = \bigcup_{i \in I} \F_i$
  by open subtoposes $\F_i = \F_{o(U_i)}$ be given
  such that
  the induced geometric morphisms $f_i : \E_i \to \F_i$,
  where $\E_i = \E_{o(f^*U_i)}$,
  are all equivalences.
  \[ \begin{tikzcd}
    \E_i \ar[r, "f_i", "\simeq"'] \ar[d, hook]
      \ar[dr, phantom, "\lrcorner", very near start] &
    \F_i \ar[d, hook] \\
    \E \ar[r, "f"] & \F
  \end{tikzcd} \]
  Then $f$ is an equivalence.
\end{lemma}

\begin{proof}
  The fact that the $\F_i$ cover $\F$ means
  $1_\F = \bigvee_{i \in I} U_i$
  and, equivalently,
  the sheafification functors $a_{U_i} : \F \to \F_i$
  are jointly conservative.
  Setting $V_i \defeq f^*U_i$
  we similarly have $1_\E = \bigvee_{i \in I} V_i$
  since $f^*$ preserves colimits.
  We will show that the unit $\eta_f$ and counit $\varepsilon_f$
  of the adjunction $f^* \dashv f_*$
  are isomorphisms
  by showing that $a_{U_i} \eta_f$ and $a_{V_i} \varepsilon_f$
  are isomorphisms for all $i \in I$.
  More precisely,
  if $\eta_{f_i}$ and $\varepsilon_{f_i}$
  are the unit and counit of ${f_i}^* \dashv {f_i}_*$
  we will show that
  $a_{U_i} \eta_f$ and $\eta_{f_i} a_{U_i}$
  only differ by an isomorphism of their codomains
  $a_{U_i} f_* f^* \cong {f_i}_* {f_i}^* a_{U_i}$,
  and similarly
  $a_{V_i} \varepsilon_f$ and $\varepsilon_{f_i} a_{V_i}$
  only differ by an isomorphism of their domains
  $a_{V_i} f^* f_* \cong {f_i}^* {f_i}_* a_{V_i}$.

  Recall that $a_U : \F \to \F_{o(U)}$ has two fully faithful adjoints
  $j_U \dashv a_U \dashv i_U$,
  where $j_U(a_U(X)) = X \times U$ and $i_U(a_U(X)) = X^U$.
  Thus we have the following functors and adjunctions
  \[ \begin{tikzcd}
        [row sep=large, column sep=large, cells={minimum width=14mm, minimum height=8mm}]
    \E_{o(V)} & \F_{o(U)} \\
    \E
      \ar[from=u, "i_V", hook, shift left=6]
      \ar[u, phantom, "\dashv", shift left=3]
      \ar[u, "a_V"', very near start]
      \ar[u, phantom, "\dashv", shift right=3]
      \ar[from=u, "j_V"', hook, shift right=6]
      \ar[from=r, "f^*"', shift right=2]
      \ar[r, phantom, "\dashv"{rotate=-90}]
      \ar[r, "f_*"', shift right=2]
      &
    \F
      \ar[from=u, "i_U", hook, shift left=6]
      \ar[u, phantom, "\dashv", shift left=3]
      \ar[u, "a_U"', very near start]
      \ar[u, phantom, "\dashv", shift right=3]
      \ar[from=u, "j_U"', hook, shift right=6]
    \rlap{,}
  \end{tikzcd} \]
  where
  $\varepsilon_U : a_U i_U \to \Id_{\F_{o(U)}}$,
  $\tilde{\eta}_U : \Id_{\F_{o(U)}} \to a_U j_U$,
  $\varepsilon_V : a_V i_V \to \Id_{\E_{o(V)}}$ and
  $\tilde{\eta}_V : \Id_{\E_{o(V)}} \to a_V j_V$
  are isomorphisms.
  The geometric morphism $\E_{o(V)} \to \F_{o(U)}$
  is given by $a_V f^* j_U \dashv a_U f_* i_V$.
  Our assumption therefore means that
  the following two natural transformations
  (denoted as string diagrams)
  are isomorphisms.
  \[ \begin{tikzpicture}[string diagram]
    \coordinate (tl) at (-3,  0);
    \coordinate (br) at ( 3, -6);
    \path
      (0,  0) coordinate (1) \dotnode{north:$\tilde{\eta}_U$}
      (0, -2) coordinate (2) \dotnode{north:$\eta_f$}
      (0, -4) coordinate (3) \dotnode{north:$\eta_V$}
      ;
    \path[string]
      ({$(1) + (-5, 0)$} |- br) node[bottom label]{$j_U$}
      to ($(1) + (-5, -5)$)
      arc[start angle=180, end angle=0, radius=5]
      to ({$(1) + (5, 0)$} |- br) node[bottom label]{$a_U$};
    \path[string, dashed, thick]
      ({$(2) + (-3, 0)$} |- br) node[bottom label]{$f^*$}
      to ($(2) + (-3, -3)$)
      arc[start angle=180, end angle=0, radius=3]
      to ({$(2) + (3, 0)$} |- br) node[bottom label]{$f_*$};
    \path[string]
      ({$(3) + (-1, 0)$} |- br) node[bottom label]{$a_V$}
      to ($(3) + (-1, -1)$)
      arc[start angle=180, end angle=0, radius=1]
      to ({$(3) + (1, 0)$} |- br) node[bottom label]{$i_V$};
  \end{tikzpicture}
  \qquad
  \begin{tikzpicture}[string diagram]
    \coordinate (tl) at (-3,  2);
    \coordinate (br) at ( 3, -4);
    \path
      (0,  0) coordinate (1) \dotnode{south:$\tilde{\varepsilon}_U$}
      (0, -2) coordinate (2) \dotnode{south:$\varepsilon_f$}
      (0, -4) coordinate (3) \dotnode{south:$\varepsilon_V$}
      ;
    \path[string]
      ({$(3) + (-5, 0)$} |- tl) node[top label]{$i_V$}
      to ($(3) + (-5, 5)$)
      arc[start angle=-180, end angle=0, radius=5]
      to ({$(3) + (5, 0)$} |- tl) node[top label]{$a_V$};
    \path[string, dashed, thick]
      ({$(2) + (-3, 0)$} |- tl) node[top label]{$f_*$}
      to ($(2) + (-3, 3)$)
      arc[start angle=-180, end angle=0, radius=3]
      to ({$(2) + (3, 0)$} |- tl) node[top label]{$f^*$};
    \path[string]
      ({$(1) + (-1, 0)$} |- tl) node[top label]{$a_U$}
      to ($(1) + (-1, 1)$)
      arc[start angle=-180, end angle=0, radius=1]
      to ({$(1) + (1, 0)$} |- tl) node[top label]{$j_U$};
  \end{tikzpicture} \]

  The only additional ingredient we need
  is a certain compatibility of~$f^*$ and~$f_*$
  with the adjunctions $j_U \dashv a_U$ and $j_V \dashv a_V$.
  Since $f^*(X \times U) = f^*X \times V$
  and $(\tilde{\varepsilon}_U)_X : X \times U \to X$
  and $(\tilde{\varepsilon}_V)_Y : Y \times V \to Y$
  are the projections on the first factor,
  we have a natural isomorphism
  \[ \begin{tikzpicture}[string diagram]
    \path
      (0, 0) \dotnode{west:$\alpha$};
    \path[string]
      (-2, 2) node[top label, xshift=-1mm]{$a_U$}
      to[out=south, in=north west] (0, 0)
      to[out=south east, in=north] (1, -2) node[bottom label, xshift=-1mm]{$a_V$};
    \path[string]
      (-1, 2) node[top label, xshift=+1mm]{$j_U$}
      to[out=south, in=north west] (0, 0)
      to[out=south east, in=north] (2, -2) node[bottom label, xshift=+1mm]{$j_V$};
    \path[string, dashed, thick]
      (1, 2) node[top label]{$f^*$}
      to[out=south, in=north east] (0, 0)
      to[out=south west, in=north] (-1, -2) node[bottom label]{$f^*$};
  \end{tikzpicture}
  \quad\text{with}\quad
  \begin{tikzpicture}[string diagram]
    \coordinate (tl) at (-3, 2);
    \coordinate (br) at (4, -4);
    \path
      ( 0,  0) coordinate (1) \dotnode{west:$\alpha$}
      ( 2, -3) coordinate (2) \dotnode{south:$\tilde{\varepsilon}_V$};
    \path[string]
      ($(1) + (-2, 2)$) to[out=south, in=north west] (1)
      to[out=south east, in=north] ($(2) + (-1, 1)$)
      arc[start angle=-180, end angle=0, radius=1]
      to[out=north, in=south east] (1)
      to[out=north west, in=south] ($(1) + (-1, 2)$);
    \path[string, dashed, thick]
      ($(1) + (1, 2)$) to[out=south, in=north east] (1)
      to[out=south west, in=north] ($(1) + (-1, -2)$)
      to ({$(1) + (-1, 0)$} |- br);
  \end{tikzpicture}
  =
  \begin{tikzpicture}[string diagram]
    \path
      (0, 0) \dotnode{south:$\tilde{\varepsilon}_U$};
    \draw[string]
      (-1, 2) to (-1, 1)
      arc[start angle=-180, end angle=0, radius=1]
      to (1, 2);
    \path[string, dashed, thick]
      (2, 2) -- (2, -2) node[midway, right]{$f^*$};
  \end{tikzpicture}
  . \]
  For~$f_*$ the situation is slightly different
  as $f_*(V) \neq U$,
  but we have an arrow $(\eta_f)_U : U \to f_*(V)$
  which becomes an isomorphism after applying $a_U$
  (since $f_*(V) \times U = U$).
  Thus we have an isomorphism
  \[ \begin{tikzpicture}[string diagram]
    \path
      (0, 0) \dotnode{west:$\beta$};
    \path[string]
      (-2, 2) node[top label, xshift=-1mm]{$a_V$}
      to[out=south, in=north west] (0, 0) to[out=south east, in=north]
      (1, -2) node[bottom label, xshift=-1mm]{$a_U$};
    \path[string]
      (-1, 2) node[top label, xshift=+1mm]{$j_V$}
      to[out=south, in=north west] (0, 0) to[out=south east, in=north]
      (2, -2) node[bottom label, xshift=+1mm]{$j_U$};
    \path[string, dashed, thick]
      (1, 2) node[top label]{$f_*$}
      to[out=south, in=north east] (0, 0) to[out=south west, in=north]
      (-1, -2) node[bottom label]{$f_*$};
    \path[string]
      (4, 2) node[top label]{$a_U$}
      to[out=south, in=east] (.25, 0)
      to[out=east, in=north] (4, -2) node[bottom label]{$a_U$};
  \end{tikzpicture}
  \quad\text{with}\quad
  \begin{tikzpicture}[string diagram]
    \coordinate (tl) at (-3, 2);
    \coordinate (br) at (4, -4);
    \path
      ( 0,  0) coordinate (1) \dotnode{west:$\beta$}
      ( 2, -3) coordinate (2) \dotnode{south:$\tilde{\varepsilon}_U$};
    \path[string]
      ($(1) + (-2, 2)$) to[out=south, in=north west] (1)
      to[out=south east, in=north] ($(2) + (-1, 1)$)
      arc[start angle=-180, end angle=0, radius=1]
      to[out=north, in=south east] (1)
      to[out=north west, in=south] ($(1) + (-1, 2)$);
    \path[string, dashed, thick]
      ($(1) + (1, 2)$) to[out=south, in=north east] (1)
      to[out=south west, in=north] ($(1) + (-1, -2)$)
      to ({$(1) + (-1, 0)$} |- br);
    \path[string]
      ($(1) + (4, 2)$)
      to[out=south, in=east] ($(1) + (.25, 0)$)
      to[out=east, in=north] ($(1) + (4, -2)$)
      to ({$(1) + (4, 0)$} |- br);
  \end{tikzpicture}
  =
  \begin{tikzpicture}[string diagram]
    \path
      (0, 0) \dotnode{south:$\tilde{\varepsilon}_U$};
    \path[string]
      (-1, 2) to (-1, 1)
      arc[start angle=-180, end angle=0, radius=1]
      to (1, 2);
    \path[string, dashed, thick]
      (2, 2) -- (2, -2) node[midway, right]{$f_*$};
    \path[string]
      (4, 2) -- (4, -2) node[midway, right]{$a_U$};
  \end{tikzpicture} .\]

  Now we can piece together isomorphisms
  reducing to $a_U \eta_f$ and $a_V \varepsilon_f$ respectively.
  (Start by introducing a new squiggle
  involving $\tilde{\varepsilon}_V$ and $\tilde{\eta}_V$
  below $\tilde{\eta}_V^{-1}$ in the first diagram,
  and a new squiggle involving $\eta_V$ and $\varepsilon_V$
  above $\varepsilon_V^{-1}$ in the second diagram.)
  \[ \begin{tikzpicture}[string diagram]

    \coordinate (tl) at (- 8,   1);
    \coordinate (br) at ( 10, -20);

    \path
      (  0,   0) coordinate  (1) \dotnode{north:$\tilde{\eta}_U$}
      (  0, - 2) coordinate  (2) \dotnode{north:$\eta_f$}
      (  0, - 4) coordinate  (3) \dotnode{north:$\eta_V$}
      (  8, - 7) coordinate  (4) \dotnode{north:$\tilde{\eta}_V$}
      (- 5, - 8) coordinate  (5) \dotnode{ west:$\alpha$}
      (  5, -10) coordinate  (6) \dotnode{ west:$\beta^{-1}$}
      (- 2, -11) coordinate  (7) \dotnode{north:$\tilde{\eta}_V^{-1}$}
      (  2, -13) coordinate  (8) \dotnode{north:$\varepsilon_V$}
      (  5, -16) coordinate  (9) \dotnode{ west:$\beta$}
      (  8, -19) coordinate (10) \dotnode{south:$\tilde{\eta}_U^{-1}$}
      ;

    \path[string]
      (5) to[out=south east, in=north] ($(7) + (-1, 1)$)
      arc[start angle=-180, end angle=0, radius=1]
      to[out=north, in=south] ($(3) + (-1, -1)$)
      arc[start angle=180, end angle=0, radius=1]
      to[out=south, in=north] ($(8) + (-1, 1)$)
      arc[start angle=-180, end angle=0, radius=1]
      to[out=north, in=south west] (6)
      to[out=north east, in=south] ($(1) + (6, -6)$)
      arc[start angle=0, end angle=180, radius=6]
      to[out=south, in=north west] (5);
    \path[string, dashed, thick]
      ({$(5) + (-1, 0)$} |- br) to ($(5) + (-1, -2)$)
      to[out=north, in=south west] (5)
      to[out=north east, in=south] ($(2) + (-4, -4)$)
      arc[start angle=180, end angle=0, radius=4]
      to[out=south, in=north west] (6)
      to[out=south east, in=north] ($(6) + (1, -2)$)
      to[out=south, in=north] ($(9) + (1, 2)$)
      to[out=south, in=north east] (9)
      to[out=south west, in=north] ($(9) + (-1, -2)$)
      to ({$(9) + (-1, 0)$} |- br);
    \path[string]
      (tl -| {$(5) + (-2, 0)$}) to ($(5) + (-2, 3)$)
      to[out=south, in=north west] (5)
      to[out=south east, in=north] ($(5) + (1, -2)$)
      to[out=south, in=north west] (9)
      to[out=south east, in=north] ($(9) + (1, -2)$)
      to ({$(9) + (1, -2)$} |- br);
    \path[string]
      (6) to[out=south west, in=north] ($(6) + (-1, -2)$)
      to[out=south, in=north] ($(9) + (-1, 2)$)
      to[out=south, in=north west] (9)
      to[out=south east, in=north] ($(10) + (-1, 1)$)
      arc[start angle=-180, end angle=0, radius=1]
      to[out=north, in=east] ($(9) + (.25, 0)$)
      to[out=east, in=south] ($(9) + (3, 2)$)
      to[out=north, in=south] ($(6) + (3, -2)$)
      to[out=north, in=east] ($(6) + (.25, 0)$)
      to[out=east, in=south] ($(4) + (1, -1)$)
      arc[start angle=0, end angle=180, radius=1]
      to[out=south, in=north east] (6);

    \draw[dotted] (-8, -5) -- (9, -5);

  \end{tikzpicture}
  =
  \begin{tikzpicture}[string diagram]
    \path
      (0, 0) \dotnode{north:$\eta_f$};
    \path[string, dashed, thick]
      (-1, -3) to (-1, -1)
      arc[start angle=180, end angle=0, radius=1]
      to (1, -3);
    \path[string]
      (2, 1) -- (2, -3) node[midway, right]{$a_U$};
  \end{tikzpicture} .\]

  \[ \begin{tikzpicture}[string diagram]
    \coordinate (tl) at (-15,   1);
    \coordinate (br) at (  3, -26);

    \path
      (  0,   0) coordinate  (1) \dotnode{north:$\tilde{\eta}_V$}
      (- 3, - 3) coordinate  (2) \dotnode{east:$\alpha^{-1}$}
      (- 6, - 6) coordinate  (3) \dotnode{south:$\tilde{\eta}_U$}
      (- 9, - 9) coordinate  (4) \dotnode{west:$\beta^{-1}$}
      (-12, -11) coordinate  (5) \dotnode{south:$\varepsilon_V^{-1}$}
      (- 9, -14) coordinate  (6) \dotnode{west:$\beta$}
      (- 3, -16) coordinate  (7) \dotnode{east:$\alpha$}
      (  0, -19) coordinate  (8) \dotnode{north:$\tilde{\eta}_V^{-1}$}
      (- 7, -21) coordinate  (9) \dotnode{south:$\tilde{\varepsilon}_U$}
      (- 7, -23) coordinate (10) \dotnode{south:$\varepsilon_f$}
      (- 7, -25) coordinate (11) \dotnode{south:$\varepsilon_V$}
      ;

    \path[string, dashed, thick]
      ({$(4) + (-1, 0)$} |- tl) to ($(4) + (-1, 2)$)
      to[out=south, in=north west] (4)
      to[out=south east, in=north] ($(4) + (1, -2)$)
      to[out=south, in=north] ($(6) + (1, 2)$)
      to[out=south , in=north east] (6)
      to[out=south west, in=north] ($(6) + (-1, -2)$)
      to[out=south, in=north] ($(10) + (-3, 3)$)
      arc[start angle=-180, end angle=0, radius=3]
      to[out=north, in=south west] (7)
      to[out=north east, in=south] ($(7) + (1, 2)$)
      to[out=north, in=south] ($(2) + (1, -2)$)
      to[out=north, in=south east] (2)
      to[out=north west, in=south] ($(2) + (-1, 2)$)
      to ({$(2) + (-1, 0)$} |- tl);
    \path[string]
      (2) to[out=south west, in=north] ($(2) + (-1, -2)$)
      to[out=south, in=north] ($(7) + (-1, 2)$)
      to[out=south, in=north west] (7)
      to[out=south east, in=north] ($(8) + (-1, 1)$)
      arc[start angle=-180, end angle=0, radius=1]
      to[out=north, in=south] ($(1) + (1, -1)$)
      arc[start angle=0, end angle=180, radius=1]
      to[out=south, in=north east] (2);
    \path[string]
      ({$(2) + (1, 0)$} |- tl) to ($(2) + (1, 2)$)
      to[out=south, in=north east] (2)
      to[out=south west, in=north] ($(4) + (1, 2)$)
      to[out=south, in=north east] (4)
      to[out=south west, in=north] ($(5) + (-2, -1)$)
      to ({$(5) + (-2, 0)$} |- br);
    \path[string]
      (4) to[out=south west, in=north] ($(4) + (-1, -2)$)
      to[out=south, in=north] ($(6) + (-1, 2)$)
      to[out=south, in=north west] (6)
      to[out=south east, in=north] ($(6) + (3, -3)$)
      to[out=south, in=north] ($(9) + (1, 1)$)
      arc[start angle=0, end angle=-180, radius=1]
      to[out=north, in=south] ($(6) + (1, -2)$)
      to[out=north, in=south east] (6)
      to[out=north west, in=south] ($(5) + (1, -1)$)
      arc[start angle=0, end angle=180, radius=1]
      to[out=south, in=north] ($(11) + (-5, 5)$)
      arc[start angle=-180, end angle=0, radius=5]
      to[out=north, in=south] ($(7) + (1, -2)$)
      to[out=north, in=south east] (7)
      to[out=north west, in=east] ($(6) + (.25, 0)$)
      to[out=east, in=south] ($(6) + (3, 2)$)
      to[out=north, in=south] ($(4) + (3, -2)$)
      to[out=north, in=east] ($(4) + (.25, 0)$)
      to[out=east, in=south] ($(3) + (1, -1)$)
      arc[start angle=0, end angle=180, radius=1]
      to[out=south, in=north east] (4);

    \draw[dotted]
      (-15, -20) -- (1, -20);

  \end{tikzpicture}
  =
  \begin{tikzpicture}[string diagram]
    \path
      (0, 0) \dotnode{south:$\varepsilon_f$};
    \path[string, dashed, thick]
      (-1, 3) to (-1, 1)
      arc[start angle=-180, end angle=0, radius=1]
      to (1, 3);
    \path[string]
      (2, 3) -- (2, -1) node[midway, right]{$a_V$};
  \end{tikzpicture}
  \]
\end{proof}

\begin{remark}
  Lemma \ref{lemma-check-equivalence-locally}
  does not generalize to covers by arbitrary subtoposes,
  if cover just means that the join of these subtoposes is $\F$.
  For example,
  the intersection of any open subtopos $\F_{o(U)}$
  with its closed complement $\F_{c(U)}$
  is empty,
  so the geometric morphism
  \[ \F_{o(U)} \amalg \F_{c(U)} \to \F \]
  from their disjoint union
  (which is given by the product of categories)
  becomes an equivalence when pulled back to
  either $\F_{o(U)}$ or $\F_{c(U)}$.
\end{remark}

\begin{corollary}
  \label{corollary-check-universality-locally}
  Let $M \in \mod{\TT}{\E}$ be a model
  and let $(\phi_i)_{i \in I}$ be closed formulas
  such that $\TT$ proves $\bigvee_{i \in I} \phi_i$.
  If $M|_{\E_i}$ is a universal model of $\TT + \phi_i$
  for every $i \in I$,
  where $\E_i \defeq \E_{o(\interpretation{\phi_i}_M)}$,
  then $M$ was already a universal model of $\TT$.
\end{corollary}

\begin{proof}
  Fixing a classifying topos $\Set[\TT]$
  with universal model $M_\TT \in \mod{\TT}{\Set[\TT]}$,
  we have a geometric morphism $f : \E \to \Set[\TT]$
  with $f^*M_\TT \cong M$.
  Since $\bigvee_{i \in I} \phi_i$ is provable,
  $U_i \defeq \interpretation{\phi_i}_{M_\TT}$
  define an open cover of $\Set[\TT]$
  with $M_\TT|_{\Set[\TT]_{o(U_i)}}$
  a universal model of $\TT + \phi_i$
  and $f^*U_i \cong \interpretation{\phi_i}_M$.
  The induced geometric morphisms
  $f_i : \E_i \to \Set[\TT]_{o(U_i)}$
  are equivalences
  since $f_i^*(M|_{\Set[\TT]_{o(U_i)}}) \cong M|_{\E_i}$
  is also a universal model of $\TT + \phi_i$ by assumption.
  So Lemma \ref{lemma-check-equivalence-locally} tells us
  that $f$ is an equivalence,
  that is,
  $f^*M_\TT \cong M$ is a universal $\TT$-model.
\end{proof}

We can now give our construction of
a syntactic presentation of a topos $\E$
covered by open subtoposes $\E_i$, $i \in I$,
from syntactic presentations of the $\E_i$
and of the intersections $\E_S \defeq \bigcap_{i \in S} \E_i$
for $S \subseteq I$ finite and nonempty.
The notation
\[ \Delta(I) \defeq
   \{\, S \subseteq I \mid 1 \leq \abs{S} < \infty \,\} \]
(thinking of simplices with vertices drawn from $I$)
for the partial order
indexing the required system of presentations
will be useful.
Note that we can write $S' \in \Delta(S)$
instead of $S' \leq S$ or $S' \subseteq S$
in this partial order.

In order to give an elegant description of the resulting theory,
the open subtoposes $\E_i$
will be assumed
to be given by the interpretations of formulas $\phi_i$
of some base theory $\TT_0$
in a (not necessarily universal)
base model $M_0 \in \TT_0\dashmod(\E)$.
This requirement can always be met,
starting from an arbitrary base theory (e.g. the empty one),
by adding proposition symbols $p_i$
and the axiom $\bigvee_{i \in I} p_i$
to it,
and then,
since the $p_i$ necessarily also show up
in the presentation of each $\E_j$,
add axioms to these presentations
defining the $p_i$ by already available formulas
for the open subtoposes $\E_i \cap \E_j \subseteq \E_j$.

\begin{theorem}
  \label{theorem-main-gluing}
  Let $\E$ be a topos,
  $\TT_0$ a geometric theory
  with closed formulas $\phi_i$
  such that $\bigvee_{i \in I} \phi_i$ is provable,
  and $M_0 \in \mod{\TT_0}{\E}$ a model.
  For $S \in \Delta(I)$,
  set $\phi_S \defeq \bigwedge_{i \in S} \phi_i$
  and $\E_S \defeq \E_{o(\interpretation{\phi_S}_{M_0})}$.
  Let a system of presentations $(\EE_S, E_S)_{S \in \Delta(I)}$
  of the $\E_S$ over $M_0$ be given.
  Then $\E$ classifies the theory
  \[ \TT \defeq \TT_0 + \sum_{S \in \Delta(I)} \EE_S / \phi_S \]
  with universal model
  \[ M \defeq M_0 + \sum_{S \in \Delta(I)} E_S / \phi_S . \]
\end{theorem}

\begin{proof}
  By Corollary \ref{corollary-check-universality-locally},
  we only have to show that $M|_{\E_{\{i\}}}$
  is a universal model of $\TT + \phi_i$
  for every $i \in I$.
  So fix $i$.
  We know that $M_0|_{\E_{\{i\}}} + E_{\{i\}}$,
  which lives in the same topos as $M|_{\E_{\{i\}}}$,
  is a universal model of $\TT_0 + \EE_{\{i\}}$
  (in particular, $\TT_0 + \EE_{\{i\}}$ proves $\phi_i$),
  and fortunately,
  $\TT + \phi_i$ can be written as 
  \[ \TT + \phi_i = \TT_0 + \EE_{\{i\}} +
     \sum_{S \in \Delta_i(I)} (\EE_{S \setminus \{i\}} + \EE_S) / \phi_S
     \rlap{,} \]
  where $\Delta_i(I) \defeq
  \{\, S \cup \{i\} \mid S \in \Delta(I \setminus \{i\}) \,\}
  = \{\, S \in \Delta(I) \mid \{i\} \subsetneq S \,\}$.
  Also,
  the $(\TT_0 + \EE_{\{i\}})$-part of $M|_{\E_{\{i\}}}$
  is $M_0|_{\E_{\{i\}}} + E_{\{i\}}$.
  Thus,
  by Lemma \ref{lemma-equivalence-extension-of-universal-model},
  we are reduced to showing that
  the remaining extension
  $\sum_{S \in \Delta_i(I)} (\EE_{S \setminus \{i\}} + \EE_S) / \phi_S$
  is an equivalence extension,
  which, by Lemma \ref{lemma-system-of-equivalence-extensions}
  and Lemma \ref{lemma-conditional-extension-equivalence-extension},
  we may do by showing that
  for every $S \in \Delta_i(I)$,
  the extension
  \[ \EE^i_S \defeq \EE_{S \setminus \{i\}} + \EE_S \]
  is an equivalence extension of
  \[ \TT_0 + \EE_{\{i\}} + \phi_S +
     \sum_{S' \in \Delta_i(S), S' \neq S} \EE^i_{S'}
     \rlap{.} \]

  For this,
  observe that $M_0|_{\E_S} + E_{\{i\}}|_{\E_S}$
  is a universal model in $\E_S$
  of $\TT_0 + \EE_{\{i\}} + \phi_S$,
  which we now treat as a base theory.
  And if we set
  \[ E^i_{S'} \defeq E_{S' \setminus \{i\}}|_{\E_S} + E_{S'}|_{\E_S}
     \rlap{,} \]
  for every $S' \in \Delta(S)$,
  this in fact constitutes a system of presentations of the topos $\E_S$,
  since by assumption,
  every
  \[ M_0|_{\E_{S'}} + \sum_{S'' \in \Delta(S')} E_{S''}|_{\E_{S'}} \]
  is a universal model,
  which can be restricted to a universal model in $\E_S$
  while adding $\phi_S$ to the theory.
  Then Corollary \ref{corollary-system-of-presentations-over-presentation}
  allows us to conclude that
  in particular
  the topmost theory extension $\EE^i_S$ of this system
  is an equivalence extension,
  as needed.
\end{proof}

Let us also show that
a system of presentations as required in
Theorem \ref{theorem-main-gluing}
always exists,
starting from presentations of the $\E_i$
(which exist simply by
Theorem \ref{theorem-extension-from-geometric-morphism}).
In particular,
it becomes clear that presentations are only needed
for two- and threefold intersections of the open subtoposes $\E_i$,
not for arbitrary finite intersections.

\begin{proposition}
  \label{proposition-construct-system-of-presentations}
  Let $M_0 \in \mod{\TT_0}{\E}$,
  let $\phi_i$, $i \in I$ be closed formulas of $\TT_0$
  and let $\phi_S$, $\E_S$ for $S \in \Delta(I)$ be as above.
  Then any family of presentations $(\EE_{\{i\}}, E_{\{i\}})$
  of the $\E_{\{i\}}$
  over $(\TT_0, M_0)$
  can be extended layer-wise
  to a system of presentations $(\EE_S, E_S)_{S \in \Delta(I)}$
  of the $\E_S$,
  where $\EE_S$ can be chosen localic for $\abs{S} = 2$,
  a quotient for $\abs{S} = 3$
  and empty for $\abs{S} \geq 4$.
  If the $\E_{\{i\}}$ are localic,
  then $\EE_S$ can even be chosen a quotient for $\abs{S} = 2$
  and empty for $\abs{S} \geq 3$.
\end{proposition}

\begin{proof}
  Let $S \in \Delta(I)$ with $\abs{S} \geq 2$
  and let $\EE_{S'}$, $E_{S'}$ be already defined
  for $S' \in \Delta(S), S' \neq S$.
  Fixing any $i \in S$
  and setting $\widetilde{S} \defeq S \setminus \{i\}$,
  we regroup these extensions as
  \[ \EE^i_{S'} \defeq \EE_{S'} + \EE_{S' \cup \{i\}} \]
  and include $\phi_S$ from the beginning
  to obtain a system of presentations
  $(\EE^i_{S'}, E^i_{S'}|_{\E_S})$,
  $S' \in \Delta(\widetilde{S}) \setminus \{\widetilde{S}\}$
  over $\TT_0 + \EE_{\{i\}} + \phi_S$
  and $M_0|_{\E_S} + E_{\{i\}}|_{\E_S}$,
  all presenting $\E_S$.
  By Corollary
  \ref{corollary-system-of-presentations-over-presentation},
  the sum of this system,
  which is
  \[ \TT_0 + \phi_S +
     \sum_{S' \in \Delta(S) \setminus \{\widetilde{S}, S\}} \EE_{S'}
     \rlap{,} \]
  also presents $\E_S$.
  We can therefore apply Lemma \ref{lemma-revert-an-extension}
  to revert the left-over extension $\EE_{\widetilde{S}}$
  (with accompanying model extension $E_{\widetilde{S}}$),
  obtaining $\EE_S$ and $E_S$
  such that
  \[ M_0|_{\E_S} + \sum_{S' \in \Delta(S)} E_{S'}|_{\E_S}
  \;\in\; \Bigl( \TT_0 + \sum_{S' \in \Delta(S)} \EE_{S'} \Bigr)\dashmod(\E_S) \]
  is universal.
  (We can drop $\phi_S$,
  since we definitely include all $\EE_{\{j\}}$, $j \in S$ now.)
  And indeed
  $\EE_S$ is localic,
  and even a quotient (respectively empty)
  if $\EE_{\widetilde{S}}$ is localic (respectively a quotient).
\end{proof}

\subsection{The localic case}

Proposition \ref{proposition-construct-system-of-presentations}
suggests that a more concrete formulation of
Theorem \ref{theorem-main-gluing}
might be feasible
if the open subtoposes $\E_i$
can be presented by localic extensions
of an appropriate common base theory.
This is the case in our applications
to algebro-geometric toposes.
We don't assume the subterminal objects $U_i$
to be expressible by formulas of the base theory here.
The theory extension adding a proposition symbol $p$
will be denoted $\angles{p}$.

\begin{corollary}
  \label{corollary-localic-gluing}
  Let $\E$ be a topos
  covered by open subtoposes $\E_i = \E_{o(U_i)}$,
  let $M_0 \in \TT_0\dashmod(\E)$
  and let $E_i \in \EE_i\dashmod(M_0|_{\E_i})$
  be model extensions along localic extensions $\EE_i$ of $\TT_0$
  such that
  \[ M_0|_{\E_i} + E_i \in (\TT_0 + \EE_i)\dashmod(\E_i) \]
  is universal for every $i \in I$.
  Then:
  \begin{enumerate}[label=(\roman*)]
    \item
      For every $i$,
      there are closed formulas $\phi_{i, j}$ of $\TT_0 + \EE_i$
      presenting the open subtoposes $\E_i \cap \E_j \subseteq \E_i$,
      and for every $i \neq j$,
      there is a diagonal quotient extension
      $\QQ_{\{i, j\}}$
      for $\EE_i + \phi_{i, j}$ and $\EE_j + \phi_{j, i}$
      over $\TT_0$
      (see Definition \ref{definition-diagonal-extension})
      consisting of axioms fulfilled by $M_0|_{\E_i \cap \E_j} +
      E_1|_{\E_i \cap \E_j} + E_2|_{\E_i \cap \E_j}$.

    \item
      For any such $\phi_{i, j}$ and $\QQ_{\{i, j\}}$,
      the topos $\E$ classifies the theory
      \begin{align*}
        \TT \;\defeq\; \TT_0
        &+ \angles{p_i}_{i \in I}
        + (\bigvee_{i \in I} p_i)
        + (\EE_i/p_i)_{i \in I}
        \\
        &+ \Bigl( (p_j \doubleturnstile{[]} \phi_{i, j})/p_i
           \Bigr)_{i \neq j \in I}
        + \Bigl( \QQ_{\{i, j\}} / (p_i \land p_j)
          \Bigr)_{i \neq j \in I}
      \end{align*}
      with universal model
      \[ M \defeq
         M_0 + (\interpretation{p_i} \defeq U_i)_{i \in I}
         + (E_i/p_i)_{i \in I}
         \rlap{.} \]
  \end{enumerate}
\end{corollary}

Note that finding appropriate quotient extensions $\QQ_{\{i, j\}}$
just means stating enough properties of
$M_0|_{\E_i \cap \E_j} + E_1|_{\E_i \cap \E_j} + E_2|_{\E_i \cap \E_j}$
as axioms
to make the relation symbols of $\EE_2$
definable by formulas of $\TT_0 + \EE_1$
and vice versa,
see Remark \ref{remark-recognizing-extensions-by-definitions}.

\begin{proof}
  The existence of the $\phi_{i, j}$ is clear.
  For $\QQ_{\{i, j\}}$,
  observe that
  \[ \TT_0 + \EE_i + \phi_{i, j}
     \qquad\text{and}\qquad
     \TT_0 + \EE_j + \phi_{j, i} \]
  both present the topos $\E_i \cap \E_j$,
  with universal models agreeing in the $\TT_0$-part.
  Thus, we obtain $\QQ_{\{i, j\}}$
  from Corollary \ref{corollary-diagonal-extension}.

  Now set
  \[ \widetilde{\TT}_0 \defeq \TT_0 + \angles{p_i}_{i \in I}
     + (\bigvee_{i \in I} p_i),
     \quad
     \widetilde{\EE}_{\{i\}} \defeq \EE_i + p_i
     + (p_j \doubleturnstile{[]} \phi_{i, j})_{j \neq i},
     \quad
     \widetilde{\EE}_{\{i, j\}} \defeq \QQ_{\{i, j\}}.
     \]
  Then $\widetilde{M}_0 \defeq M_0 +
  (\interpretation{p_i} \defeq U_i)_{i \in I}$
  is a model of $\widetilde{\TT}_0$ in $\E$,
  and $\widetilde{M}_0|_{\E_i} + E_i$
  is a universal model of
  $\widetilde{\TT}_0 + \widetilde{\EE}_{\{i\}}$,
  since this only differs from $\TT_0 + \EE_i$
  by an extension by definitions,
  and similarly for $\widetilde{\EE}_{\{i, j\}}$.
  By Proposition \ref{proposition-construct-system-of-presentations},
  this system of presentations can be extended
  by setting $\widetilde{\EE}_S = \varnothing$
  for all $S \in \Delta(I)$ with $\abs{S} \geq 3$.
  Then we can apply
  Theorem \ref{theorem-main-gluing}
  and obtain the theory $\TT$ with universal model $M$ as stated.
  (We only split up $\widetilde{\EE}_{\{i\}} / p_i$
  and dropped the trivial extension $p_i/p_i$.)
\end{proof}

\subsection{Application to Zariski toposes}
\label{subsection-gluing-Zar}

We now apply our results
to deduce a syntactic presentation
for the big Zariski topos of an arbitrary scheme,
which can be viewed as glued from big Zariski toposes
of affine schemes.
We first recall the definition,
and the classifying property in the affine case.

The \emph{big Zariski site} of a scheme $S$
is the category $\Sch/S$ of schemes over $S$,
equipped with the Zariski topology $J_\Zar$
in which a sieve on an object $T \to S$
is covering if and only if
there is an open cover $T = \bigcup_i T_i$
of the scheme $T$
such that the sieve contains the open embeddings
\[ \begin{tikzcd}[column sep=tiny]
  T_i \ar[rr, hook] \ar[dr] & & T \ar[dl] \\
  & S \rlap{.} &
\end{tikzcd} \]
The site defined in this way
does not admit any (small) dense set of objects,
so one has to restrict the class of objects in an appropriate way
before the sheaf topos can be formed.
(The resulting topos is still called a \emph{big} Zariski topos,
unless the site only contains open subschemes of $S$.)
We choose the subcategory $(\Sch/S)_\lofp$
of schemes locally of finite presentation over $S$,
yielding the \emph{big Zariski topos}
\[ S_\Zarfp \defeq \Sh((\Sch/S)_\lofp, J_\Zar) \rlap{.} \]

The motivation for this particular class of objects is
that in the affine case $S = \Spec K$
(where $K$ is any commutative unitary ring),
the big Zariski topos $(\Spec K)_\Zarfp$
is the classifying topos of the theory of local $K$-algebras.

\begin{definition}
  \label{definition-theories-for-Zar}
  \leavevmode
  \begin{enumerate}[label=(\roman*)]
    \item
      The (algebraic) theory of rings will be denoted $\Ring$.
      It has one sort $A$,
      function symbols $0, 1 : A$, ${-} : A \to A$,
      ${+}, {\cdot} : A \times A \to A$
      and the usual axioms of a commutative unitary ring.

    \item
      In the theory $\Ring$ (or any extension of it),
      we use the abbreviation
      \[ \inv(x) \quad\defeq\quad
         \ex{\overline{x}}{A} (x \overline{x} = 1)
         \rlap{.} \]

    \item
      We can require a ring to be local
      by the quotient extension
      \[ (\loc) \defeq \quad
         \Bigl\{\;
         \ex{\vec{y}}{A^n}
         (\sum_{i = 1}^n x_i y_i = 1)
         \turnstile{\vec{x} : A^n}
         \bigvee_{i \in \{1, \dots, n\}} \inv(x_i)
         \quad\mid\quad
         n \in \NN
         \;\Bigr\}
         \rlap{.} \]
      (An alternative, syntactically equivalent axiomatization
      is given by the two sequents
      $0 = 1 \turnstile{[]} \bot$
      and
      $\top \turnstile{x : A} \inv(x) \lor \inv(1 - x)$.)

    \item
      For a ring $K$,
      the extension $\AlgStr{K}$ of the theory $\Ring$
      consists of
      constant symbols $c_\lambda : A$ for all $\lambda \in K$
      and the axioms (for all $\lambda, \mu \in K$ as needed)
      \[ c_0 = 0, \quad
         c_\lambda + c_\mu = c_{\lambda + \mu}, \quad
         c_1 = 1, \quad
         c_\lambda c_\mu = c_{\lambda \mu}
         \rlap{.} \]

    \item
      The theory of $K$-algebras is
      \[ \Alg{K} \defeq \Ring + \AlgStr{K} \rlap{.} \]
  \end{enumerate}
\end{definition}

A quick sketch of a proof
that $(\Spec K)_\Zarfp$
classifies $\Alg{K} + (\loc)$,
using notions from Section \ref{section-presheaf-type},
goes as follows.
Since the theory $\Alg{K}$ is algebraic,
a presheaf site for its classifying topos
is given by the opposite of the category of
finitely presented $K$-algebras.
But this is equivalent to the full subcategory
on the affine objects of $(\Sch/\Spec K)_\lofp$,
which is a dense subcategory with respect to $J_\Zar$,
and therefore,
equipped with the topology induced by $J_\Zar$,
an alternative site for the same topos
by the Comparison Lemma.
Finally,
the subtopos defined by
this Zariski topology on the finitely presented $K$-algebras
is the subtopos presented by the quotient extension $(\loc)$.

The universal model of $\Alg{K} + (\loc)$ in $(\Spec K)_\Zarfp$
is the structure sheaf
\[ \mathcal{O} = \mathcal{O}_{S_\Zarfp} \;:\;
   {(\Sch/S)_\lofp}^\op \to \Set, \quad
   T \mapsto \mathcal{O}_T(T) \rlap{.} \]

\begin{remark}
  \label{remark-economical-formulation-of-Alg}
  The theory extension $\AlgStr{K}$
  (and therefore the theory $\Alg{K}$)
  can be formulated more economically,
  using a presentation of the ring $K$ by generators and relations.
  Namely,
  if $K = \ZZ[X_i]/(r_j)$,
  where $i$ and $j$ run over any two index sets,
  it suffices to add constant symbols $c_i : A$
  and axioms $r_j(c_i) = 0$
  to the theory of rings,
  where $r_j(c_i)$ is to be interpreted as a closed term,
  using the available ring structure.
  The theory $\Alg{K}$ from
  Definition \ref{definition-theories-for-Zar}
  can then be obtained by an extension by definitions.
  So for example,
  the topos $(\Spec \ZZ[X])_\Zarfp$
  classifies the theory of local rings
  with one distinguished element.
\end{remark}

The big Zariski topos is functorial in the scheme $S$,
in that a morphism $S' \to S$ of schemes
induces a geometric morphism
$S'_\Zarfp \to S_\Zarfp$,
see \citestacks{0210}.
But we only need the special case where
$S'$ is an open subscheme of $S$ here.
In this case,
observe that $S'$ can be regarded as
a subterminal object of the site $(\Sch/S)_\lofp$
and,
since the Zariski topology $J_\Zar$ is subcanonical,
determines a subterminal sheaf
\[ U_{S'} \defeq \Hom_S(\hole, S')
   \;:\; {(\Sch/S)_\lofp}^\op \to \Set
   \rlap{.} \]

\begin{lemma}
  \label{lemma-for-gluing-Zar}
  \leavevmode
  \begin{enumerate}[label=(\roman*)]
    \item
      The mapping
      \[ S' \mapsto (S_\Zarfp)_{o(U_{S'})} \]
      from open subschemes of $S$
      to open subtoposes of $S_\Zarfp$
      is monotone and
      preserves finite intersections and arbitrary unions.

    \item
      For any open subscheme $S'$ of $S$,
      the open subtopos $(S_\Zarfp)_{o(U_{S'})}$
      is equivalent to ${S'}_\Zarfp$.
      The inverse image $\iota^*$
      and the further left adjoint $\iota_!$
      of the open embedding
      $\iota : {S'}_\Zarfp \hookrightarrow S_\Zarfp$
      are given by
      \[ (\iota^*F)(T') = F(T') \qquad (T' \in (\Sch/S')_\lofp) \]
      and
      \[ (\iota_!F')(T) =
         \begin{cases}
           F'(T), & \text{if $T \to S$ factors through $S'$} \\
           \varnothing, & \text{otherwise}
         \end{cases}
         \qquad (T \in (\Sch/S)_\lofp) \rlap{.} \]

    \item
      If $S = \Spec K$ is affine
      (and therefore $S_\Zarfp$ classifies $\Alg{K} + (\loc)$)
      and $f \in K$,
      then the open subtopos $(S_\Zarfp)_{o(U_{D(f)})}$
      corresponding to the standard open $D(f) \subseteq S$
      is presented by the closed geometric formula
      \[ \inv(c_f) \quad\defeq\quad \ex{x}{A} (x c_f = 1) \rlap{.} \]
  \end{enumerate}
\end{lemma}

\begin{proof}
  \begin{enumerate}[label=(\roman*)]
    \item
      We can identify open subtoposes with subterminal sheaves
      $F \subseteq 1_{S_\Zarfp}$,
      which can in turn be thought of
      as certain classes of objects of $(\Sch/S)_\lofp$,
      where $(S_\Zarfp)_{o(U_{S'})}$
      corresponds to the class of all
      scheme morphisms $T \to S$
      (locally of finite presentation)
      that factor through $S'$.
      Monotonicity is then clear,
      as well as $U_S = 1_{S_\Zarfp}$.
      For binary intersections,
      we have to check that $T \to S$
      factors through $S'_1 \cap S'_2$
      if and only if it factors through $S'_1$ and $S'_2$.
      But this is also clear.
      For unions,
      let $(S'_i)_{i \in I}$ be some family of open subschemes
      and let $F \subseteq 1_{S_\Zarfp}$
      be given with $U_{S'_i} \leq F$ for all $i$,
      that is, $\abs{F(S'_i)} = 1$.
      Because $S' \defeq \bigcup_{i \in I} S'_i$
      is covered by the $S'_i$ with respect to $J_\Zar$,
      this implies $\abs{F(S')} = 1$,
      so $U_{S'} \leq F$, as needed.

    \item
      The underlying category of
      the open subtopos $(S_\Zarfp)_{o(U_{S'})}$
      is equivalent to
      \[ \Sh((\Sch/S)_\lofp, J_\Zar) / \Hom_S(\hole, S') \rlap{,} \]
      the category of all those sheaves $F$ on $(\Sch/S)_\lofp$
      for which $F(T) \neq \varnothing$ implies that
      $T \to S$ factors through $S'$.
      The category of presheaves on $(\Sch/S)_\lofp$ with this property
      is clearly equivalent to $\PSh((\Sch/S')_\lofp)$,
      and one can check that the sheaf conditions
      with respect to the two Zariski topologies
      are then equivalent.
      Then we have
      \[ (\iota^*F)(T') = (F \times U_{S'})(T') = F(T') \]
      for $T' \to S'$,
      and the further left adjoint $\iota_!$
      is the forgetful functor
      $S_\Zarfp / U_{S'} \to S_\Zarfp$,
      as stated.

    \item
      We know that $S_\Zarfp$
      classifies $\Alg{K} + (\loc)$
      with universal model
      $\mathcal{O} = \mathcal{O}_{S_\Zarfp}$,
      and $(D(f))_\Zarfp = (\Spec K_f)_\Zarfp$
      classifies $\Alg{K_f} + (\loc)$,
      which is equivalent to $\Alg{K} + (\loc) + (\inv(c_f))$,
      with universal model
      $\mathcal{O}' = \mathcal{O}_{(D(f))_\Zarfp}$.
      So to show that the embedding $\iota$
      is presented by the quotient extension $(\loc)$,
      all we have to check is that
      $\iota^* \mathcal{O} \cong \mathcal{O}'$
      as models of $\Alg{K}$.
      And this follows from the description of $\iota^*$ in (ii).
  \end{enumerate}
\end{proof}

\begin{remark}
  Not all open subtoposes of $S_\Zarfp$
  correspond to open subschemes of $S$ in this way.
  For example,
  $\Hom_S(\hole, S')$
  is also a subterminal object
  if $S'$ is a \emph{closed} subscheme of $S$.
\end{remark}

Before we formulate the theory classified by $S_\Zarfp$
in generality,
we discuss the example of the projective line.

\begin{proposition}
  \label{proposition-gluing-projective-line}
  The big Zariski topos $(\PP^1_K)_\Zarfp$
  of the projective line over a ring $K$
  classifies the theory $\TT_{\PP^1_K}$,
  which is $\Alg{K} + (\loc)$
  expanded by two relation symbols
  $\widetilde{c}_1, \widetilde{c}_2 \subseteq A$
  and certain axioms as follows.
  \begin{align*}
    \TT_{\PP^1_K} \;\defeq\;
    {}&\Alg{K} + (\loc)
    + \angles{\widetilde{c}_1, \widetilde{c}_2 \subseteq A}
    \\
    &+ \Bigl(
         x \in \widetilde{c}_i \land x' \in \widetilde{c}_i
         \turnstile{x, x'} x = x'
       \Bigr)_{i = 1, 2}
    \\
    &+ \Bigl(
         \top \turnstile{[]}
         (\exx{x} x \in \widetilde{c}_1) \lor
         (\exx{x} x \in \widetilde{c}_2)
       \Bigr)
    \\
    &+ \Bigl(
         x_1 \in \widetilde{c}_1 \land x_2 \in \widetilde{c}_2
         \turnstile{x_1, x_2} x_1 x_2 = 1
       \Bigr)
    \\
    &+ \Bigl(
         x \in \widetilde{c}_i \land \inv(x) \turnstile{x}
         \exx{y} y \in \widetilde{c}_{i'}
       \Bigr)_{(i, i') = (1, 2), (2, 1)}
  \end{align*}
\end{proposition}

The relation symbols $\widetilde{c}_1, \widetilde{c}_2 \subseteq A$
should be understood as
\enquote{partial constants} $c_1, c_2 : A$,
with the following properties.
At least one of them is defined,
and if one is defined then the other is its inverse,
which may or may not be defined.

\begin{proof}
  An affine open cover of the scheme $\PP^1_K$
  is given by two copies of the affine line $\mathbb{A}^1_K$,
  we write
  \[ \PP^1_K = \Spec K[X_1] \cup \Spec K[X_2] \rlap{,} \]
  and their intersection is
  \[ \Spec K[X_1] \cap \Spec K[X_2] = \Spec K[X_1, X_2]/(X_1 X_2 - 1)
     \rlap{,} \]
  with the open inclusions corresponding to
  the $K$-algebra maps suggested by the notation.
  In other words,
  as an open subscheme of each $\Spec K[X_i]$,
  the overlap is the standard open $D(X_i)$,
  and the identification corresponds to
  the isomorphism of rings is
  \[ \varphi : K[X_1]_{X_1} \to K[X_2]_{X_2}, \quad
     X_1 \mapsto {X_2}^{-1} \rlap{.} \]

  We want to apply Corollary \ref{corollary-localic-gluing}
  to the open cover of toposes
  \[ (\PP^1_K)_\Zarfp = \E = \E_1 \cup \E_2 \]
  induced by the above
  via Lemma \ref{lemma-for-gluing-Zar}.
  Consider the structure sheaf of $(\PP^1_K)_\Zarfp$,
  which is a sheaf of $K$-algebras,
  as a model
  \[ M_0 \defeq \mathcal{O} \in (\Alg{K})\dashmod(\E) \rlap{.} \]
  Then the $(\Alg{K})$-models $M_0|_{\E_i}$
  are the respective structure sheaves
  (by Lemma \ref{lemma-for-gluing-Zar} (ii)),
  that is,
  they become universal models of
  $\Alg{K} + \angles{c_i : A} + (\loc)$
  when the constant symbol $c_i$
  is interpreted as the global section
  of $M_0|_{\E_i} = \mathcal{O}_{(\Spec K[X_i])_\Zarfp}$
  corresponding to $X_i$.
  In particular,
  $M_0$ satisfies $(\loc)$,
  so we can use as our base theory
  \[ \TT_0 \defeq \Alg{K} + (\loc) \rlap{,} \]
  and we have presentations $(\TT_0 + \EE_i, M_0|_{\E_i} + E_i)$
  of the $\E_i$ with
  \[ \EE_i \defeq \angles{c_i : A} \rlap{.} \]
  Using Lemma \ref{lemma-for-gluing-Zar} (ii),
  the subtopos $\E_1 \cap \E_2$
  is presented in $\E_i$
  by the closed geometric formula
  \[ \phi_i \defeq \inv(c_i) \rlap{.} \]

  It remains to find a diagonal quotient extension
  of $\EE_1 + \phi_1$ and $\EE_2 + \phi_2$ over $\TT_0$
  satisfied by $\mathcal{O}_{\E_1 \cap \E_2}$
  (with the interpretations of $c_i$ coming from $\mathcal{O}_{\E_i}$).
  Since $\mathcal{O}_{\E_1 \cap \E_2}$ is
  the universal local $K[X_1, X_2]/(X_1 X_2 - 1)$-algebra,
  we can take
  \[ \QQ_{\{1, 2\}} \defeq
     \Bigl( \top \turnstile{[]} c_1 c_2 = 1 \Bigr)
     \rlap{.} \]

  Now,
  to form conditional extensions like $\EE_i/p_i$,
  we have to replace the constant symbols $c_i$
  by relation symbols $\widetilde{c}_i \subseteq A$
  with appropriate axioms,
  and rewrite the formulas in which they appear.
  (We need to \enquote{desugar} the syntax.)
  Thus we obtain
  \begin{align*}
    \EE_i &= \angles{\widetilde{c}_i \subseteq A}
      + \Bigl( \top \turnstile{[]} \exx{x} x \in \widetilde{c}_i \Bigr)
      + \Bigl( x \in \widetilde{c}_i \land x' \in \widetilde{c}_i
               \turnstile{x, x'} x = x' \Bigr),
    \\
    \phi_i &= \exx{x} (x \in \widetilde{c}_i \land \inv(x)),
    \\
    \QQ_{\{1, 2\}} &= \exx{x_1, x_2}
      ( x_1 \in \widetilde{c}_1 \land
        x_2 \in \widetilde{c}_2 \land
        x_1 x_2 = 1 ).
  \end{align*}
  Then Corollary \ref{corollary-localic-gluing}
  gives us a theory $\TT$ classified by $(\PP^1_K)_\Zarfp$.
  This $\TT$ contains proposition symbols $p_1$, $p_2$,
  which can however be eliminated since $\TT$ proves
  \[ p_i \doubleturnstile{[]} \exx{x} x \in \widetilde{c}_i
     \rlap{.} \]
  (This is always the case
  when the extensions $\EE_i$ contain at least one constant symbol.)
  Then one can check that the resulting theory
  is syntactically equivalent to $\TT_{\PP^1_K}$
  as defined in the statement.
\end{proof}

In the general case,
we will of course have more than two affine schemes in the covering,
but also,
their intersections
will not be given as single standard open subschemes.
Recall, however,
that the intersection $S_1 \cap S_2$
of two affine open subschemes
of a scheme $S$
can always be covered by open subschemes
which are standard opens in both $S_1$ and $S_2$
\cite[5.3.1. Proposition]{vakil}.

\begin{theorem}
  \label{theorem-gluing-Zar}
  Let $S = \bigcup_{i \in I} S_i$ be a scheme
  covered by affine open subschemes $S_i = \Spec K_i$,
  and for every $i \neq i' \in I$,
  let an open cover
  \[ S_i \cap S_{i'} = \bigcup_{j \in J_{\{i, i'\}}} S_{\{i, i'\}}^j \]
  be given,
  such that
  \[ \Spec K_i \supseteq D(f_{i, i'}^j)
     = S_{\{i, i'\}}^j =
     D(f_{i', i}^j) \subseteq \Spec K_{i'}
     \rlap{,} \]
  with corresponding ring isomorphisms
  \[ \varphi_{i, i'}^j = (\varphi_{i', i}^j)^{-1} :
     (K_i)_{f_{i, i'}^j} \to (K_{i'})_{f_{i', i}^j}
     \rlap{.} \]

  Then the big Zariski topos $S_\Zarfp$
  classifies the theory
  \begin{align*}
    \TT_S \;\defeq\;
    {}&\Ring + (\loc)
    + \angles{p_i}_{i \in I}
    + (\bigvee_{i \in I} p_i)
    + ((\AlgStr{K_i})/p_i)_{i \in I}
    \\
    &+ \Bigl(
         x \in \widetilde{c}_{f_{i, i'}^j} \land \inv(x)
         \turnstile{x}
         p_{i'}
       \Bigr)_{i \neq i' \in I, j \in J_{\{i, i'\}}}
    \\
    &+ \Bigl(
         p_i \land p_{i'}
         \turnstile{[]}
         \bigvee_{j \in J_{\{i, i'\}}} \exx{x}
           (x \in \widetilde{c}_{f_{i, i'}^j} \land \inv(x))
       \Bigr)_{i \neq i' \in I, j \in J_{\{i, i'\}}}
    \\
    &+ \Bigl(
         x \in \widetilde{c}_{f_{i', i}^j} \land
         \inv(x) \land
         y \in \widetilde{c}_\lambda \land
         z \in \widetilde{c}_{\lambda'}
         \turnstile{x, y, z}
         x^n y = z
       \Bigr)_{i \neq i' \in I, j \in J_{\{i, i'\}}, \lambda \in K_i}
       \rlap{,}
  \end{align*}
  where in the last family of axioms,
  $\lambda' \in K_{i'}$ and $n \in \NN$
  are chosen for each $\lambda \in K_i$
  such that
  $\varphi_{i, i'}^j(\lambda) = (f_{i', i}^j)^{-n} \lambda'$.
\end{theorem}

\begin{proof}
  By Lemma \ref{lemma-for-gluing-Zar},
  the topos $\E \defeq S_\Zarfp$ is covered by the open subtoposes
  $\E_i \defeq (S_i)_\Zarfp$.
  The restrictions $\mathcal{O}|_{\E_i}$
  of the $\Ring$-model $\mathcal{O} = \mathcal{O}_{S_\Zarfp}$
  can be extended by $K_i$-algebra structures,
  resulting in universal models of $\Alg{K_i} + (\loc)$.
  So we can use the base theory
  \[ \TT_0 \defeq \Ring + (\loc) \]
  with the model
  \[ M_0 \defeq \mathcal{O} \;\in\;
     (\Ring + (\loc))\dashmod(\E)
     \rlap{,} \]
  and have presentations of the $\E_i$ with
  \[ \EE_i \defeq \AlgStr{K_i} \rlap{.} \]
  The open subtopos $\E_i \cap \E_{i'} \subseteq \E_i$
  is then presented by the closed geometric formula
  \[ \phi_{i, i'} \;\defeq\;
     \bigvee_{j \in J_{\{i, i'\}}} \inv(c_{f_{i, i'}^j})
     \rlap{.} \]

  To find appropriate diagonal quotient extensions $\QQ_{\{i, i'\}}$,
  consider the model
  \[ M_0|_{\E_i \cap \E_{i'}} = \mathcal{O}_{(S_i \cap S_{i'})_\Zarfp}
     \rlap{,} \]
  which carries both a $K_i$-algebra structure
  and a $K_{i'}$-algebra structure.
  For every $j \in J_{\{i, i'\}}$,
  the further restriction to $(S_{\{i, i'\}}^j)_\Zarfp$
  is the universal local $(K_i)_{f_{i, i'}^j}$-algebra
  and at the same time
  the universal local $(K_{i'})_{f_{i', i}^j}$-algebra,
  with the algebra structures coinciding
  via the isomorphism $\varphi_{i, i'}^j$.
  For every $j \in J_{\{i, i'\}}$
  and every $\lambda \in K_i$,
  write
  \[ \varphi_{i, i'}^j(\lambda) = (f_{i', i}^j)^{-n} \lambda' \]
  for some $\lambda' = \lambda'_{j, \lambda} \in K_{i'}$,
  $n = n_{j, \lambda} \in \NN$.
  Then $\mathcal{O}_{(S_i \cap S_{i'})_\Zarfp}$
  satisfies the axioms
  \begin{align*}
    \QQ_{i, i'} \quad&\defeq\quad
    \Bigl\{\quad
      \inv(c_{f_{i', i}^j})
      \turnstile{[]}
      (c_{f_{i', i}^j})^n c_\lambda = c_{\lambda'}
    \quad\mid\quad
      j \in J_{\{i, i'\}},
      \lambda \in K_i
    \quad\Bigr\}
    \rlap{,} \\
    \QQ_{\{i, i'\}} \quad&\defeq\quad \QQ_{i, i'} + \QQ_{i', i}
    \rlap{.}
  \end{align*}
  Note that the sequents in $\QQ_{i, i'}$ also make
  $\inv(c_{f_{i', i}^j}) \turnstile{[]} \inv(c_{f_{i, i'}^j})$
  provable,
  because we can set $\lambda \defeq f_{i, i'}^j$
  and then $\lambda'$ will be invertible
  as an element of $(K_{i'})_{f_{i', i}}$.

  The quotient extension $\QQ_{i, i'}$
  clearly makes each $c_\lambda$
  definable in terms of the $K_{i'}$-algebra structure
  if we assume one of the formulas $\inv(c_{f_{i', i}^j})$.
  But we need $c_\lambda$ to be definable by
  a single formula of $\TT_0 + \EE_{i'} + \phi_{i', i}$.
  Thus, consider
  \[ \phi_{c_\lambda}(x) \quad\defeq\quad
     \bigvee_{j \in J_{\{i, i'\}}} \Bigl(
       \inv(c_{f_{i', i}^j}) \land
       (c_{f_{i', i}^j})^n x = c_{\lambda'} \Bigr)
     \rlap{.} \]
  Then we see that $\QQ_{i, i'}$
  is exactly what we need,
  together with $\phi_{i', i}$,
  to prove
  \[ \phi_{c_\lambda}(x) \doubleturnstile{x} x = c_\lambda
     \rlap{,} \]
  completing our argument why $\QQ_{\{i, i'\}}$
  is a diagonal quotient extension
  of $\EE_i + \phi_{i, i'}$ and $\EE_{i'} + \phi_{i', i}$
  over $\TT_0$.
  (It is automatic that
  $\phi_{c_\lambda}(x)$ is provably functional
  as a formula of $\TT_0 + \EE_{i'} + \phi_{i', i}$,
  see Remark \ref{remark-recognizing-extensions-by-definitions}.)

  Now Corollary \ref{corollary-localic-gluing}
  can finally be applied to the data
  $\TT_0$, $\EE_i$, $\phi_{i, i'}$ and $\QQ_{\{i, i'\}}$,
  considering all the constant symbols $c_\lambda : A$
  to be \enquote{syntactic sugar} for
  relation symbols $\widetilde{c}_\lambda \subseteq A$
  together with appropriate axioms.
  The resulting theory $\TT$
  is syntactically equivalent to $\TT_S$ as given in the statement.
  We have only simplified the axioms
  by dropping $p_i$ in the antecedent
  whenever a formula of the form $x \in \widetilde{c}_\lambda$
  appears there too,
  for some $\lambda \in K_i$,
  since the equivalence
  \[ p_i \doubleturnstile{[]} \exx{x} x \in \widetilde{c}_\lambda \]
  is already contained in
  $(\AlgStr{K_i})/p_i$.
\end{proof}

\begin{remark}
  Theorem \ref{theorem-gluing-Zar}
  can be regarded as a generalization of
  \cite[Proposition 16.3]{blechschmidt:phd},
  where the $\Set$-based points of $S_\Zarfp$
  are identified as the \enquote{local rings over $S$},
  that is,
  local rings $A$
  together with a scheme morphism $\Spec A \to S$.
  One can check that
  this category is indeed equivalent to
  $\TT_S\dashmod(\Set)$.
\end{remark}


\subsection{Conditional extensions and Artin gluing}

We now consider the question
what the classifying topos of
a conditionally extended theory $\TT + \EE/\phi$
looks like,
that is,
we aim to describe $\Set[\TT + \EE/\phi]$
in terms of the geometric morphisms
\[ \begin{tikzcd}
  \Set[\TT + \phi + \EE] \ar[r, "\pi_\EE"] &
  \Set[\TT + \phi] \ar[r, hook, "\pi_\phi"] &
  \Set[\TT] \rlap{,}
\end{tikzcd} \]
where $\pi_\phi$ is an open embedding,
but $\pi_\EE$ can be an arbitrary geometric morphism
(see Theorem \ref{theorem-extension-from-geometric-morphism}).
This situation can perhaps be visualized
by the following picture.
\begin{center} \begin{tikzpicture}
  [x={(1cm, 0cm)}, y={(0cm, .25cm)}, z={(0cm, 1cm)}]
  \draw (1, 0) circle [radius=3];
  \node[label=east:${\Set[\TT]}$] at (4, 0) {};
  \draw (0, 0) circle [radius=1];
  \node[label=east:${\Set[\TT + \phi]}$] at (1, 0) {};
  \draw (0, 0, 2.5) circle [radius=1];
  \draw (1, 0, 1.5) arc [radius=1, start angle=0, end angle=-180];
  \draw (1, 0, 1.5) -- (1, 0, 2.5);
  \draw (-1, 0, 1.5) -- (-1, 0, 2.5);
  \node[label=east:${\Set[\TT + \phi + \EE]}$] at (1, 0, 2) {};
  \draw[->] (0, 0, 1) -- (0, 0, 0);
\end{tikzpicture} \end{center}

In the special case where
the base theory is simply the theory $\angles{p}$
of a proposition symbol $p$ with no axioms,
we can give an answer
directly by analyzing the syntactic site of the resulting theory.
Recall that the \emph{Sierpiński cone} (or Freyd cover) $\scn(\E)$
of a topos $\E$ is the topos with underlying category
the comma category $(\Set \downarrow \Gamma)$,
where $\Gamma : \E \to \Set$ is the global sections functor,
and that it has a canonical subterminal object
$U = (\varnothing, 1_\E, !)$,
where $! : \varnothing \to \Gamma(1_\E) = \{*\}$,
such that the corresponding open subtopos is
$\scn(\E)_{o(U)} \simeq \E$,
while its closed complement is just a point,
$\scn(\E)_{c(U)} \simeq \Set$.
This corresponds to the fact that
$\angles{p} + \TT/p + p$ is equivalent to $\TT$
and $\angles{p} + \TT/p + \lnot p$
is equivalent to the empty theory
(over the empty signature)
classified by $\Set$.

\begin{proposition}
  \label{proposition-sierpinski-cone}
  For any geometric theory $\TT$,
  the theory $\angles{p} + \TT/p$
  is classified by the Sierpiński cone $\scn(\Set[\TT])$.
\end{proposition}

\begin{proof}
  A construction of the Sierpiński cone
  on the level of sites
  proceeds as follows.
  Given a site $(C, J)$,
  we define $C'$ by freely adjoining a terminal object to $C$,
  that is, $\Ob(C') = \Ob(C) \sqcup \{*\}$
  and $\abs{\Hom(c, *)} = 1$, $\Hom(*, c) = \varnothing$
  for all $c \in \Ob(C)$.
  The topology $J'$ is given by simply leaving the covering sieves of any $c \in C$
  (which are still sieves in $C'$, as there is no arrow $* \to c$)
  as they are
  and declaring only the maximal sieve on $*$ to be covering.
  The axioms of a Grothendieck topology are clearly satisfied.
  Then a sheaf $F'$ on $(C', J')$
  is exactly a sheaf $F$ on $(C, J)$
  together with a set $F'(*)$
  and a map $F'(*) \to \Gamma(F)$,
  as required.

  We claim that this is precisely the relationship
  between the syntactic sites $C_\TT$ and $C_{\angles{p} + \TT/p}$.
  First note that we have a full and faithful functor
  \[ C_{\angles{p} + \TT/p + p} \to C_{\angles{p} + \TT/p}, \quad
     \{\vec{x}.\; \phi\} \mapsto \{\vec{x}.\; \phi \land p\}
     \]
  with essential image $C_{\angles{p} + \TT/p} / \{[].\; p\}$,
  the subcategory of all $\{\vec{x}.\; \phi\}$
  with $\phi \turnstile{\vec{x}} p$ provable in $\angles{p} + \TT/p$.
  (This is true for any closed formula of any geometric theory.)
  But of course $C_{\angles{p} + \TT/p + p} \simeq C_{\TT}$.
  And if an object $\{\vec{x}.\; \phi\} \in C_{\angles{p} + \TT/p}$
  does not lie over $\{[].\; p\}$,
  we can show that the context $\vec{x}$ must be empty
  (because $\angles{p} + \TT/p$ contains the axiom
  $\top \turnstile{x \oftype S} p$ for any of its sorts)
  and, by induction over $\phi$,
  that $\top \turnstile{[]} \phi$ is provable in $\angles{p} + \TT/p$.
  So the only missing object (up to isomorphism)
  is the terminal object $\{[].\; \top\}$,
  and this terminal object is strict,
  since $\top \turnstile{[]} p$ is not provable in $\angles{p} + \TT/p$.
  It is easy to verify that
  the Grothendieck topology on $C_{\angles{p} + \TT/p}$
  is also the one coming from $C_{\TT}$ as above.
\end{proof}

For the general case,
recall from \cite[Proposition 4.5.6]{elephant}
that any topos $\E$
equipped with a subterminal object $U$
can be reconstructed from the toposes $\E_{o(U)}$ and $\E_{c(U)}$
(the open respectively closed subtopos of $\E$
corresponding to $U$)
and the left exact functor
\[ F \defeq j^* \circ i_* : \E_{o(U)} \to \E_{c(U)} \rlap{,} \]
called the \emph{fringe functor} of $\E$
(with respect to $U$),
where $i : \E_{o(U)} \to \E$ and $j : \E_{c(U)} \to \E$
are the inclusion geometric morphisms.
Namely,
$\E$ is equivalent to the \emph{Artin gluing} of $F$,
defined, for any left exact functor $F : \E_1 \to \E_2$ between toposes,
as the comma category
\[ \Gl(F) \defeq (\E_2 \downarrow F) \rlap{,} \]
equipped with the subterminal object
$(\varnothing_{\E_2}, 1_{\E_1}, !)$,
where $! : \varnothing_{\E_2} \to F(1_{\E_1}) = 1_{\E_2}$.
In our situation,
the topos $\Set[\TT + \EE/\phi]$
is therefore an Artin gluing of the open subtopos
\[ \Set[\TT + \EE/\phi + \phi] = \Set[\TT + \phi + \EE] \]
and its closed complement
\[ \Set[\TT + \EE/\phi + \lnot \phi] =
   \Set[\TT + \lnot \phi + \EE/\bot] \simeq
   \Set[\TT + \lnot \phi]
   \rlap{.} \]

\begin{conjecture}
  \label{conjecture-replace-open-subtopos}
  Let $\TT$ be a geometric theory,
  $\phi$ a closed formula of $\TT$
  and $\EE$ an extension of $\TT + \phi$.
  Then the theory $\TT + \EE/\phi$
  is classified by the topos
  \[
    \Gl\Bigl( \Set[\TT + \phi + \EE]
    \xrightarrow{{\pi_{\phi + \EE}}_*}
    \Set[\TT]
    \xrightarrow{{\pi_{\lnot \phi}}^*}
    \Set[\TT + \lnot \phi] \Bigr)
    \rlap{.}
  \]
\end{conjecture}

The functor ${\pi_{\lnot \phi}}^* \circ {\pi_{\phi + \EE}}_*$
can also be understood as the composition $F \circ {\pi_\EE}_*$,
where $F : \Set[\TT + \phi] \to \Set[\TT + \lnot \phi]$
is the fringe functor of $\Set[\TT]$,
that is, $\Gl(F) \simeq \Set[\TT]$.
What the conjecture says, then,
is that the topos $\Set[\TT + \EE/\phi]$
is the result of replacing
the open subtopos $\Set[\TT + \phi]$ of $\Set[\TT]$
by $\Set[\TT + \phi + \EE]$
in a canonical way,
using the given geometric morphism $\pi_\EE$.

To prove it,
we would have to show that the given left exact functor
is isomorphic to the fringe functor of $\Set[\TT + \EE/\phi]$.
One attempt to do this is to show that
the following diagram commutes.
\[ \begin{tikzcd}[row sep=small]
  \Set[\TT + \EE/\phi + \phi]
    \ar[rd, "{\pi_\phi}_*"] \ar[d, equals] & &
  \Set[\TT + \EE/\phi + \lnot \phi]
    \ar[dd, "{\pi_{\EE/\phi}}_*", "\rotatebox{90}{$\simeq$}"']
  \\
  \Set[\TT + \phi + \EE]
    \ar[d, "{\pi_\EE}_*"] &
  \Set[\TT + \EE/\phi]
    \ar[dd, "{\pi_{\EE/\phi}}_*"] \ar[ru, "{\pi_{\lnot \phi}}^*"] &
  \\
  \Set[\TT + \phi]
    \ar[dr, "{\pi_\phi}_*"']
    \ar[ur, phantom, "\rotatebox{45}{$\cong$}"] & &
  \Set[\TT + \lnot \phi]
  \\
  & \Set[\TT] \ar[ur, "{\pi_{\lnot \phi}}^*"']
    \ar[uuur, phantom, "\rotatebox{45}{$\Rightarrow$}"] &
\end{tikzcd} \]
The left half, consisting of direct image functors, does commute.
In the right half,
where direct and inverse images occur,
we can write a \enquote{Beck--Chevalley} natural transformation
which we would hope to be an isomorphism.
For this, however,
it cannot suffice to use the information that
the corresponding square of geometric morphisms is a pullback
with the horizontal geometric morphisms closed embeddings
and the right vertical one an equivalence,
since this would also be true,
for example,
if we wrote instead of $\Set[\TT + \EE/\phi]$
the disjoint union of $\Set[\TT + \phi + \EE]$
and $\Set[\TT + \lnot \phi]$
(which is the Artin gluing of the constant functor
with value $1$).

If we assume the conjecture to be true,
we can spell out another special case
(containing Proposition \ref{proposition-sierpinski-cone})
and give a syntactic presentation for the
\emph{open mapping cylinder} of a geometric morphism,
which is the Artin gluing of its direct image functor.

\begin{proposition}
  Let $\pi_\EE : \Set[\TT + \EE] \to \Set[\TT]$
  be a geometric morphism
  with chosen syntactic presentation.
  Assuming Conjecture \ref{conjecture-replace-open-subtopos},
  the theory $\TT + \angles{p} + \EE/p$
  is classified by the open mapping cylinder of $\pi_\EE$,
  $\Gl({\pi_\EE}_*)$.
\end{proposition}

\begin{proof}
  We regard $\EE$ as an extension of $\TT + \angles{p} + p$
  and apply the conjecture to the base theory $\TT + \angles{p}$.
  What then remains to show is that the composite
  \[
    \Set[\TT] \simeq \Set[\TT + \angles{p} + p]
    \xrightarrow{{\pi_p}_*}
    \Set[\TT + \angles{p}]
    \xrightarrow{{\pi_{\lnot p}}^*}
    \Set[\TT + \angles{p} + \lnot p] \simeq \Set[\TT]
  \]
  is isomorphic to the identity functor.
  This is true,
  since $\Set[\TT + \angles{p}]$
  is the product of $\Set[\TT]$
  with the Sierpiński topos $\Set[\angles{p}]$,
  which can indeed be obtained as
  the Artin gluing along the identity functor.
\end{proof}

\begin{remark}
  There is a dual notion to conditional extensions,
  where a model involves a model extension in some closed subtopos
  instead of an open one,
  and one could make a dual conjecture
  about replacing a closed subtopos,
  using the inverse image part of a geometric morphism.
  Namely,
  if $\EE$ is an extension
  (without function symbols)
  of $\TT + \lnot \phi$
  (where $\phi$ is a closed geometric formula
  and $\lnot \phi$ is the axiom $\phi \turnstile{[]} \bot$),
  then we would denote $\EE/\lnot \phi$ the extension of $\TT$
  consisting of the following.
  \begin{itemize}
    \item
      For every sort $S$ of $\EE$,
      a sort $S$ and the axioms
      $\phi \turnstile{x, x' \oftype S} x = x'$
      and $\phi \turnstile{[]} \ex{x}{S} \top$.
    \item
      For every relation symbol $R \subseteq \vec{S}$ of $\EE$,
      a relation symbol $R \subseteq \vec{S}$
      and the axiom $\phi \turnstile{\vec{x} \oftype \vec{S}} R(\vec{x})$.
    \item
      For every axiom $\psi \turnstile{\Gamma} \chi$,
      the axiom $\psi \turnstile{\Gamma} \chi \lor \psi$.
  \end{itemize}
\end{remark}

\newpage

\section{Theories of presheaf type}
\label{section-presheaf-type}

\subsection{Introduction}

In this section,
we develop some results,
to be used in the next section,
concerning geometric theories
whose classifying topos is
of the form
\[ \Set[\TT] \simeq \PSh(C) \rlap{,} \]
called theories \emph{of presheaf type}
\cite{beke:presheaf-type}.
This class of theories is of great significance
for the problem of finding a concise syntactic presentation
of a given topos,
since it offers the following shortcut.
Whenever we know that a theory $\TT$ is of presheaf type,
there is a canonical (presheaf) site of definition for $\Set[\TT]$
defined in a purely categorical way
from the category $\TT\dashmod(\Set)$
of $\Set$-based models of $\TT$.
Namely,
$\TT$ is of presheaf type if and only if
\[ \Set[\TT] \simeq [\TT\dashmod(\Set)_c, \Set] \rlap{,} \]
where $\TT\dashmod(\Set)_c$ denotes the category of
compact objects in $\TT\dashmod(\Set)$,
also known as the finitely presentable models of $\TT$.
In particular,
a theory of presheaf type is fully determined
by its $\Set$-based models,
in contrast to the fact that a general geometric theory
can have no $\Set$-based models at all
without being inconsistent.

To make use of this feature,
we need ways to recognize theories of presheaf type
without having to construct presheaf sites
for their classifying toposes by hand.
The basic (and already very useful) result of this kind
is that a theory is of presheaf type
if its axioms satisfy certain syntactic restrictions,
such as, in the simplest case,
algebraic theories.
If a given theory
does not meet this condition,
one can of course try to replace the problematic axioms
by equivalent axioms of the required form,
or, much more drastically,
to transform the whole theory into a Morita equivalent theory
satisfying the syntactic restrictions,
using equivalence extensions
to introduce and eliminate
function symbols, relation symbols and sorts.
However, it is not clear how to find such a transformation
if it is not obvious from the theory.

Instead, one can also regard $\TT$ as a quotient
of a theory $\TT_0$ which contains only syntactically simple axioms.
Then, by the duality between quotient theories and subtoposes,
the additional axioms of $\TT$ correspond to a Grothendieck topology
on the canonical presheaf site for $\Set[\TT_0]$.
If this topology is rigid,
meaning that its sheaves
are just the presheaves on some smaller presheaf site
and therefore $\TT$ is again of presheaf type,
we speak of a \emph{rigid-topology quotient}.
For testing this rigidity condition,
it is important to have a convenient description
of the topology induced by the additional axioms,
which we provide in
Theorem \ref{theorem-induced-topology}.

This strategy can also be used
to show that certain classes of axioms,
such as axioms stated in the empty context,
can be added to \emph{any} theory of presheaf type
without destroying the presheaf type property.
This allows us, for example,
to start with a syntactically simple base theory,
apply some equivalence extension to it
(possibly involving syntactically complex axioms),
and \emph{then} add an axiom
that is known to always preserve presheaf type.
This is in contrast to the syntactically simple axioms,
like algebraic axioms,
which are only harmless in a base theory
but not when added to an arbitrary theory of presheaf type,
as we will see.

We further explore the possibilities
of building up theories of presheaf type incrementally
by looking at extensions involving function
symbols.
While adding finitely many constant symbols
does in fact always preserve presheaf type,
we quickly obtain negative results after that.
Not only can countably many constant symbols
or a single unary function symbol
destroy presheaf type,
they can even do so
when added to a theory which
is trivial up to Morita equivalence,
that is, classified by $\Set$.

The following table summarizes our findings
about which syntactic forms of extensions
always preserve presheaf type
and which can destroy presheaf type.

\begingroup
\small
\begin{center} \begin{tabular}{l l l}
  \toprule
  Extension & Always preserves & Reference \\
  & presheaf type & \\
  \midrule
  Algebraic axiom &
  no &
  Example \ref{example-single-algebraic-axiom} \\
  Axiom in empty context &
  yes &
  Corollary \ref{corollary-axiom-in-empty-context} \\
  Countably many axioms in empty context &
  no &
  Remark \ref{remark-infinitely-many-axioms-in-empty-context} \\
  Any number of negated axioms &
  yes &
  Corollary \ref{corollary-negated-axioms} \\
  \addlinespace
  Constant symbol &
  yes &
  Proposition \ref{proposition-add-constant-symbol} \\
  Function symbol &
  no &
  Example \ref{example-theories-of-a-function} \\
  Countably many constant symbols &
  no &
  Examples \ref{example-add-many-constants-to-quotient}
  and \ref{example-theories-of-a-function} \\
  \bottomrule
\end{tabular} \end{center}
\endgroup

\subsection{Background}

A geometric theory $\TT$ is \emph{of presheaf type}
if it admits a universal model in some presheaf topos,
that is,
\[ \Set[\TT] \simeq \PSh(C) \]
for some small category $C$.
To discuss this notion properly,
we first have to recall the definition of compact objects.

\begin{definition}
  Let $C$ be a category with all filtered colimits.
  An object $X$ of $C$
  is \emph{compact}
  if $\Hom(X, \hole) : C \to \Set$
  preserves filtered colimits.
  We denote the full subcategory of $C$
  on the compact objects
  by $C_c$.
\end{definition}

The compact objects in $\TT\dashmod(\Set)$
are usually called the \emph{finitely presentable} models of $\TT$,
but we will simply call them the compact models
and denote the category of these models by
\[ \TT\dashmod(\Set)_c \rlap{.} \]
For this to make sense,
we must show that the category $\TT\dashmod(\Set)$
has all filtered colimits.
We show this for an arbitrary (Grothendieck) topos $\E$ instead,
to illustrate how well-behaved models of geometric theories are
with respect to filtered colimits.

\begin{lemma}
  \label{lemma-filtered-colimits-of-models}
  Let $\TT$ be a geometric theory and $\E$ be a (Grothendieck) topos.
  Then the category $\mod{\TT}{\E}$
  admits all filtered colimits,
  and for every geometric formula-in-context $\phi$,
  the interpretation functor
  \[ \mod{\TT}{\E} \to \E, \quad
     M \mapsto \interpretation{\phi}_M \]
  preserves filtered colimits.
\end{lemma}

\begin{proof}
  Let $(M_i)_{i \in I}$ be a filtered diagram in $\mod{\TT}{\E}$.
  We start with the case where $\TT$ is the empty theory
  over some signature $\Sigma$,
  meaning that $\TT$-models are just $\Sigma$-structures.
  For every sort $A$ of $\Sigma$,
  set $\interpretation{A}_M \defeq
  \colim_{i \in I} \interpretation{A}_{M_i}$.
  For every function symbol $f : A_1 \times \dots \times A_n \to B$,
  we have
  $\interpretation{A_1}_M \times \dots \times \interpretation{A_n}_M =
  \colim_{i \in I} (\interpretation{A_1}_{M_i} \times \dots \times \interpretation{A_n}_{M_i})$
  since filtered colimits commute with finite limits in $\E$,
  and the $\interpretation{f}_{M_i}$ induce an arrow
  $\interpretation{f}_M : \interpretation{A_1}_M \times \dots \times \interpretation{A_n}_M
  \to \interpretation{B}_M$.
  Similarly,
  the interpretations $\interpretation{R}_{M_i} \hookrightarrow
  \interpretation{A_1}_{M_i} \times \dots \times \interpretation{A_n}_{M_i}$
  of a relation symbol
  induce an arrow
  $\interpretation{R}_M \defeq \colim_{i \in I} \interpretation{R}_{M_i}
  \hookrightarrow
  \interpretation{A_1}_M \times \dots \times \interpretation{A_n}_M$,
  which is a monomorphism
  because filtered colimits commute with pullbacks in $\E$.
  In this way,
  the canonical maps $\interpretation{A}_{M_i} \to \interpretation{A}_M$
  obviously constitute $\Sigma$-structure homomorphisms.
  And for a cocone of $\Sigma$-structure homomorphisms $M_i \to \widetilde{M}$,
  the induced maps $\interpretation{A}_M \to \interpretation{A}_{\widetilde{M}}$
  respect the interpretations of function and relation symbols too.

  Now let $\phi$ be a geometric formula (over $\Sigma$)
  in the context $x_1 : A_1, \dots, x_n : A_n$.
  Then we can prove by induction on the structure of $\phi$
  that $\interpretation{\phi}_M =
  \colim_{i \in I} \interpretation{\phi}_{M_i}$
  as subobjects of
  $\interpretation{A_1}_M \times \dots \times \interpretation{A_n}_M =
  \colim_{i \in I} (\interpretation{A_1}_{M_i} \times \dots
  \times \interpretation {A_n}_{M_i})$
  (using for example that filtered colimits
  commute with image factorizations).
  This shows in particular
  that any axiom $\phi \turnstile{} \psi$
  satisfied by all $M_i$
  is also satisfied by $M$,
  so $\mod{\TT}{\E}$ is closed under filtered colimits
  in the category of $\Sigma$-structures in $\E$.
\end{proof}

Let us also prove
the following lemma on compact objects,
which will be a convenient basic tool
for drawing conclusions about compact models
in some situations.

\begin{lemma}
  \label{lemma-adjoint-functor-preserves-compacts}
  Let $F \dashv G$ be adjoint functors
  $F : C \to D$, $G : D \to C$
  between categories admitting filtered colimits.
  If $G$ preserves filtered colimits,
  then $F$ preserves compact objects.
\end{lemma}

\begin{proof}
  Let $X \in C$ be compact
  and let $Y_i$ be a filtered diagram in $D$.
  Then we have
  \begin{align*}
    \Hom(F(X), \colim_i Y_i) &= \Hom(X, \colim_i G(Y_i)) \\
    = \colim_i \Hom(X, G(Y_i)) &= \colim_i \Hom(F(X), Y_i)
    \rlap{.}
    \qedhere
  \end{align*}
\end{proof}

Now we are ready to understand the usefulness
of theories of presheaf type.
It consists in the fact that
if a geometric theory $\TT$ is of presheaf type,
then $\TT$ is classified by the topos
\[ \Set[\TT] =
   \PSh({\TT\dashmod(\Set)_c}^\op) =
   [\TT\dashmod(\Set)_c, \Set]
   \rlap{.} \]
That is,
if there is any presheaf site for $\Set[\TT]$,
then the canonically given category
${\TT\dashmod(\Set)_c}^\op$
is also a presheaf site of definition for $\Set[\TT]$.
We note that the category $\TT\dashmod(\Set)_c$
is essentially small,
so that it can actually be used as a site.
There is also a simple description
of the universal $\TT$-model $M$ in $[\TT\dashmod(\Set)_c, \Set]$,
namely,
$M$ is the \enquote{tautological} model,
interpreting a sort $A$
by the functor
\[ \TT\dashmod(\Set)_c \to \Set, \quad M' \mapsto
   \interpretation{A}_{M'}
   \rlap{,} \]
and similarly for relation and function symbols
\cite[Theorem 6.1.1]{caramello:tst}.

For a theory of presheaf type $\TT$,
the category $\TT\dashmod(\Set)$ is finitely accessible.
This can be a useful criterion
for showing that a theory is not of presheaf type.

The basic source of theories of presheaf type
is the following.
Recall that a geometric theory is a \emph{Horn} theory,
if all of its axioms are of the form
$\phi \turnstile{\Gamma} \psi$
with $\phi$ and $\psi$ finite conjunctions of atomic formulas.
In particular,
algebraic theories are Horn theories.
Recall also that in a \emph{cartesian} theory
the formulas $\phi$ and $\psi$ can
in addition to finite conjunctions
also contain existential quantifiers,
but only if these refer to unique existence,
provably relative to all preceding axioms
in some chosen well-ordering.
Then, all Horn theories
and more generally all cartesian theories
are of presheaf type
\cite[Theorem 2.1.8]{caramello:tst}.
For Horn theories,
there is even a good general description
of the compact models.

\begin{lemma}
  \label{lemma-compact-Horn-models}
  Let $\TT$ be a Horn theory.
  Then a model $M$ of $\TT$ is compact
  if and only if
  it is presented by some Horn formula
  (finite conjunction of atomic formulas)
  in context $\vec{x}.\, \phi$,
  meaning that there is a natural isomorphism
  \[ \Hom(M, \hole) \cong \interpretation{\phi}_\hole
     \;:\; \TT\dashmod(\Set) \to \Set
     \rlap{.} \]
  Furthermore,
  every Horn formula in context
  presents some model.
\end{lemma}

\begin{proof}
  See \cite[Lemma 4.3]{blechschmidt:nullstellensatz}.
\end{proof}

In other words,
the compact models of a Horn theory
are the ones presented by
finitely many generators
(with specified sorts)
and finitely many relations
in the form of atomic formulas.

For understanding the notion of rigid-topology quotients,
we should recall the Comparison Lemma,
which allows to eliminate some of the objects
of a site under certain conditions.
If $(C, J)$ is a site
and $C' \subseteq C$ is any full subcategory,
then we can restrict the Grothendieck topology $J$ to $C'$,
letting a sieve $S$ on an object $c \in C'$
be covering for $J|_{C'}$
if and only if the sieve generated by the arrows of $S$
in the bigger category $C$
is covering for $J$.
In other words,
$J|_{C'}$ is the biggest topology on $C'$
such that the inclusion $C' \hookrightarrow C$
preserves covers.
A \emph{dense} full subcategory $C' \subseteq C$
of a site $(C, J)$
is one such that every object of $C$
can be covered by (a sieve generated by arrows from)
objects in $C'$.

\begin{theorem}[Comparison Lemma]
  If $C'$ is a dense full subcategory
  of a site $(C, J)$,
  then restricting sheaves on $C$
  to $C'$ is an equivalence of categories
  \[ \Sh(C', J|_{C'}) \simeq \Sh(C) . \]
\end{theorem}

\begin{proof}
  See \cite[Theorem C2.2.3]{elephant}.
\end{proof}

Now suppose we want to add some axioms
to a theory of presheaf type
and would like to show that the resulting theory
is again of presheaf type.
By the duality between
quotient theories and subtoposes,
and the fact that the latter
correspond to Grothendieck topologies
on any given site of definition,
a quotient extension $\QQ$ of a theory of presheaf type $\TT$
induces a topology on ${\TT\dashmod(\Set)_c}^\op$,
which we denote $J_\QQ$.
If there is a dense full subcategory
of ${\TT\dashmod(\Set)_c}^\op$
with respect to $J_\QQ$,
such that $J_\QQ$ becomes the trivial topology
when restricted to this subcategory,
then by the Comparison Lemma,
$\TT + \QQ$ is again of presheaf type.
This condition can equivalently be expressed by saying that
the \emph{irreducible} objects of ${\TT\dashmod(\Set)_c}^\op$,
which admit no covering sieve except the maximal sieve,
are dense.
A topology with this property is called a \emph{rigid} topology.

\begin{definition}
  Let $\TT$ be a theory of presheaf type.
  We say that a quotient extension $\QQ$ of $\TT$
  is a \emph{rigid-topology quotient}
  if the Grothendieck topology $J_\QQ$
  on ${\mod{\TT}{\Set}_c}^\op$
  is rigid.
\end{definition}

Note that for a quotient extension $\QQ$
of an arbitrary geometric theory $\TT$,
we would have to fix a site of definition for $\Set[\TT]$
before we can ask if the induced topology is rigid.

After proving that some $\QQ$
is a rigid-topology quotient,
we would probably like to have a description of
the compact models of $\TT + \QQ$.
From Lemma \ref{lemma-filtered-colimits-of-models}
it follows that
the full subcategory
\[ (\TT + \QQ)\dashmod(\Set) \subseteq \TT\dashmod(\Set) \]
is closed under filtered colimits.
With this,
one easily sees that a $(\TT + \QQ)$-model
which is compact as a $\TT$-model
is also compact as a $(\TT + \QQ)$-model
(even without any presheaf type conditions).
The converse is part of the following result,
which simultaneously answers the question
whether $\QQ$ being a rigid-topology quotient
is necessary for $\TT + \QQ$ to be again of presheaf type.

\begin{lemma}
  \label{lemma-rigid-topology-quotients-compact-models}
  A quotient extension $\QQ$
  of a theory of presheaf type $\TT$
  is a rigid-topology quotient
  if and only if
  $\TT + \QQ$ is again of presheaf type
  and additionally
  every compact $(\TT + \QQ)$-model
  is also compact as a $\TT$-model.
\end{lemma}

\begin{proof}
  See \cite[Theorem 8.2.6]{caramello:tst}.
\end{proof}

So for example,
take as $\TT$ the theory consisting of a sort $A$
and countably many constant symbols $c_n \oftype A$,
and let $\QQ$ be the axiom
\[ \top \turnstile{x, y \oftype A} x = y \rlap{.} \]
Then both $\TT$ and $\TT + \QQ$ are algebraic theories
and thus of presheaf type.
But in the unique $\Set$-based model of $\TT + \QQ$,
the constants $c_n$ are all identified,
which can not happen in a compact model of $\TT$,
by Lemma \ref{lemma-compact-Horn-models}.
So $\QQ$ is not a rigid-topology quotient.

\subsection{Rigid-topology quotients}

Let $M \in \TT\dashmod(\Set)_c$ be a compact model
of a theory of presheaf type
and let
\[ \phi \turnstile{\vec{x} \oftype \vec{A}} \psi \]
be an axiom that we would like to add to $\TT$.
This axiom is satisfied in the model $M$
if and only if we have
$\interpretation{\phi}_M \subseteq \interpretation{\psi}_M$.
But in any case,
given some $\vec{x} \in \interpretation{\phi}_M$,
we can consider the collection of all arrows
$f : M \to M'$ in $\TT\dashmod(\Set)_c$
that map $\vec{x}$ into $\interpretation{\psi}_{M'}$,
which turns out to be a cosieve on $M$,
since any further model homomorphism $M' \to M''$
preserves truth of the geometric formula $\psi$.
We now show that these sieves
in ${\TT\dashmod(\Set)_c}^\op$
generate the Grothendieck topology
corresponding to the axiom
$\phi \turnstile{\vec{x} \oftype \vec{A}} \psi$.

\begin{theorem}
  \label{theorem-induced-topology}
  Let $\TT$ be a theory of presheaf type
  and let $\QQ =
  \{\, \phi \turnstile{\vec{x} \oftype \vec{A}} \psi \,\}$
  be a quotient extension of $\TT$
  adding a single axiom.
  Then the Grothendieck topology $J_\QQ$
  on ${\mod{\TT}{\Set}_c}^\op$
  is generated by the sieves
  (cosieves in $\mod{\TT}{\Set}_c$)
  \[ S_{\vec{x}} =
     S_{M, \vec{x}, \psi} \defeq
     \{\, f : M \to M' \mid
     f(\vec{x}) \in \interpretation{\psi}_{M'}
     \,\}
     \rlap{,} \]
  where $M \in \TT\dashmod(\Set)_c$
  and $\vec{x} = (x_1, \dots, x_n)
  \in \interpretation{\phi}_M \subseteq
  \interpretation{A_1}_M \times \dots \times \interpretation{A_n}_M$.
\end{theorem}

\begin{proof}
  The axiom $\phi \turnstile{\vec{x} \oftype \vec{A}} \psi$
  is satisfied in a model $M \in \TT\dashmod(\E)$
  (in any topos)
  if and only if
  the inclusion
  $\interpretation{\psi \land \phi}_M \hookrightarrow
  \interpretation{\phi}_M$
  of subobjects of
  $\interpretation{A_1}_M \times \dots \times \interpretation{A_n}_M$
  is an isomorphism.
  For the universal $\TT$-model in $[\TT\dashmod(\Set)_c, \Set]$,
  this is the inclusion
  \[ \begin{tikzcd}
    \interpretation{\psi \land \phi}_\hole \ar[r, hook] &
    \interpretation{\phi}_\hole
  \end{tikzcd} \]
  of functors $\TT\dashmod(\Set)_c \to \Set$.
  The classifying topos of $\TT + \QQ$
  is therefore the greatest subtopos
  (corresponding to the Grothendieck topology
  with the fewest covering sieves)
  with the property that
  this inclusion is a \emph{local isomorphism},
  that is, it becomes an isomorphism when sheafified.

  Now, let $M \in \TT\dashmod(\Set)_c$
  and let $\vec{x} \in \interpretation{\phi}_M$ be given,
  corresponding to an arrow
  $\Hom(M, \hole) \to \interpretation{\phi}_\hole$.
  Then the cosieve $S_{\vec{x}}$ on $M$,
  regarded as a subobject of $\Hom(M, \hole)$
  in $[\TT\dashmod(\Set)_c, \Set]$,
  is the pullback
  \[ \begin{tikzcd}
    \interpretation{\psi \land \phi}_\hole \ar[r, hook] &
    \interpretation{\phi}_\hole \\
    S_{\vec{x}} \ar[u] \ar[r, hook]
    \ar[ru, phantom, very near start, "\urcorner"] &
    \Hom(M, \hole) \ar[u]
    \rlap{.}
  \end{tikzcd} \]
  Thus, if the top row is a local isomorphism,
  then the bottom row
  is a local isomorphism
  for all $M$ and $\vec{x}$.
  But the converse is also true,
  since any local section
  $\widetilde{x} \in a(\interpretation{\phi}_\hole)(M)$
  of the sheafification of $\interpretation{\phi}_\hole$
  is locally given by some $\vec{x} \in \interpretation{\phi}_{M'}$,
  which then locally lies in $\interpretation{\psi \land \phi}_\hole$.
  From this, the statement follows,
  because the covering sieves of a Grothendieck topology
  are precisely those subobjects of representable presheaves
  which are local isomorphisms.
\end{proof}

\begin{remark}
  Theorem \ref{theorem-induced-topology}
  is closely related to \cite[Theorem 8.1.10]{caramello:tst},
  which describes the topology $J_\QQ$
  by only one generating sieve per axiom,
  assuming that the axioms are given in a certain form,
  basically consisting of formulas that present compact models.
  While this yields much more concise descriptions of the topology,
  it seems harder to apply to concrete theories,
  as it might be nontrivial to determine
  whether a formula presents a model
  (this property is characterized in
  \cite[Theorem 6.1.13]{caramello:tst}
  as being an irreducible object of the syntactic site).
  Also,
  for our main purpose of recognizing the topology $J_\QQ$ as rigid,
  the big number of covering sieves $S_{\vec{x}}$
  in Theorem \ref{theorem-induced-topology}
  will actually be quite convenient
  and can therefore be seen as a feature in this context.
\end{remark}

If the quotient extension $\QQ$ contains more than one axiom,
the topology $J_\QQ$ is of course generated by
the union of all the sets of generators for the individual axioms.
(So there is one generating cosieve
for every axiom, compact model and
appropriate family of elements of the model.)
It should be noted that
the collection of cosieves of the form $S_{\vec{x}}$
satisfies the pullback-stability axiom of a Grothendieck topology
(in the terminology of \cite[C2.1]{elephant},
it is a (sifted) coverage).
Indeed, if $g : M \to M'$ is a homomorphism between compact models
then the push-forward of a cosieve $S_{\vec{x}}$ on $M$
for a family of elements $\vec{x} \in \interpretation{\phi}_M$
is simply the cosieve $S_{g(\vec{x})}$ on $M'$,
and we do have $g(\vec{x}) \in \interpretation{\phi}_{M'}$.

As a first application of Theorem \ref{theorem-induced-topology},
we can give a short proof of the following result.

\begin{corollary}[see {\cite[Theorem 8.2.5]{caramello:tst}}]
  Let $\TT$ be a theory of presheaf type
  and $\QQ$ a quotient extension of $\TT$.
  Then a compact model of $\TT$
  is $J_\QQ$-irreducible
  as an object of ${\mod{\TT}{\Set}_c}^\op$
  if and only if
  it satisfies the axioms of $\QQ$.
\end{corollary}

\begin{proof}
  A compact model $M$ of $\TT$
  satisfies an axiom $\phi \turnstile{\vec{x} : \vec{A}} \psi$
  if and only if
  for every family of elements $\vec{x} \in \interpretation{\phi}_M$,
  the cosieve $S_{\vec{x}}$
  from Theorem \ref{theorem-induced-topology}
  is the maximal cosieve on $M$.
  This shows the \enquote{only if} direction.
  For the \enquote{if} direction,
  one can check that saturating a collection of sieves
  with respect to the transitivity condition of Grothendieck topologies
  preserves the pullback-stability of the collection and
  can never produce any nontrivial covers of an object
  if there were none before.
\end{proof}

\begin{remark}
  We can understand the cosieves $S_{\vec{x}}$
  from Theorem \ref{theorem-induced-topology}
  as \emph{operations} correcting the failure of a single instance
  of one of the axioms of $\QQ$.
  Showing that $J_\QQ$ is rigid
  then amounts to
  providing an algorithm that turns any compact $\TT$-model
  into a $(\TT + \QQ)$-model using these operations,
  as follows.
  Starting with a compact $\TT$-model $M$
  which is not a model of $\TT + \QQ$,
  we have to pick an axiom $\phi \turnstile{\vec{x} : \vec{A}} \psi$
  of $\QQ$
  and a family $\vec{x}$ of elements of (the appropriate sorts of) $M$
  with $\phi(\vec{x})$ but not $\psi(\vec{x})$.
  Then, if we are lucky,
  the cosieve $S_{\vec{x}}$ is generated by a single arrow
  \[ M \to M' \rlap{,} \]
  meaning that we have reduced
  the problem of covering $M$ by irreducibles
  to covering $M'$ by irreducibles
  by the transitivity property of the Grothendieck topology $J_\QQ$.
  If we can repeat this procedure a finite number of times
  and arrive at a model that satisfies all axioms in $\QQ$,
  then we are done;
  we have covered $M$ by a single irreducible object.

  It can of course happen that $S_{\vec{x}}$ is not a principal cosieve,
  so that we need multiple arrows
  \[ M \to M'_i \]
  to generate it.
  Then $S_{\vec{x}}$ should perhaps be seen as
  a nondeterministic operation,
  that will turn $M$ into one of the $M'_i$,
  but we don't know in advance which one it will be.
  In this case,
  we have to deal with all $M'_i$,
  and our algorithm must terminate in the sense that
  any possible path of execution is finite.
  (On the other hand,
  we can also be even \emph{more} lucky,
  in that $S_{\vec{x}}$ is the empty cosieve on $M$.
  Then we are immediately done with the current branch.)
\end{remark}

Here is a first illustration of this strategy.

\begin{example}
  \label{example-surjectivity-rigid}
  Start with the theory $\TT$
  with two sorts
  and a function symbol $A \to B$,
  and add to it the axiom
  \[ \top \turnstile{y \oftype B} \ex{x}{A} f(x) = y \rlap{,} \]
  requiring $f$ to be surjective.
  We already know that the resulting theory is of presheaf type,
  since it is Morita equivalent to the theory
  with just one sort $A$
  and an equivalence relation ${\sim} \subseteq A \times A$
  by Example \ref{example-materialize-quotient},
  and this is a Horn theory.
  But let us show that the surjectivity axiom
  is a rigid-topology quotient of $\TT$.

  By Lemma \ref{lemma-compact-Horn-models},
  a model $M = (A \to B) \in \TT\dashmod(\Set)$
  is compact if and only if
  both $A$ and $B$ are finite sets.
  Given such a compact model $(A \to B)$
  and an element $y \in B$,
  the cosieve $S_y$ on $(A \to B)$
  from Theorem \ref{theorem-induced-topology}
  is generated by the single arrow
  \[ \begin{tikzcd}[row sep=small]
    (A \to B) \ar[d] \\
    (A \sqcup \{x\} \to B) \rlap{,}
  \end{tikzcd} \]
  where $x$ is sent to $y$ in the new model
  (which is clearly still compact),
  since every model homomorphism
  \[ \begin{tikzcd}[column sep=small]
    A \ar[r, "f"] \ar[d, "g_A"] & B \ar[d, "g_B"] \\
    A' \ar[r, "f'"] & B'
  \end{tikzcd} \]
  with $\exists x \in A'.\; f'(x) = g_B(y)$
  factors through $(A \sqcup \{x\} \to B)$.
  The model $(A \to B)$ can thus be covered
  (in the topology induced by the surjectivity axiom)
  by the model $(A \sqcup \{x\} \to B)$.
  We can go on by creating preimages of other elements of $B$,
  until we have covered $(A \to B)$ by a model
  satisfying the new axiom
  (since $B$ is finite).
\end{example}

\begin{lemma}
  \label{lemma-compose-rigid-topology-quotients}
  If $\TT$ is of presheaf type,
  $\QQ_1$ is a rigid-topology quotient of $\TT$
  and $\QQ_2$ is a rigid-topology quotient of $\TT + \QQ_1$,
  then $\QQ_1 + \QQ_2$ is a rigid-topology quotient of $\TT$.
\end{lemma}

\begin{proof}
  This follows easily from
  Lemma \ref{lemma-rigid-topology-quotients-compact-models}:
  $\TT + \QQ_1 + \QQ_2$ is of presheaf type
  and every compact $(\TT + \QQ_1 + \QQ_2)$-model
  is compact as a $(\TT + \QQ_1)$-model
  and hence as a $\TT$-model.

  Note that arguing more directly
  by the transitivity of $J_{\QQ_1 + \QQ_2}$
  is not so simple,
  as we would still have to relate the topology induced by $\QQ_2$
  on $\mod{(\TT + \QQ_1)}{\Set}_c$
  with that on $\mod{\TT}{\Set}_c$.
\end{proof}

While we can assume,
after choosing a representing set of axioms
for the quotient extension $\QQ_2$,
that its base theory is $\TT$ instead of $\TT + \QQ_1$,
the assumptions of Lemma \ref{lemma-compose-rigid-topology-quotients}
are not at all the same as
requiring $\QQ_1$ and $\QQ_2$
to be rigid-topology quotients of $\TT$ independently.
In fact, $\QQ_1 + \QQ_2$ can fail
to be a rigid-topology quotient of $\TT$
in this case,
as we will see in Example
\ref{example-sum-of-two-rigid-topology-quotients}.
By the same method as there,
one can easily construct quotient extensions $\QQ_1, \dots, \QQ_n$
such that the sum of any proper subset
induces a rigid topology,
but adding all $n$ even destroys presheaf type.

\subsection{Syntactic constructions preserving presheaf type}

We now investigate
in which ways a theory of presheaf type can be extended
such that the resulting theory is again of presheaf type.
At first sight,
one might hope for a result like
\enquote{an algebraic (or Horn, or even cartesian) extension
preserves presheaf type},
but unfortunately,
this is far too optimistic.
Indeed,
if we have any geometric axiom
$\phi \turnstile{\vec{x} : \vec{A}} \psi$
we would like to add to a theory,
we can instead first perform an extension by definitions
(which surely preserves presheaf type),
defining $R_\phi, R_\psi \subseteq \vec{A}$
by $\phi$ and $\psi$ respectively,
and then add the Horn axiom
$R_\phi \turnstile{\vec{x} : \vec{A}} R_\psi$.
Example \ref{example-single-algebraic-axiom} shows that
also an algebraic axiom can destroy presheaf type.

\begin{corollary}[see {\cite[Theorem 8.2.8]{caramello:tst}}]
  \label{corollary-negated-axioms}
  Adding arbitrarily many axioms of the form
  \[ \phi \turnstile{\vec{x} : \vec{A}} \bot \]
  to a theory of presheaf type
  is always a rigid-topology quotient.
\end{corollary}

\begin{proof}
  Let $\TT$ be a theory of presheaf type,
  $\QQ = \{\, \phi^i \turnstile{\vec{x}^i : \vec{A}^i} \bot
  \mid i \in I \,\}$
  and $M \in \mod{\TT}{\Set}_c$.
  If there is some $i \in I$
  and $\vec{x} \in \interpretation{\phi^i}_M$,
  then
  \[ S_{\vec{x}} = \{\, f : M \to M'
    \mid f(\vec{x}) \in \interpretation{\bot}_{M'} = \varnothing \,\} \]
  is the empty cosieve on $M$
  and we are done.
  But otherwise
  $M$ was already a model of $\TT + \QQ$,
  i.e. irreducible itself.
\end{proof}

\begin{corollary}
  \label{corollary-axiom-in-empty-context}
  Adding finitely many axioms
  $\phi \turnstile{[]} \psi$
  with empty context
  to a theory of presheaf type
  is always a rigid-topology quotient.
\end{corollary}

\begin{proof}
  For one axiom $\phi \turnstile{[]} \psi$ with empty context,
  Theorem \ref{theorem-induced-topology}
  says that any compact model $M$
  that satisfies the closed formula $\phi$
  is covered by all arrows
  (necessarily preserving truth of $\phi$)
  to compact models that also satisfy $\psi$,
  as required.
  And any model that does not satisfy $\phi$
  satisfies the new axiom anyway.
  For more than one axiom of this form,
  we can use Lemma \ref{lemma-compose-rigid-topology-quotients}.
\end{proof}

\begin{corollary}
  Any propositional theory
  with finitely many axioms
  is of presheaf type.
\end{corollary}

\begin{proof}
  This immediately follows from
  Corollary \ref{corollary-axiom-in-empty-context},
  as a propositional theory
  has no other contexts than the empty one.
\end{proof}

\begin{remark}
  \label{remark-infinitely-many-axioms-in-empty-context}
  Corollary \ref{corollary-axiom-in-empty-context}
  becomes false if we omit the word \enquote{finitely}.
  Indeed, it would otherwise imply that
  all propositional theories are of presheaf type
  (see Example \ref{example-theories-of-a-function}
  for one which is not).
  In particular,
  Lemma \ref{lemma-compose-rigid-topology-quotients}
  becomes wrong if instead of two consecutive rigid-topology quotients
  we consider the sum $\TT + \sum_{i \in \NN} \QQ_i$
  of an infinite sequence of quotient extensions
  where every $\QQ_n$ is a rigid-topology quotient
  of $\TT + \sum_{i = 0}^{n - 1} \QQ_i$.
\end{remark}

\begin{proposition}
  \label{proposition-add-constant-symbol}
  Let $\TT$ be a theory of presheaf type
  and $A$ a sort of $\TT$.
  Then the theory $\TT + (c : A)$
  obtained from $\TT$ by introducing a new constant symbol $c : A$
  is also of presheaf type.
  Moreover,
  a model of $\TT + (c : A)$ in $\Set$
  is compact if and only if
  the underlying $\TT$-model is compact.
\end{proposition}

\begin{proof}
  Let $\Set[\TT]$ be a classifying topos for $\TT$
  with universal model $M_\TT$.
  Then $\TT + (c : A)$ is classified by
  the slice topos $\Set[\TT] / \interpretation{A}_{M_\TT}$,
  see \cite[Exercise X.3]{maclane-moerdijk}.
  But if $\Set[\TT] = \PSh(C)$ is a presheaf topos,
  then we can use the formula
  \cite[Exercise III.8]{maclane-moerdijk}
  \[ \PSh(C) / X \simeq \PSh(\int_C X) \]
  for $X = \interpretation{A}_{M_\TT}$
  to see that $\TT + (c : A)$ is of presheaf type.

  For the compact models,
  we observe that $(\TT + (c : A))\dashmod(\Set)$ is equivalent,
  as a category over $\TT\dashmod(\Set)$,
  to the category
  $\int^{\TT\dashmod(\Set)} \interpretation{A}_\hole$
  of elements of the interpretation functor
  $\interpretation{A}_\hole : \TT\dashmod(\Set) \to \Set$.
  This functor preserves filtered colimits
  by Lemma \ref{lemma-filtered-colimits-of-models},
  so Lemma \ref{lemma-compacts-in-category-of-elements} below
  tells us exactly what we need.
\end{proof}

\begin{lemma}
  \label{lemma-compacts-in-category-of-elements}
  Let $C$ be a category with filtered colimits
  and let $F : C \to \Set$ be a functor preserving filtered colimits.
  Then the category of elements of $F$, $\int^C F$,
  has filtered colimits too,
  and the projection functor $U : \int^C F \to C$
  preserves and reflects compact objects.
\end{lemma}

\begin{proof}
  Let $(c_i, x_i)_{i \in I}$,
  where $x_i \in F(c_i)$,
  be a filtered diagram in $\int^C F$.
  It is clear that the $x_i$ all represent the same element
  in $\colim_i F(c_i) \cong F(\colim_i c_i)$,
  let $x^* \in F(\colim_i c_i)$ be this element.
  Then we have a cocone on $(c_i, x_i)_{i \in I}$
  with apex $(\colim_i c_i, x^*)$,
  and it is easy to check that this is the colimit in $\int^C F$.
  This also shows that
  the functor $U : \int^C F \to C$
  preserves and reflects filtered colimits.

  Let $(c, x)$ be a compact object of $\int^C F$,
  we want to show that $c$ is compact in $C$.
  So let $(c_i)_{i \in I}$ be a filtered diagram in $C$
  and fix a morphism
  \[ f : c \to \colim_{i \in I} c_i \rlap{,} \]
  then we have to show that there is exactly one
  $[g] \in \colim_i \Hom(c, c_i)$ that induces $f$.
  If we want to use the compactness of $(c, x)$,
  we need to construct a filtered diagram in $\int^C F$.
  Consider the element $f_* x \in F(\colim_i c_i)
  \cong \colim_i F(c_i)$
  and pick a representative $x_{i_0} \in F(c_{i_o})$ for it.
  Since $I$ is filtered,
  the comma category $(i_0 \downarrow I)$ is also filtered
  and the functor $(i_0 \downarrow I) \to I$ is final.
  This means that we have
  $\colim_i c_i = \colim_{i_0 \to i} c_i$
  and $\colim_i \Hom(c, c_i) = \colim_{i_0 \to i} \Hom(c, c_i)$,
  so we can replace the diagram $(c_i)_{i \in I}$
  with $(c_i)_{(i_0 \to i) \in (i_0 \downarrow I)}$
  and assume without loss of generality
  that $i_0$ is initial in $I$.
  Then setting $x_i \defeq ({!} : i_0 \to i)_* x_{i_0} \in F(c_i)$
  yields a filtered diagram
  $(c_i, x_i)_{i \in I}$ in $\int^C F$.
  We also have
  \[ \widetilde{f} : (c, x) \to \colim_{i \in I} {(c_i, x_i)} \]
  with $U(\widetilde{f}) = f$
  by our original choice of $x_{i_0}$.
  So by assumption,
  there is a $\widetilde{g} : (c, x) \to (c_i, x_i)$
  for some $i$
  that induces $\widetilde{f}$.
  Then $g \defeq U(\widetilde{g})$ induces $f$.
  On the other hand,
  if $g_1 : c \to c_{i_1}$ and $g_2 : c \to c_{i_2}$
  both induce $f$,
  then we have $[{g_1}_* x] = [{g_2}_* x] = [x_{i_0}] \in \colim_i F(c_i)$,
  in particular $[{g_1}_* x] = [x_{i_1}]$ and $[{g_2}_* x] = [x_{i_2}]$,
  so we can prolong $g_1$
  to some $g'_1 : c \to c_{i'_1}$
  such that ${g'_1}_* x = x_{i'_1}$,
  and similarly for $g_2$,
  meaning that we have
  $\widetilde{g}_1 : (c, x) \to (c_{i'_1})$ and
  $\widetilde{g}_2 : (c, x) \to (c_{i'_2})$
  that both induce $\widetilde{f}$.
  Then we can conclude
  $[\widetilde{g}_1] = [\widetilde{g}_1]
  \in \colim_i \Hom((c, x), (c_i, x_i))$,
  and from this it follows that
  $[g_1] = [g'_1] = [g'_2] = [g_2]
  \in \colim_i \Hom(c, c_i)$.

  Now let $c \in C$ be compact
  and let $x \in F(c)$,
  we want to show that $(c, x)$ is compact in $\int^C F$.
  So let $(c_i, x_i)_{i \in I}$ be a filtered diagram in $\int^C F$
  and let
  \[ \widetilde{f} : (c, x) \to \colim_{i \in I} (c_i, x_i)
     = (\colim_{i \in I} c_i, x^*) \]
  be given.
  That is, we have $f : c \to \colim_i c_i$
  with $f_* x = x^*$.
  This $f$ is induced by some $g : c \to c_i$,
  implying that $[g_* x] = x^* \in \colim_i F(c_i)$,
  so it is also induced by some
  $g' : c \to c_{i'}$ with $g'_* x = x_{i'}$.
  This means that we have
  $\widetilde{g} : (c, x) \to (c_{i'}, x_{i'})$
  inducing $\widetilde{f}$.
  And for $\widetilde{g}_1$, $\widetilde{g}_2$
  both inducing $\widetilde{f}$,
  we immediately see that
  $[U(\widetilde{g}_1)] = [U(\widetilde{g}_2)]
  \in \colim_i \Hom(c, c_i)$
  and therefore
  $[\widetilde{g}_1] = [\widetilde{g}_2]
  \in \colim_i \Hom((c, x), (c_i, x_i))$.
\end{proof}

\begin{example}
  \label{example-add-constant-to-imported-set}
  From Example \ref{example-import-a-set}
  we know that the theory $\underline{A}$,
  where we have just imported a set $A$,
  is Morita equivalent to the empty theory
  (over the empty signature).
  In particular, it is of presheaf type
  and we can apply Proposition \ref{proposition-add-constant-symbol}
  to conclude that the theory $c : \underline{A}$,
  where we have added one more constant symbol
  is also of presheaf type.
  It is classified by $\Set / A$.

  It is also Morita equivalent to a disjoint disjunction
  of an $A$-indexed family of proposition symbols,
  \[ \top \turnstile{[]} \bigvee_{a \in A} p_a
     \qquad\text{and}\qquad
     p_a \land p_{a'} \turnstile{[]} \bot
     \quad (\text{for $a \neq a' \in A$})
     \rlap{,} \]
  by identifying $p_a$ with $c = c_a$.
  This way,
  we can easily see that the theory is of presheaf type
  using Corollary \ref{corollary-axiom-in-empty-context}
  for the first axiom
  and Corollary \ref{corollary-negated-axioms}
  for the $A$-indexed family of negated axioms.
\end{example}

\begin{example}
  \label{example-add-constant-to-quotient}
  Consider the theory of a surjective function
  $f : A \twoheadrightarrow B$,
  which is of presheaf type by
  Example \ref{example-materialize-quotient}
  or Example \ref{example-surjectivity-rigid}.
  By Proposition \ref{proposition-add-constant-symbol},
  it is still of presheaf type
  after adding a constant symbol $c : B$.

  When we translate this extension
  to the theory of an equivalence relation
  ${\sim} \subseteq A \times A$,
  we obtain a relation symbol $R \subseteq A$ with the axioms
  \[ R(x) \land x \sim x' \turnstile{x, x' \oftype A} R(x')
     \rlap{,} \qquad
     R(x) \land R(x') \turnstile{x, x' \oftype A} x \sim x'
     \rlap{,} \qquad
     \top \turnstile{[]} \ex{x}{A} R(x)
     \rlap{.} \]
  The existential quantifier in the last axiom
  is not cartesian,
  but the context of the axiom is empty
  (and all other axioms of the theory are Horn),
  so we again find that the theory is of presheaf type
  using Corollary \ref{corollary-axiom-in-empty-context}.
  However,
  if the sort $A$ and the equivalence relation $\sim$
  are instead given as part of a theory of presheaf type
  which is not cartesian,
  then the Horn axioms on $R$ can not be argued away so easily,
  but we can still introduce the quotient $B = A/{\sim}$
  and add a constant symbol $c : B$
  without destroying presheaf type.
\end{example}

\subsection{Some counterexamples}

We now give a number of examples
of non-presheaf-type theories,
to show that the above positive results
are optimal in certain aspects
and to illustrate the various techniques
useful for recognizing such theories.
Our first example shows that
even a very simple, finite theory
can fail to be of presheaf type.

\begin{example}
  Consider the theory
  \[ \TT = \left\{ \begin{tikzcd}
    A \ar[loop right, two heads, "f"]
  \end{tikzcd} \right\} \]
  consisting of one sort $A$,
  a unary function symbol $f : A \to A$
  and one axiom
  $\top \turnstile{y : A} \ex{x}{A} f(x) = y$
  stating that $f$ is surjective.
  We show that $\mod{\TT}{\Set}$ is not finitely accessible,
  implying that $\TT$ is not of presheaf type.

  Specifically,
  we show that the model $(\ZZ, +1)$,
  where $A = \ZZ$ and $f(n) = n+1$,
  is not a filtered colimit of compact models.
  Consider the models
  \[ M_n = \Big(\,
     \{\, (x, y) \in \ZZ^2 \mid y \leq \min(n, -x) \,\}
     \,,\,
     (x, y) \mapsto (x+1, \min(y, -x-1))
     \,\Big)
     \rlap{,} \]
  which can be pictured like this.
  \[ \begin{tikzcd}[row sep=tiny, column sep=small]
    & \bullet \ar[r, mapsto] & \bullet \ar[r, mapsto]
    & \bullet \ar[r, mapsto] & \bullet \ar[dr, mapsto] \\
    \mathclap{\dots}
    & \bullet \ar[r, mapsto] & \bullet \ar[r, mapsto]
    & \bullet \ar[r, mapsto] & \bullet \ar[r, mapsto]
    & \bullet \ar[rd, mapsto] \\
    & \bullet \ar[r, mapsto] & \bullet \ar[r, mapsto]
    & \bullet \ar[r, mapsto] & \bullet \ar[r, mapsto]
    & \bullet \ar[r, mapsto] & \bullet \\
    & & & & \mathclap{\dots} & & & \mathclap{\smash{\ddots}}
  \end{tikzcd} \]
  The inclusions $M_n \hookrightarrow M_{n+1}$
  (adding another row on top)
  are model homomorphisms,
  and we have a homomorphism
  \[ (\ZZ, +1) \to M_\omega \defeq \colim_n M_n ,\quad
     k \mapsto (k, -k)
     \rlap{,} \]
  that does not factor through any $M_n$.
  This shows that $(\ZZ, +1)$ is not compact.
  But any map $M \to (\ZZ, +1)$ from a nonempty model $M$
  is necessarily surjective,
  so the composite $M \to (\ZZ, +1) \to M_\omega$
  does not factor through any $M_n$ either,
  showing that $(\ZZ, +1)$ is not a colimit of compact models.
\end{example}

Another method
to show that a theory has \enquote{too few compact models}
to be of presheaf type
is illustrated
by the next example.
In it,
we see that adding countably many constant symbols of the same sort
can destroy presheaf type
(in contrast to Proposition \ref{proposition-add-constant-symbol}).
An even more drastic version of this statement
will be one of the conclusions of
Example \ref{example-theories-of-a-function}.

\begin{example}
  \label{example-add-many-constants-to-quotient}
  Start with the theory of a surjective function
  $f : A \twoheadrightarrow B$
  as in Example \ref{example-add-constant-to-quotient},
  but then add countably many constant symbols
  $c_0, c_1, \ldots : B$
  instead of just a single one.
  We show that this theory $\TT$ is not of presheaf type
  by investigating its compact models.
  \[ A \twoheadrightarrow B \ni c_0, c_1, \dots \]

  Let $\TT'$ be the theory of inhabited sets,
  with one sort $A$ and the axiom
  $\top \turnstile{[]} (\ex{x}{A} \top)$.
  Then we can regard $\TT$ as an extension of $\TT'$
  and the forgetful functor
  \[ U : \TT\dashmod(\Set) \to \TT'\dashmod(\Set)
     , \quad
     (A \twoheadrightarrow B) \mapsto A
     \]
  has a right adjoint
  $A' \mapsto (A' \twoheadrightarrow \{*\})$,
  which preserves filtered colimits
  by their description in
  Lemma \ref{lemma-filtered-colimits-of-models}.
  Thus, $U$ preserves compact objects by
  Lemma \ref{lemma-adjoint-functor-preserves-compacts}.
  But the compact objects of $\TT'\dashmod(\Set)$
  are the finite nonempty sets,
  so, in summary,
  for any compact model $(A \twoheadrightarrow B)$ of $\TT$,
  $A$ is a finite set.
  This means that $B$ is finite too,
  from which we can in particular conclude that
  the geometric formula
  \[ \phi \defeq \quad
     \bigvee_{i \neq j} (c_i = c_j) \]
  is satisfied in all compact models.
  If $\TT$ was of presheaf type,
  this would mean that $\phi$ is satisfied
  in the universal model
  and therefore provable in $\TT$.
  This is not the case,
  as witnessed by the model $(\NN \twoheadrightarrow \NN)$
  with $c_n = n$.
\end{example}

The following example will allow us
to draw several conclusions.
The technique used here is
to compare a theory to
another theory with the same $\Set$-based models
which is known to be of presheaf type.

\begin{example}
  \label{example-theories-of-a-function}
  Given two sets $A$ and $B$,
  there are (at least) two different theories
  \enquote{of a map from $A$ to $B$}.
  One is
  \[ \angles{f : \underline{B^A}} \rlap{,} \]
  where we import
  (see Example \ref{example-import-a-set})
  the set $B^A$ of functions from $A$ to $B$
  and add a constant symbol $f$ of this sort;
  it is of presheaf type
  as in Example \ref{example-add-constant-to-imported-set}.
  The other is
  \[ \angles{f : \underline{A} \to \underline{B}} \rlap{,} \]
  where we import $A$ and $B$ separately
  and add a unary function symbol as indicated.
  For both theories,
  the category of models in $\Set$
  is (up to equivalence)
  the discrete category with $B^A$ objects.
  So if $\angles{f : \underline{A} \to \underline{B}}$
  was of presheaf type too,
  then this would imply that the two theories are Morita-equivalent.

  Specialize to the case $A = \NN$ and $B = 2 = \{0, 1\}$.
  Then $\angles{f : \underline{2^\NN}}$ is classified by $\Set/2^\NN$,
  or equivalently, by $\Sh(X)$,
  where $X$ is the discrete space with $2^\NN$ points.
  But by \cite[Exercise VIII.10]{maclane-moerdijk},
  $\angles{f : \underline{\NN} \to \underline{2}}$
  is classified by $\Sh(2^\NN)$,
  where $2^\NN$ is the Cantor space.
  Since these two sober topological spaces are not homeomorphic,
  $\angles{f : \underline{\NN} \to \underline{2}}$
  is not of presheaf type.

  Another way to see that $\angles{f : \underline{2^\NN}}$
  and $\angles{f : \underline{\NN} \to \underline{2}}$
  are not Morita equivalent
  is as follows.
  First, as in Example \ref{example-add-constant-to-imported-set},
  a constant symbol in an imported set,
  such as $f : \underline{2^\NN}$,
  can be replaced by a disjoint disjunction of $2^\NN$
  proposition symbols.
  On the other hand,
  the function symbol $f : \underline{\NN} \to \underline{2}$
  can be replaced by
  countably many constant symbols $c_n : \underline{2}$
  and therefore by countably many decidable propositions
  (disjoint binary disjunctions).
  In this way,
  one sees that the models of both theories
  in any topos $\E$ form discrete categories,
  namely,
  $\angles{f : \underline{2^\NN}}$ classifies
  decompositions of $\E$ into $2^\NN$ clopen subtoposes,
  while $\angles{f : \underline{\NN} \to \underline{2}}$
  classifies countable families of compositions
  into two clopen subtoposes.
  Now,
  we can consider the topological spaces
  $X_n \defeq \{0, \dots, n\}$
  (with the discrete topology)
  and $X_\omega \defeq \NN \cup \{\infty\}$,
  the one-point compactification of $\NN$,
  and the continuous maps
  \[ X_\omega \to X_n,\quad x \mapsto \min(x, n) \rlap{.} \]
  Then we see that every model of $\angles{f : \underline{2^\NN}}$
  in $X_\omega$
  is (isomorphic to)
  the pullback of a model in some $X_n$,
  because any open set of $X_\omega$ containing $\infty$ is cofinite,
  but this is not the case for
  $\angles{f : \underline{\NN} \to \underline{2}}$,
  because we can choose a family of partitions like
  $X_\omega = \{0, \dots, n\} \sqcup \{n+1, \dots, \infty\}$.

  In summary,
  we can record the following.
  \begin{enumerate}[label=(\roman*)]
    \item
      Adding a single unary function symbol
      (such as $f : \underline{\NN} \to \underline{2}$)
      to a theory of presheaf type
      (even a theory that is classified by $\Set$)
      can destroy presheaf type.

    \item
      Adding countably many constant symbols
      (such as $c_n : \underline{2}$)
      to a theory of presheaf type
      (even a theory that is classified by $\Set$)
      can destroy presheaf type.

    \item
      A propositional theory
      (such as the theory of
      a countable family of decidable propositions,
      with $p_n \lor q_n$ and $p_n \land q_n \turnstile{[]} \bot$)
      can fail to be of presheaf type.
      (There are of course many other examples of this.)

    \item
      Adding countably many closed geometric formulas
      (such as $p_n \lor q_n$ in the previous item,
      where all other axioms were negated ones,
      which can't destroy presheaf type)
      can destroy presheaf type.
  \end{enumerate}
\end{example}

\begin{example}
  \label{example-single-algebraic-axiom}
  Let us show that adding an algebraic axiom
  can destroy presheaf type.
  We do this by showing that
  adding an algebraic axiom is,
  up to Morita equivalence,
  at least as expressive as
  adding an arbitrary family
  $\phi_i$
  of closed geometric formulas
  as axioms.
  We can then for example use $\phi_n = (p_n \lor q_n)$
  to obtain the theory of countably many decidable propositions
  from a theory with only negated axioms,
  as in example \ref{example-theories-of-a-function}.

  So let $\TT$ be a theory
  with closed geometric formulas $\phi_i$, $i \in I$.
  Import the set $I \sqcup \{*\}$.
  We can define (by a geometric formula)
  an equivalence relation $\sim$
  on $\underline{I \sqcup \{*\}}$
  such that $c_i \sim c_* \doubleturnstile{[]} \phi_i$.
  Introduce the quotient
  $\underline{I \sqcup \{*\}} \twoheadrightarrow A$
  by $\sim$.
  These equivalence extensions allow us to formulate the algebraic axiom
  \[ \top \turnstile{x, y : A} x = y \rlap{,} \]
  which is syntactically equivalent (as a quotient extension) to
  $\top \turnstile{x : \underline{I \sqcup \{*\}}} x \sim c_*$,
  and therefore to the set of axioms
  $\{\, \top \turnstile{[]} \phi_i \mid i \in I \,\}$.
\end{example}

\begin{example}
  \label{example-sum-of-two-rigid-topology-quotients}
  Here we give an example of
  two rigid-topology quotients $\QQ_1$, $\QQ_2$
  of the same theory of presheaf type $\TT$
  such that $\QQ_1 + \QQ_2$ is not a rigid-topology quotient again
  and $\TT + \QQ_1 + \QQ_2$ is not even of presheaf type.
  Start with the theory $\TT$
  of a sequence of maps,
  infinite to the left:
  \[ \dots \xrightarrow{f_2} A_2 \xrightarrow{f_1} A_1
     \xrightarrow{f_0} A_0 . \]
  The compact $\Set$-models of $\TT$
  are those with finitely many elements in total.
  Consider the quotient extension $\QQ_1$
  requiring every second map
  to be surjective:
  \[ \top \turnstile{y : A_n} \ex{x}{A_{n+1}} f_n(x) = y
  \qquad\qquad \text{for $n$ even} . \]
  The topology on ${\mod{\TT}{\Set}_c}^\op$ induced by these axioms
  is rigid:
  Given $M = (\dots \to A_1 \to A_0) \in \mod{\TT}{\Set}_c$
  and an element $y \in A_n$ not in the image of $f_n$,
  the cosieve $S_y$
  is generated by the single model homomorphism
  \[ \begin{tikzcd}[row sep=small, column sep=small]
    \dots \ar[r] & A_{n+2} \ar[r] \ar[d] & A_{n+1} \ar[r] \ar[d] &
    A_n \ar[r] \ar[d]& \dots \ar[r] & A_0 \ar[d]
    \\
    \dots \ar[r] & A_{n+2} \ar[r] & A_{n+1} \sqcup \{x\} \ar[r]
    & A_n \ar[r] & \dots \ar[r] & A_0
    \rlap{,}
  \end{tikzcd} \]
  where we have freely added a preimage $x$ of $y$.
  After finitely many steps,
  we have covered $M$
  by a compact $\TT$-model that satisfies $\QQ_1$.

  However,
  denoting by $\QQ_2$ the quotient extension
  requiring $f_n$ to be surjective for $n$ odd,
  which is a rigid-topology quotient for the same reason,
  the theory
  \[ \TT' \defeq \TT + \QQ_1 + \QQ_2 \]
  is \emph{not} of presheaf type.
  To see this,
  we show that $\mod{\TT'}{\Set}$
  is not finitely accessible.
  In fact,
  we can show that
  the only compact object of $\mod{\TT'}{\Set}$
  is the empty model $(\dots \to \varnothing \to \varnothing)$,
  by considering the following sequence of inclusions.
  \[ \begin{tikzcd}[row sep=small, column sep=small]
    \dots \ar[r] & \{0\} \ar[r] & \{0\} \ar[r] &
    \{0\} \ar[r] & \{0\}
    &=& M^0 \ar[d, hook]
    \\
    \dots \ar[r] & \{0, 1\} \ar[r] & \{0, 1\} \ar[r] &
    \{0, 1\} \ar[r, "1 \mapsto 0"] & \{0\}
    &=& M^1 \ar[d, hook]
    \\
    \dots \ar[r] & \{0, 1, 2\} \ar[r] &
    \{0, 1, 2\} \ar[r, "2 \mapsto 1"] &
    \{0, 1\} \ar[r, "1 \mapsto 0"] & \{0\}
    &=& M^2 \ar[d, hook]
    \\
    &&\vdots&&&& \vdots
  \end{tikzcd} \]
  The colimit of this diagram
  is a model $M^\omega$ with $A^\omega_n = \{0, \dots, n\}$.
  If $M = (\dots \to A_1 \to A_0)$ is any nonempty model,
  then $A_n$ is nonempty for all $n$,
  so the model homomorphism $M \to M^\omega$
  sending any $x \in A_n$
  to $n \in A^\omega_n$
  does not factor through any of the models in the diagram,
  showing that $M$ is not compact.
\end{example}

Finally,
it seems appropriate to show that
the theory of local rings,
which played an important role in
Subsection \ref{subsection-gluing-Zar}
and will do so again in
Section \ref{section-crystalline},
is not of presheaf type.
Equivalently,
this shows that the big Zariski topos $(\Spec \ZZ)_\Zarfp$
is not equivalent to a presheaf topos.

\begin{proposition}
  The theory of local rings,
  $\Ring + (\loc)$
  (see Definition \ref{definition-theories-for-Zar}),
  is not of presheaf type.
\end{proposition}

\begin{proof}
  The argument given here resembles
  the one given in \cite[Section 9.4]{caramello:tst}
  showing that the usual geometric theory of fields
  is not of presheaf type.
  Assume that the theory of local rings
  is of presheaf type,
  that is,
  it is classified by $[\mathrm{locRing}_c, \Set]$
  with the tautological local ring as universal model.
  We aim to reveal a subterminal object of $[\mathrm{locRing}_c, \Set]$
  which is not definable by a closed geometric formula.
  Note that every local ring $A$
  is a $\ZZ_{(p)}$-algebra for some prime number $p$,
  since at most one prime number can be non-invertible
  as an element of $A$.
  This means that the join of
  the subterminal objects
  \[ U_p \defeq \Hom(\ZZ_{(p)}, \hole) :
     \mathrm{locRing}_c \to \Set \]
  is $\bigvee_p U_p = 1$.
  Now, if there were geometric formulas $\phi_p$
  such that $\interpretation{\phi_p} = U_p$
  (i.e. $\phi_p$ expresses that
  all prime numbers except $p$ are invertible)
  for all $p$,
  then the disjunction $\bigvee_p \phi_p$,
  being valid in the universal model,
  would be provable in the theory of local rings.
  But since this is a coherent theory,
  a finite disjunction $\bigvee_{p \in P} \phi_p$
  with $P$ a finite set of primes
  would also have to be provable
  by \cite[Theorem 10.8.6 (iii)]{caramello:tst}.
  This is not the case,
  as $\bigvee_{p \in P} \phi_p$
  is falsified by the local ring $\ZZ_{(p')}$
  for any $p' \notin P$.
\end{proof}

\newpage

\section{Syntactic presentations for crystalline toposes}
\label{section-crystalline}

\subsection{Introduction}

In this section,
we determine a syntactic presentation
of the big crystalline toposes
used for studying crystalline cohomology in algebraic geometry.
This fulfills a promise made by Wraith
in 1979
in \cite[p.\ 743]{wraith},
where he writes:
\enquote{%
It is my belief that
a great many of the toposes occurring in algebraic geometry
can be conveniently described in terms of the theories they classify.
This is certainly so in the case of étale and crystalline toposes,
for example.}
This optimistic statement is however not accompanied by any hint
of what the classified theory might be,
and no answer has been given since.

The reason why the question is not as simple
as for the big Zariski topos,
for which the answer has been known
since the infancy of the field of classifying toposes,
is that much more data is involved in its construction.
To wit,
the crystalline topos depends on two schemes,
a scheme morphism between them,
and a certain structure known as a divided powers structure,
or PD structure on one of them.

In the affine case,
a PD structure on a ring $A$
of positive characteristic $p$
assigns to an element $a \in A$
another element of $A$,
acting as a replacement for the expression
\[ \frac{1}{n} a^n \rlap{,} \]
which can never be taken literally
when $p$ divides $n$.
The intuition why such a structure might be useful
for cohomological techniques
is that
in positive characteristic $p$,
taking the derivative of a polynomial,
or a power series,
\[ t^n \mapsto n t^{n - 1} \rlap{,} \]
is never a surjective map,
as there is no integral of $t^{p-1}$,
and this leads to a failure of the Poincaré Lemma.
But if divided powers of $t$ are available,
then the integral does exist.

For us, this means that the syntactic presentation
of the crystalline topos
will involve two rings and a geometric theory formulation
of a PD structure.
To handle this somewhat complex theory,
we make heavy use of the techniques around theories of presheaf type
developed in Section \ref{section-presheaf-type}.

\subsection{Background on divided power rings}

Let us collect a number of
properties and constructions of divided power rings
that will be relevant.
A \emph{PD structure}
or \emph{divided powers structure}
on an ideal $I$ of a ring $A$
(always commutative and unitary)
is a family of functions
\[ \gamma_n : I \to I, \quad n \geq 1 \]
such that $x^n = n! \gamma_n(x)$
for every $x \in I$,
compatible with natural operations on the ideal $I$
and with each other,
as specified in
Definition \ref{definition-theories-with-ideals} below.
For example,
the zero ideal $(0)$ of any ring
carries the trivial PD structure $\gamma_n(0) = 0$,
while any ideal of a $\QQ$-algebra $A$
admits a unique PD structure given by
$\gamma_n(a) = \frac{a^n}{n!}$.
Slightly more interestingly,
the maximal ideal $(p)$ of $\ZZ_{(p)}$
admits a unique PD structure,
as the prime factor $p$ occurs
at most $n - 1$ times in $n!$,
so $\frac{p^n}{n!} \in (p) \subseteq \ZZ_{(p)}$
for all $n \geq 1$.

An ideal equipped with a PD structure is called a \emph{PD ideal}
and a ring equipped with a PD ideal
(that is,
a model in $\Set$ of the theory $\Ring + \Ideal + \PD$,
see Definition \ref{definition-theories-with-ideals} below)
is a \emph{PD ring}.
This also defines homomorphisms of PD rings.
We sometimes write a PD ring $(A, I, \gamma)$
as $(\gamma \curvearrowright I \triangleleft A)$.

Not every ideal $J \subseteq I$
contained in the PD ideal of a PD ring $(A, I, \gamma)$
is closed under the $\gamma_n$.
But we can always close it up in one step,
that is,
the ideal generated by $\gamma_n(a)$
for $n \geq 1$ and $a \in J$
is a sub-PD-ideal of $I$,
and it actually suffices to take $\gamma_n(a)$
for $a$ in any generating set of the ideal $J$.
More generally,
for any ideal $J$ of $A$,
we let $\overline{J}$ be the ideal generated by $J$
and $\{\, \gamma_n(a) \mid n \geq 1, a \in J \cap I \,\}$
and call it the \emph{PD saturation} of the ideal $J$
(in the PD ring $(A, I, \gamma)$).
There is an induced PD structure
on the ideal $I / (J \cap I)$
of the ring $A/J$
if and only if
$J$ is PD saturated ($J = \overline{J}$)
\citestacks{07H2},
and this constitutes a bijective correspondence
between PD saturated ideals of $A$
and isomorphism classes of
maps of PD rings $(A, I, \gamma) \to (A', I', \gamma')$
such that both $A \to A'$ and $I \to I'$ are surjective.

For a PD ring $(A, I, \gamma)$
and an element $a \in A$,
there is a unique PD structure $\gamma_a$
on the ideal $I_a$ of $A_a = A[a^{-1}]$
such that $A \to A_a$ is a homomorphism of PD rings,
namely $\gamma_{a, n}(\frac{a'}{a^k}) = \frac{\gamma_n(a')}{a^{n k}}$.
The new PD ring
has the expected universal property:
a homomorphism
$(\gamma_a \curvearrowright I_a \triangleleft A_a) \to
(\gamma' \curvearrowright I' \triangleleft A')$
is the same as one
$(\gamma \curvearrowright I \triangleleft A) \to
(\gamma' \curvearrowright I' \triangleleft A')$
that makes $a$ invertible in $A'$.

The \emph{PD envelope} construction
\citestacks{07H8}
is a left adjoint functor for the forgetful functor
\[ (\Ring + \Ideal+ \PD)\dashmod(\Set) \to
   (\Ring + \Ideal)\dashmod(\Set)
   \rlap{,} \]
and more generally,
for a fixed PD ring
$(K, I_K, \gamma_K)$,
a left adjoint for the forgetful functor
\[ \begin{tikzcd}[row sep=small]
  (\Alg{K} + \Ideal_{I_K} + \PD_{\gamma_K})\dashmod(\Set)
  \ar[d, equal] \ar[r] &
  (\Alg{K} + \Ideal_{I_K})\dashmod(\Set)
  \ar[d, equal] \\
  (K, I_K, \gamma_K) /
  (\Ring + \Ideal + \PD)\dashmod(\Set) &
  (K, I_K) / (\Ring + \Ideal)\dashmod(\Set)
  \rlap{.}
\end{tikzcd} \]
Forming the PD envelope $(A', I', \gamma')$ of $(A, I)$
(possibly over some $(K, I_K, \gamma_K)$)
not only \enquote{enlarges} the ideal $I$ as necessary
but also the ring $A$.
However,
there is an isomorphism
$A/I \cong A'/I'$,
induced by the unit of the adjunction
$\eta_{(A, I)} : (A, I) \to (A', I')$.

The polynomial ring $A[X]$
over a PD ring $A = (A, I, \gamma)$
inherits a PD structure on the ideal $I[X]$
(generated by the elements of $I$),
such that PD maps $A[X] \to A'$
over $(A, I, \gamma)$
correspond to elements of $A'$.
The same is true for an arbitrary set of polynomial variables
instead of a single one.
But there is a separate notion of \emph{PD polynomial algebra}
\citestacks{07H4},
denoted $A\angles{X}$,
which is freely generated as an $A$-module
by $1$ and the divided powers $\gamma_n(X)$ of $X$
(instead of the ordinary powers $X^n$).
The PD ideal of $A\angles{X}$
is generated by the elements of $I$ and the $\gamma_n(X)$,
and PD maps $A\angles{X} \to A'$ over $(A, I, \gamma)$
correspond to elements of the PD ideal $I'$ of $A'$.
Equivalently,
$A\angles{X}$ is the PD envelope of
the ideal $I[X] + (X)$ in $A[X]$
over $(A, I, \gamma)$.
Again, an arbitrary set of variables works just as well.

\subsection{Relevant geometric theories}

Here,
we define the various theories and theory extensions
from which the syntactic presentation of the crystalline topos
will be built.
We do this right away,
since it will also be of some use
in the definition of the crystalline topos itself.
In particular,
by giving our formal definition of a PD structure
in the form of an extension of geometric theories,
it can immediately be used to define PD structures on schemes
as well.

\begin{definition}
  \label{definition-theories-with-ideals}
  \leavevmode
  \begin{enumerate}[label=(\roman*)]
    \item
      We denote $\Ideal$ the extension of the theory $\Ring$
      (see Definition \ref{definition-theories-for-Zar})
      consisting of a relation symbol $I \subseteq A$
      with these axioms:
      \[ 0 \in I ,\qquad
         x \in I \land y \in I
         \turnstile{x, y : A} x + y \in I ,\qquad
         x \in I \turnstile{\lambda, x : A} \lambda x \in I
         \rlap{.}
      \]

    \item
      In the theory $\Ring + \Ideal$,
      we can require the ideal to be a nil ideal
      by adding the axiom:
      \[ (\nil) \defeq \qquad
         x \in I \turnstile{x : A} \bigvee_{n \in \NN} (x^n = 0)
         \rlap{.} \]

    \item
      We denote $\PDIdeal$ the following extension
      of the theory $\Ring$.
      First introduce a sort $S_I$,
      a function symbol $\iota : S_I \to A$
      with
      \[ \iota(x) = \iota(y) \turnstile{x, y : S_I} x = y \]
      and function symbols
      $0 : S_I$, ${+} : S_I \times S_I \to S_I$
      and ${\cdot} : A \times S_I \to S_I$
      together with equational axioms stating that
      $\iota$ is an $A$-module homomorphism.
      Then add function symbols $\gamma_n : S_I \to S_I$
      for all $n \geq 1$
      and these equational axioms
      (see \citestacks{07GL}):
      \begin{itemize}[nosep]
        \item
          $\gamma_1(x) = x$
        \item
          $\gamma_n(x + y) =
          \gamma_n(x) + \gamma_n(y)
          + \sum\limits_{i + j = n,\; i, j \geq 1}
            \iota(\gamma_i(x)) \gamma_j(y)$
        \item
          $\gamma_n(\lambda x) = \lambda^n \gamma_n(x)$
        \item
          $\iota(\gamma_m(x)) \gamma_n(x) =
          \binom{m + n}{m} \gamma_{m + n}(x)$
        \item
          $\gamma_m(\gamma_n(x)) =
          \frac{(m n)!}{m! (n!)^m} \gamma_{m n}(x)$
      \end{itemize}

    \item
      To be able to talk about a PD structure
      on an existing ideal,
      we define the extension $\PD$ of $\Ring + \Ideal$
      to be $\PDIdeal$ plus the axiom
      \[ \ex{x}{S_I} \iota(x) = y
         \doubleturnstile{y : A}
         y \in I
         \rlap{.} \]
  \end{enumerate}
\end{definition}

Note that $\Ideal + \PD$ and $\PDIdeal$
are Morita equivalent extensions of the theory $\Ring$,
since the axioms in $\Ideal$ are redundant here,
and the relation symbol $I$ is definable by the formula
$\ex{x}{S_I} \iota(x) = y$.
We prefer to write $\Ring + \Ideal + \PD$
because we think of the divided powers
as an additional structure on an ideal.
Also note that $\PD$ is equivalent to
a localic extension of $\Ring + \Ideal$,
as the divided power structure
could alternatively be implemented as relation symbols
$\widetilde{\gamma}_n \subseteq A \times A$
expressing partial functions from $A$ to $A$.
The main purpose of the sort $S_I$ above
is to be able to write the axioms in a more readable equational style.

\begin{definition}
  \label{definition-extensions-of-K-Alg}
  \leavevmode
  \begin{enumerate}[label=(\roman*)]
    \item
      Similarly to the extension $\AlgStr{K}$ of $\Ring$
      from Definition \ref{definition-theories-for-Zar},
      for a $K$-algebra $R$
      we denote $\AlgStr{R}_K$ the extension of $\Alg{K}$
      adding an $R$-algebra structure
      compatible with the given $K$-algebra structure.
      If necessary,
      we write $\AlgStr{R}_K(A)$ to indicate the sort $A$
      to which the $R$-algebra structure is added.

    \item
      For a ring $K$ and a $K$-algebra $R$,
      we denote $\AlgAlg{K}{R}$ the theory $\Alg{K} + \Alg{R}$
      extended by a function symbol $f : A \to B$
      (where $A$ is the $K$-algebra and $B$ is the $R$-algebra)
      and equational axioms
      expressing that $f$ is a $K$-algebra homomorphism.
      A model can be pictured as
      \[ \begin{tikzcd}
        K \ar[r] \ar[d] & R \ar[d] \\
        A \ar[r, "f"] & B
        \rlap{.}
      \end{tikzcd} \]
      When we treat $\AlgAlg{K}{R}$
      as if it contained $\Alg{K} + \Ideal$,
      we intend to use the kernel of $f$ as the ideal,
      that is, $x \in I$ is defined as $f(x) = 0$.

    \item
      We set $\AlgQuot{K}{R} \defeq \AlgAlg{K}{R} + (\surj)$,
      where
      \[ (\surj) \defeq \qquad
         \top \turnstile{y : B} \ex{x}{A} f(x) = y
         \rlap{.} \]

    \item
      For a ring $K$ and an ideal $I_K \subseteq K$,
      we denote $\Ideal_{I_K}$ the extension of $\Alg{K}$
      consisting of the extension $\Ideal$ and the additional axioms
      $c_\lambda \in I$ for all $\lambda \in I_K$.

    \item
      For a PD ring $(K, I_K, \gamma_K)$,
      we denote $\PD_{\gamma_K}$
      the extension of $\Alg{K} + \Ideal_{I_K}$
      consisting of the extension $\PD$
      and the additional axioms
      (for every $\lambda \in I_K$)
      \[ \iota(x) = c_\lambda \turnstile{x : S_I}
         \iota(\gamma_n(x)) = c_{\gamma_{K, n}(\lambda)}
         \rlap{.} \]
      We can write $\AlgAlg{K}{R} + \PD_{\gamma_K}$
      if $I_K$ vanishes in $R$,
      since $\AlgAlg{K}{R}$ then proves
      $f(c_\lambda) = 0$ for $\lambda \in I_K$.
  \end{enumerate}
\end{definition}

Similarly to the economical version of the theory $\Alg{K}$
mentioned in Remark \ref{remark-economical-formulation-of-Alg},
we can use a presentation
$R = K[X_i]/(r_j)$ of $R$ as a $K$-algebra
to implement the $R$-algebra structure of $B$
in the theory $\AlgAlg{K}{R}$.
That is,
$\AlgAlg{K}{R}$ is equivalent to
the extension of $\AlgAlg{K}{K}$
where we only add constant symbols $c_i : B$
and axioms $r_j(c_i) = 0$,
using the available $K$-algebra structure
to interpret $r_j(c_i)$ as a term of sort $B$.
This will be used in Propositions
\ref{proposition-PD-presheaf-type-finitely-presented}
and
\ref{proposition-pd-thickenings-presheaf-type}.

\begin{remark}
  The theories
  $\Alg{K}$, $\Alg{K} + \Ideal_{I_K}$ and
  $\Alg{K} + \Ideal_{I_K} + \PD_{\gamma_K}$
  have in common that their categories of $\Set$-models
  are equivalent to certain slice categories.
  \begin{align*}
    \Alg{K}\dashmod(\Set) &\simeq
    K / (\Ring\dashmod(\Set)) \\
    (\Alg{K} + \Ideal_{I_K})\dashmod(\Set) &\simeq
    (K, I_K) / ((\Ring + \Ideal)\dashmod(\Set)) \\
    (\Alg{K} + \Ideal_{I_K} + \PD_{\gamma_K})\dashmod(\Set) &\simeq
    (K, I_K, \gamma_K) / ((\Ring + \Ideal + \PD)\dashmod(\Set))
  \end{align*}
  We even have
  \[ (\AlgAlg{K}{R})\dashmod(\Set) \simeq
     (K \to R) / (\Ring^{\to}\dashmod(\Set)) \]
  if we use a theory $\Ring^{\to}$ of ring homomorphisms.
  We could in fact define these theories
  by a general construction
  that turns a geometric theory $\TT$
  and a $\Set$-model $M$ of $\TT$
  into a theory of \enquote{$M$-algebras}.
  This is the same construction that is used by Blechschmidt
  in \cite[Definition 3.1]{blechschmidt:nullstellensatz},
  who however applies it
  internally in the classifying topos of $\TT$,
  to the universal model of $\TT$ instead of a $\Set$-based model.
  But for our present goals,
  we need to be able to combine the above theories and extensions
  in flexible ways,
  so that such a formulation does not help us much.
\end{remark}

\subsection{Definition of the big crystalline topos}

Before we can define the crystalline topos,
we need to introduce a number of notions
involving PD structures on schemes.
Recall that a closed embedding $U \to T$
defined by a quasi-coherent ideal sheaf
$\mathcal{I} \subseteq \mathcal{O}_T$
is a \emph{thickening}
(in the terminology of \citestacks{04EX}),
that is, $U \to T$ is a homeomorphism,
if and only if every local section of $\mathcal{I}$
is locally nilpotent,
which in turn means simply that
the internal ideal $\mathcal{I}$ of the internal ring $\mathcal{O}_T$
in $\Sh(T)$
satisfies the axiom $(\nil)$.

\begin{definition}[see \citestacks{07I1}]
  \leavevmode
  \begin{enumerate}[label=(\roman*)]
    \item
      A \emph{PD scheme} is a scheme $S$
      together with a quasi-coherent ideal sheaf
      $\mathcal{I} \subseteq \mathcal{O}_S$
      and a PD structure on $\mathcal{I}$,
      that is,
      a model extension along $\PD$
      of the $(\Ring + \Ideal)$-model
      $(\mathcal{I} \triangleleft \mathcal{O}_S)$
      in $\Sh(S)$.
      A \emph{morphism} of PD schemes
      $(S, \mathcal{I}, \gamma) \to (S', \mathcal{I}', \gamma')$
      is a morphism of schemes $f : S \to S'$
      such that $f^\sharp : f^{-1}\mathcal{O}_{S'} \to \mathcal{O}_S$
      is a $(\Ring + \Ideal + \PD)$-model homomorphism.
      In particular,
      it induces a morphism
      between the closed subschemes of $S$ and $S'$
      defined by $\mathcal{I}$ respectively $\mathcal{I}'$,
      \[ \begin{tikzcd}
        S' & V(\mathcal{I}') \ar[l, hook] \\
        S \ar[u] & V(\mathcal{I}) \ar[l, hook] \ar[u]
        \rlap{.}
      \end{tikzcd} \]

    \item
      A \emph{PD thickening}
      is a PD scheme $(T, \mathcal{I}, \gamma)$
      such that the closed embedding $V(\mathcal{I}) \to T$
      is a thickening,
      that is,
      the $(\Ring + \Ideal + \PD)$-model
      $(\mathcal{O}_T, \mathcal{I}, \gamma)$
      in $\Sh(T)$
      fulfills the axiom $(\nil)$.
      PD thickenings are sometimes denoted $(U, T, \gamma)$,
      where $U = V(\mathcal{I})$.

    \item
      If $(S, \mathcal{I}_S, \gamma_S)$ is a PD scheme
      and $X$ is a scheme over
      the closed subscheme $S_0 \defeq V(\mathcal{I}_S)$ of $S$,
      then a PD thickening \emph{over} $S$ \emph{and} $X$
      is a PD thickening $(T, \mathcal{I}, \gamma)$
      over $(S, \mathcal{I}_S, \gamma_S)$
      together with a morphism of schemes $V(\mathcal{I}) \to X$ over $S$.
      \[ \begin{tikzcd}[column sep=small]
        S & S_0 \ar[l, hook] & X \ar[l] \\
        T \ar[u] & & V(\mathcal{I}) \ar[ll, hook] \ar[u]
      \end{tikzcd} \]
      A \emph{morphism} of PD thickenings over $S$ and $X$
      is a morphism of PD schemes over $S$
      such that the induced morphism $V(\mathcal{I}) \to V(\mathcal{I}')$
      respects the structure morphisms to $X$.
  \end{enumerate}
\end{definition}


The PD thickenings over $S$ and $X$
will be the objects of the crystalline site of $X/S$.
Most important to us is that
if $S$ and $X$ are affine
and we also require $T$ to be affine,
then we can nicely describe the resulting category
by a geometric theory.

\begin{lemma}
  \label{lemma-affine-PD-objects-are-models}
  The full subcategory of the category of PD thickenings
  where the underlying scheme $T$ is affine
  is equivalent to
  \[ \mod{(\Ring + \Ideal + \PD + (\nil))}{\Set}^\op \rlap{.} \]

  For a fixed PD ring $(K, I_K, \gamma_K)$
  and a $K/I_K$-algebra $R$,
  the full subcategory of the category of PD thickenings
  over $S = \Spec(K)$ and $X = \Spec(R)$
  where the underlying scheme $T$ is affine
  is equivalent to
  \[ \mod{(\AlgQuot{K}{R} + \PD_{\gamma_K} + (\nil))}{\Set}^\op
     \rlap{.} \]
\end{lemma}

\begin{proof}
  The category
  $\mod{(\Ring + \Ideal)}{\Set}$
  is of course dual to
  the category of affine schemes
  equipped with a quasi-coherent ideal sheaf,
  via $(I \triangleleft A) \mapsto (\Spec A, \tilde{I})$.
  One can check that the closed embedding
  $\Spec (A/I) \to \Spec A$
  is a thickening if and only if every element of $I$ is nilpotent.
  Similarly,
  a PD structure on $I$
  uniquely determines a PD structure on $\tilde{I}$,
  and morphisms respecting one
  correspond to morphisms respecting the other
  \cite[p. 3.18]{berthelot-ogus}.

  For the second part,
  the theory extension $\PD_{\gamma_K}$
  (requiring a PD structure on the kernel of $f : A \to B$
  extending $\gamma_K$)
  makes sense because the PD ideal $I_K$ vanishes in $R$,
  so the theory $\AlgQuot{K}{R}$ proves $f(c_\lambda) = 0$
  for all $\lambda \in I_K$.
  The rest follows from the first part.
\end{proof}

If $(T, \mathcal{I}, \gamma)$ is a PD thickening
and $T' \subseteq T$ is an open subscheme,
then $(T', \mathcal{I}|_{T'}, \gamma|_{T'})$
is again a PD thickening
(as restricting to $T'$ is pulling back along
the geometric morphism
$\Sh(T') \to \Sh(T)$)
and there is a canonical inclusion morphism
$(T', \mathcal{I}|_{T'}, \gamma|_{T'}) \to (T, \mathcal{I}, \gamma)$.
This allows us to define a Zariski topology
on the category of PD thickenings over $S$ and $X$.

\begin{definition}
  Let $(S, \mathcal{I}_S, \gamma_S)$ be a PD scheme
  and let $X$ be a scheme over $S_0 \defeq V(\mathcal{I}_S)$.
  The \emph{big crystalline site} $\Cris(X/S)$
  is the category of PD thickenings over $S$ and $X$,
  endowed with the Zariski topology,
  that is,
  a sieve on $(T, \mathcal{I}, \gamma)$ is covering
  if and only if
  it contains all the arrows
  $(T_i, \mathcal{I}|_{T_i}, \gamma|_{T_i}) \to (T, \mathcal{I}, \gamma)$
  for some open cover $T = \bigcup_i T_i$ of $T$.
\end{definition}

If we want to consider the topos $\Sh(\Cris(X/S))$
of sheaves on the crystalline site,
there remains an issue of size.
There is no way to talk about the collection of presheaves
on a big category such as $\Cris(X/S)$
(or generally,
about the collection of functions
between two proper classes)
in ZFC set theory.
If the category in question is essentially small,
then one can circumvent this
by choosing a skeleton.
More generally,
if $J$ is a Grothendieck topology
on a category $\mathcal{C}$
such that there is a (small) dense set $S \subseteq \mathcal{C}$
of objects of $\mathcal{C}$,
then one can write
\[ \Sh(\mathcal{C}, J) \defeq \Sh(S, J|_S) \]
(where $J|_S$ is the greatest topology on $S$
such that $S \hookrightarrow \mathcal{C}$ preserves covers).
This is justified because
if $S' \subseteq \mathcal{C}$ is another dense set of objects,
then by the Comparison Lemma,
there are canonical equivalences
\[ \Sh(S, J|_S) \simeq \Sh(S \cup S', J|_{S \cup S'})
   \simeq \Sh(S', J|_{S'})
   \rlap{.} \]

However,
the site $\Cris(X/S)$ does not admit a dense set of objects.
To define a crystalline topos,
a choice of some class $\mathcal{C} \subseteq \Cris(X/S)$
which does admit a dense subset
is therefore necessary.
We leave this choice
as an explicit parameter to the definition for now,
but we will specialize to a particular class $\mathcal{C}$
in Definition \ref{definition-Cris-fp}.

To avoid an ambiguity in the topology
on the resulting site,
we require $\mathcal{C}$ to be closed under
taking open subschemes $T' \subseteq T$
(with the induced PD structure on $T'$).
While there is always the induced topology $J_\Zar|_\mathcal{C}$,
for which a sieve on $T \in \mathcal{C}$ is covering
if and only if
it generates a covering sieve in $\Cris(X/S)$,
if $\mathcal{C}$ does not contain all open subschemes of $T$,
then there is no guarantee that such a sieve
actually contains a \enquote{distinguished} Zariski cover
by open subschemes.
With this assumption,
the two conditions become equivalent.

\begin{definition}
  Let $S$, $X$ be as before and
  let $\mathcal{C} \subseteq \Cris(X/S)$
  be a class of objects
  admitting a dense subset
  and closed under taking open subschemes.
  Then we denote $\Cris_\mathcal{C}(X/S)$
  the full subcategory on these objects
  endowed with the Zariski topology as above.
  And the topos
  \[ (X/S)_{\Cris_\mathcal{C}} \defeq \Sh(\Cris_\mathcal{C}(X/S)) \]
  is the \emph{big crystalline topos}
  of $X$ over $(S, \mathcal{I}_S, \gamma_S)$
  (defined using the objects in $\mathcal{C}$).
\end{definition}

There is a structure sheaf on $\Cris_\mathcal{C}(X/S)$
(see \citestacks{07IN}),
\[ \mathcal{O} = \mathcal{O}_{(X/S)_{\Cris_\mathcal{C}}} \;:\;
   \Cris_\mathcal{C}(X/S)^\op \to \Set, \quad
   (T, \mathcal{I}, \gamma) \mapsto \mathcal{O}_T(T)
   \rlap{.} \]
As in the case of the Zariski topos,
this sheaf carries a ring structure.
But here, there is more.
We have a second ring object
\[ \mathcal{O}' = \mathcal{O}'_{(X/S)_{\Cris_\mathcal{C}}} \;:\;
   \Cris_\mathcal{C}(X/S)^\op \to \Set, \quad
   (T, \mathcal{I}, \gamma) \mapsto
   \mathcal{O}_U(U) =
   \mathcal{O}_{V(\mathcal{I})}(V(\mathcal{I}))
   \rlap{,} \]
sometimes denoted $\mathcal{O}_X$.
Furthermore, there is a surjective ring homomorphism
$\mathcal{O} \to \mathcal{O}'$,
and its kernel
\[ \mathcal{J} = \mathcal{J}_{(X/S)_{\Cris_\mathcal{C}}} \;:\;
   \Cris_\mathcal{C}(X/S)^\op \to \Set, \quad
   (T, \mathcal{I}, \gamma) \mapsto \mathcal{I}(T)
   \]
carries a canonical PD structure.
In summary,
we have a model
\[ (\mathcal{O}, \mathcal{J}) \;\in\;
   (\Ring + \Ideal + \PD)\dashmod((X/S)_{\Cris_\mathcal{C}})
   \rlap{.} \]
An extension of this model
(along a localic theory extension)
will turn out to be
the universal model of a theory
classified by $(X/S)_{\Cris_\mathcal{C}}$,
at least for an appropriate choice of $\mathcal{C}$.

The Zariski topology on $\Cris(X/S)$
(or $\Cris_\mathcal{C}(X/S)$)
in particular ensures that every object can be covered
by objects $(T, \mathcal{I}, \gamma)$ with $T$ affine,
in other words, the affine objects are dense.
Thus we can use the Comparison Lemma
to obtain an equivalent site
consisting of affine objects,
which we have already described as a category in
Lemma \ref{lemma-affine-PD-objects-are-models}.
The following lemma complements this
with a description of the topology
on the new site.

\begin{lemma}
  \label{lemma-affine-Zariski-topology-for-Cris}
  Let $(K, I_K, \gamma_K)$ be a PD ring
  and let $R$ be a $K/I_K$-algebra.
  The topology induced via the Comparison Lemma
  on $(\AlgQuot{K}{R} + \PD_{\gamma_K} + (\nil))
  \dashmod(\Set)^\op$
  regarded as a (dense) full subcategory
  of $\Cris(\Spec R / \Spec K)$
  is the following:
  A cosieve on an object
  \[ (\gamma \curvearrowright I \triangleleft
     A \twoheadrightarrow B) \]
  is covering
  if and only if
  it contains the canonical arrows to
  \[ (\gamma_{a_i} \curvearrowright I_{a_i} \triangleleft A_{a_i}
     \twoheadrightarrow B_{a_i}) \]
  for some finite family of elements $a_i \in A$
  with $(a_1, \dots, a_n) = (1)$.
\end{lemma}

\begin{proof}
  We first note that the model
  $(\gamma_{a_i} \curvearrowright I_{a_i} \triangleleft A_{a_i}
  \twoheadrightarrow B_{a_i})$
  is well-defined
  --- the ring map $A_{a_i} \to B_{a_i}$ is surjective
  and has kernel $I_{a_i}$
  since localization is exact,
  and $I_{a_i}$ is still a nil ideal ---
  and corresponds to the open subscheme $D(a_i)$
  of the PD scheme $\Spec A_{a_i}$.
  A cosieve on
  $(\gamma \curvearrowright I \triangleleft A \twoheadrightarrow B)$
  is covering for the induced topology
  if it generates a covering sieve in $\Cris(\Spec R / \Spec K)$.
  But such a sieve contains
  (the morphisms of
  PD thickenings over $\Spec K$ and $\Spec R$
  induced by)
  an open cover of $\Spec A$
  if and only if it contains a cover by standard opens $D(a_i)$,
  which in turn
  cover $\Spec A$ if and only if $(a_1, \dots, a_n) = (1)$.
\end{proof}

\subsection{Preliminary presheaf type results}

We saw in Lemma \ref{lemma-affine-PD-objects-are-models}
that the crystalline site of affine schemes
$S = \Spec K$, $X = \Spec R$
is closely tied to the theory
\[ \AlgQuot{K}{R} + \PD_{\gamma_K} + (\nil) \rlap{.} \]
Our strategy for proving a classification result
on the crystalline topos
is to show,
under appropriate assumptions,
that this theory is of presheaf type,
then choose the class of objects of the crystalline site
in such a way that the affine objects
turn out to be exactly the compact models of the theory,
and finally add the quotient extension $(\loc)$
to produce the Zariski topology.

Right now,
we want to approach the presheaf type part
by looking at somewhat simpler theories.
While $\AlgAlg{K}{R}$ is of presheaf type
simply because it is an algebraic theory,
the situation already becomes more interesting
for $\AlgQuot{K}{R}$,
as the axiom $(\surj)$ is not cartesian.
We also include the extension $\PD$
but assume the PD structure on $K$ to be trivial
and forget about $(\nil)$ for now.

The following two propositions
can also serve as an illustration
of the different ways in which
one might try to decompose the same theory into parts
to show that it is of presheaf type,
and how this can lead to results of different strengths.

\begin{proposition}
  \label{proposition-PD-presheaf-type-finitely-presented}
  Let $R$ be a finitely presented $K$-algebra.
  Then the theory
  \[ \AlgQuot{K}{R} + \PD \]
  is of presheaf type.
\end{proposition}


\begin{proof}
  The cartesian theory $\Alg{K} + \Ideal + \PD$
  is Morita-equivalent to $\AlgQuot{K}{K} + \PD$
  by introducing a sort $B$ for the quotient ring $A/I$,
  similarly to Example \ref{example-materialize-quotient}.
  So the latter theory is of presheaf type
  and we are only missing the $R$-algebra structure
  on the $K$-algebra $B$.
  Since $R \cong K[X_1, \dots, X_n]/(r_1, \dots, r_m)$
  is a finitely presented $K$-algebra,
  the extension $\AlgStr{R}_K(B)$ is equivalent to
  finitely many constant symbols $c_i \oftype B$,
  $i = 1, \dots, n$,
  and finitely many axioms in the empty context
  $\top \turnstile{[]} r_j(\vec{c}) = 0$,
  $j = 1, \dots, m$.
  Thus we are done by
  Proposition \ref{proposition-add-constant-symbol}
  and Corollary\ref{corollary-axiom-in-empty-context}.
\end{proof}

We can reduce the assumption on $R$
to finite type instead of finite presentation
if we handle the axiom $(\surj)$
in a different way.
For this, we need a lemma.

\begin{lemma}
  \label{lemma-PD-envelope}
  Let $M = (\gamma \curvearrowright \mathfrak{a}
  \triangleleft A \xrightarrow{f} B) \in
  \mathcal{C} =
  \mod{(\AlgAlg{K}{R} + \PD)}{\Set}$
  and $b \in B$.
  Then there is a universal triple $(M', a', g)$,
  where $M' = (\gamma' \curvearrowright \mathfrak{a}'
  \triangleleft A' \xrightarrow{f'} B')
  \in \mathcal{C}$,
  $a' \in A'$
  and $g = (g_A, g_B) : M \to M'$,
  such that $f'(a') = g_B(b)$;
  that is,
  the following diagram in $[\mathcal{C}, \Set]$
  is a pullback:
  \[ \begin{tikzcd}
    \Hom(M', \hole) \ar[r, "g^*"] \ar[d, "a'", swap]
    \ar[rd, "\lrcorner", very near start, phantom] &
    \Hom(M, \hole) \ar[d, "b"] \\
    \interpretation{A}_\hole \ar[r, "\interpretation{f}_\hole", swap] &
    \interpretation{B}_\hole
  \end{tikzcd} \]
  Furthermore, $g_B : B \to B'$ is an isomorphism.
\end{lemma}

\begin{proof}
  Consider the $A$-algebra homomorphism
  $\widetilde{f} : \widetilde{A} \defeq A[X] \to B$
  with $\widetilde{f}(X) = b$
  and set $\widetilde{\mathfrak{a}} \defeq \ker \widetilde{f}$.
  Let $\gamma' \curvearrowright \mathfrak{a}' \triangleleft A'$
  be the PD envelope (over $\gamma$)
  of the $(\mathfrak{a} \triangleleft A)$-algebra
  $\widetilde{\mathfrak{a}} \triangleleft \widetilde{A}$.
  In particular, we have
  $\widetilde{A}/\widetilde{\mathfrak{a}} \cong A'/\mathfrak{a}'$
  (as $A$-algebras)
  via the unit
  $\eta_{\widetilde{\mathfrak{a}} \triangleleft \widetilde{A}} :
  (\widetilde{\mathfrak{a}} \triangleleft \widetilde{A}) \to
  (\mathfrak{a}' \triangleleft A')$
  of the adjunction,
  so that $\widetilde{f}$ induces
  an $A$-algebra homomorphism $f' : A' \to B$
  with $\ker f' = \mathfrak{a}'$.
  We have thus defined a model $M' \in \mathcal{C}$
  and an arrow $g : M \to M'$,
  where $g_A : A \to A'$ is just the homomorphism of PD-rings
  from the PD envelope construction
  and $g_B = \id_B$.
  Setting $a' \defeq
  \eta_{\widetilde{\mathfrak{a}} \triangleleft \widetilde{A}}(X)$
  satisfies $f'(a') = b$.

  To check the universal property,
  let $h : M \to M'' = (\gamma'' \curvearrowright \mathfrak{a}''
  \triangleleft A'' \xrightarrow{f''} B'')$
  be given.
  An element $a'' \in A'' = \interpretation{A}_{M''}$
  with $f''(a'') = h_B(b)$
  is the same thing as
  an $A$-algebra homomorphism $\widetilde{A} \to A''$
  fitting in the square
  \[ \begin{tikzcd}
    \widetilde{A} \ar[r, "\widetilde{f}"] \ar[d] & B \ar[d, "h_B"] \\
    A'' \ar[r, "f''"] & B'' \rlap{.}
  \end{tikzcd} \]
  But then it is automatic that
  $\widetilde{\mathfrak{a}}$ is sent to $\mathfrak{a}''$,
  so there is a unique homomorphism of PD-rings
  $(\gamma' \curvearrowright \mathfrak{a}' \triangleleft A') \to
  (\gamma'' \curvearrowright \mathfrak{a}'' \triangleleft A'')$
  over $A$
  sending $a'$ to $a''$,
  and it makes the square
  \[ \begin{tikzcd}
    A' \ar[r, "f'"] \ar[d] & B \ar[d, "h_B"] \\
    A'' \ar[r, "f''"] & B''
  \end{tikzcd} \]
  commute,
  as required for a model homomorphism $M' \to M''$.
\end{proof}

\begin{proposition}
  \label{proposition-PD-presheaf-type-finitely-generated}
  Let $R$ be a finitely generated $K$-algebra.
  Then the theory
  \[ \AlgQuot{K}{R} + \PD \]
  is of presheaf type.
\end{proposition}

\begin{proof}
  The theory $\AlgAlg{K}{R} + \PD$
  (with a PD structure on the kernel of the $K$-algebra homomorphism
  $A \to B$)
  is of presheaf type because it is cartesian.
  We show that the missing axiom $(\surj)$
  is a rigid-topology quotient.

  Let a compact model
  \[ M = (\gamma \curvearrowright \mathfrak{a} \triangleleft A
  \xrightarrow{f} B) \]
  of $\AlgAlg{K}{R} + \PD$
  and an element $b \in B$ be given.
  Consider the model homomorphism $g : M \to M'$
  from Lemma \ref{lemma-PD-envelope}.
  We first note that $M'$ is also compact,
  since $\Hom(M', \hole)$ is a finite limit of functors
  preserving filtered colimits.
  (For $\interpretation{A}_\hole$ and $\interpretation{B}_\hole$
  preserving filtered colimits,
  see Lemma \ref{lemma-filtered-colimits-of-models}.)
  Secondly,
  a map $h : M \to M''$ factors through $M'$
  (not necessarily uniquely)
  if and only if
  $h_B$ maps $b$ to something in the image of $f''$.
  This means that the cosieve on $M$ generated by $g$
  (in the category $\mod{(\AlgAlg{K}{R} + \PD)}{\Set}_c$)
  is exactly the cosieve $S_b$
  as in Theorem \ref{theorem-induced-topology}
  and thus $M$ is $J_{(\surj)}$-covered by $M'$.
  And lastly,
  the subset $f'(A') \subseteq B' = B$
  contains of course both $f(A)$ and the element $b$.

  Now,
  since $M$ is compact,
  $B$ is a finitely presented $R$-algebra.
  This follows from
  Lemma \ref{lemma-adjoint-functor-preserves-compacts}
  by observing that the right adjoint of the forgetful functor
  in the following diagram preserves filtered colimits.
  \[ \begin{tikzcd}[row sep=0mm]
    (\AlgAlg{K}{R} + \PD)\dashmod(\Set) \ar[r] &
    (\Alg{R})\dashmod(\Set) \\
    (\gamma \curvearrowright \mathfrak{a} \triangleleft
    A \to B) \ar[r, mapsto] &
    B \\
    ((0) \triangleleft B' \xrightarrow{=} B') &
    B' \ar[l, mapsto]
    \ar[lu, phantom, "\rotatebox{90}{$\vdash$}"]
  \end{tikzcd} \]
  Thus, $B$ is also a finitely generated $K$-algebra by assumption.
  But if $b_1, \dots, b_n \in B$ are $K$-algebra generators of $B$,
  then we can successively apply Lemma \ref{lemma-PD-envelope}
  to them
  to obtain a covering arrow $M \to M_n$
  (using the transitivity of the topology $J_{(\surj)}$)
  such that $b_i \in f_n(A_n)$ for all $i$,
  implying that $f_n$ is surjective.
\end{proof}

\begin{remark}
  Propositions \ref{proposition-PD-presheaf-type-finitely-presented}
  and \ref{proposition-PD-presheaf-type-finitely-generated}
  become wrong if we drop the finiteness assumption entirely.
  For example, take $K = \QQ$, $R = \CC$.
  The extension $\PD$ is an equivalence extension
  of $\AlgQuot{\QQ}{\CC}$
  because the $\gamma_n$ are definable by
  \[ \gamma_n(x) = y \doubleturnstile{x, y \oftype S_I}
     \frac{1}{n!}\iota(x)^n = \iota(y)
     \rlap{,} \]
  using the available $\QQ$-algebra structure.
  And $\AlgQuot{\QQ}{\CC}$ has
  too few compact $\Set$-based models
  to be of presheaf type,
  similarly to the situation in
  Example \ref{example-add-many-constants-to-quotient}.

  Namely,
  we can apply
  Lemma \ref{lemma-adjoint-functor-preserves-compacts}
  to the pair of adjoint functors
  \[ \begin{tikzcd}[row sep=0mm]
    (\AlgQuot{\QQ}{\CC})\dashmod(\Set) \ar[r] &
    (\Alg{\QQ})\dashmod(\Set) \\
    (A \twoheadrightarrow B) \ar[r, mapsto] &
    A \\
    (A' \twoheadrightarrow 0) &
    A' \ar[l, mapsto]
    \ar[lu, phantom, "\rotatebox{90}{$\vdash$}"]
  \end{tikzcd} \]
  to see that
  if $A \twoheadrightarrow B$ is a compact
  model of $\AlgQuot{\QQ}{\CC}$,
  then $A$ is a finitely generated $\QQ$-algebra.
  Then the $\CC$-algebra $B$
  is also a finitely generated $\QQ$-algebra,
  which implies $B = 0$
  for cardinality reasons.
  But there are certainly non-compact models with $B \neq 0$
  (for example, let $A$ be freely generated by the elements of $B$),
  so the sequent
  \[ \top \turnstile{[]} 1_B = 0_B \]
  can not be provable in $\AlgQuot{\QQ}{\CC}$,
  meaning that the compact $\Set$-based models
  are not jointly conservative
  and therefore $\AlgQuot{\QQ}{\CC}$ is not of presheaf type
  (see \cite[Theorem 6.1.1]{caramello:tst}).
\end{remark}

We can now prove a \enquote{lazy version}
of our classification result,
where we simply assume the crystalline site
to contain exactly the desired objects.
More precisely,
we even have to assume that a suitable class of objects
$\mathcal{C} \subseteq \Cris(X/S)$
exists,
which is not trivial
given the additional assumption that
$\mathcal{C}$ is closed under taking open subschemes.
The axiom $(\nil)$ is made redundant here
by an additional assumption on $K$.

\begin{corollary}
  \label{corollary-lazy-classification-result}
  Let $K$ be a ring of nonzero characteristic,
  regard $K$ as a PD ring with the trivial PD ideal $(0)$,
  and let $R$ be a finitely generated $K$-algebra.
  Assume there is a class
  $\mathcal{C} \subseteq \Cris(\Spec R / \Spec K)$
  admitting a dense subset
  and closed under open subschemes,
  such that an affine object belongs to $\mathcal{C}$
  if and only if the corresponding model of
  $\AlgQuot{K}{R} + \PD + (\nil)$
  is compact.
  Then the topos $(\Spec R / \Spec K)_{\Cris_\mathcal{C}}$
  classifies the theory
  \[ \AlgQuot{K}{R} + \PD + (\nil) + (\loc) \rlap{.} \]
\end{corollary}

\begin{proof}
  If $n = 0$ in $K$ for some $n \geq 1$,
  then for any PD ring $(A, I, \gamma)$ over $K$
  and any $a \in I$,
  we have $a^n = n! \gamma_n(a) = 0$.
  This calculation can be carried out
  within the theory $\AlgQuot{K}{R} + \PD$,
  so that the axiom $(\nil)$ is redundant.
  Thus by Proposition
  \ref{proposition-PD-presheaf-type-finitely-generated},
  $\AlgQuot{K}{R} + \PD + (\nil)$ is of presheaf type.
  By our assumption on $\mathcal{C}$,
  the underlying category of
  the site of affine objects
  for $(\Spec R / \Spec K)_{\Cris_\mathcal{C}}$
  is
  \[ {(\AlgQuot{K}{R} + \PD + (\nil))\dashmod(\Set)_c}^\op
     \rlap{,} \]
  the site for the presheaf topos classifying this theory.
  The only difference is the Zariski topology
  as described in Lemma \ref{lemma-affine-Zariski-topology-for-Cris}.
  By Theorem \ref{theorem-induced-topology},
  this is precisely the topology corresponding to
  the quotient extension $(\loc)$,
  as each arrow
  \[ \begin{tikzcd}
    (\gamma \curvearrowright I \triangleleft A \twoheadrightarrow B)
    \ar[r] &
    (\gamma_{a_i} \curvearrowright I_{a_i} \triangleleft A_{a_i}
    \twoheadrightarrow B_{a_i})
  \end{tikzcd} \]
  is the universal arrow
  that sends $a_i \in A$ to something invertible.
\end{proof}

\begin{remark}
  \label{remark-loc-with-surj-and-nil}
  When adding $(\loc)$ to a theory like $\AlgAlg{K}{R}$,
  we mean to impose the axioms on the $K$-algebra $A$,
  not on the $R$-algebra $B$.
  However,
  in the presence of $(\surj)$ and $(\nil)$,
  the two options are in fact syntactically equivalent.
  This is because
  an elementary calculation using these axioms
  shows
  \[ \inv(f(x)) \doubleturnstile{x \oftype A} \inv(x) \rlap{.} \]
\end{remark}

\subsection{Finiteness conditions for PD schemes}

The set of objects that we include in the crystalline site
will be determined by the compactness condition
on models of the appropriate geometric theory.

\begin{lemma}
  \label{lemma-compact-pd-algebras}
  Let $(K, I_K, \gamma_K)$ be a PD ring.
  The compact models of $\Alg{K} + \Ideal_{I_K} + \PD_{\gamma_K}$
  are those PD rings over $K$ which are of the form
  \[ K\angles{X_1, \dots, X_m}[Y_1, \dots, Y_n]
     / \overline{(r_1, \dots, r_k)}
     \rlap{.} \]
\end{lemma}

\begin{proof}
  While the theory $\Ring + \Ideal + \PD$
  as defined in Definition \ref{definition-theories-with-ideals}
  is cartesian,
  the equivalent theory $\Ring + \PDIdeal$
  is even a Horn theory.
  We can similarly construct a Horn theory
  equivalent to $\Alg{K} + \Ideal_{I_K} + \PD_{\gamma_K}$
  by adding constants $\widetilde{c}_\lambda : S_I$
  in addition to $c_\lambda : A$
  for every $\lambda \in I_K$,
  with $\iota(\widetilde{c}_\lambda) = c_\lambda$.

  By Lemma \ref{lemma-compact-Horn-models},
  the compact models of a Horn theory
  are those presented by some Horn formula in context.
  In our case,
  the model presented by the formula $\top$
  in the context $\vec{x} : {S_I}^m, \vec{y} : A^n$
  is $K\angles{\vec{X}}[\vec{Y}]$.
  And since our theory has no relation symbols
  and the inclusion $\iota : S_I \to A$ is required to be injective,
  any atomic formula in this context
  is provably equivalent to
  an equality of two terms of sort $A$,
  and therefore also to one of the form $r = 0$,
  where $r$ is a term representing an element of
  $K\angles{\vec{X}}[\vec{Y}]$.
  Thus we are done,
  as the universal arrow out of $K\angles{\vec{X}}[\vec{Y}]$
  that kills $r_1, \dots, r_n$
  is $K\angles{\vec{X}}[\vec{Y}] \to
  K\angles{\vec{X}}[\vec{Y}] / \overline{(\vec{r})}$.
\end{proof}

We now want to globalize this property of PD rings
over $(K, I, \gamma)$
to be able to apply it to morphisms of PD schemes.
This parallels the treatment of
morphisms of finite type and morphisms of finite presentation
as for example in \citestacks{01T0} and \citestacks{01TO}.

\begin{definition}
  We say a homomorphism of PD rings
  is \emph{of finite PD type}
  if it is of the form
  \[ (K, I_K, \gamma_K) \to
     K\angles{\vec{X}}[\vec{Y}] / \rho \]
  for finite sets of variables
  $\vec{X} = X_1, \dots, X_n$
  and $\vec{Y} = Y_1, \dots, Y_m$,
  where $\rho = \overline{\rho}$
  is any PD saturated ideal of
  $K\angles{\vec{X}}[\vec{Y}]$.
  We say it is \emph{of finite PD presentation}
  if it is of the form
  $(K, I_K, \gamma_K) \to
  K\angles{\vec{X}}[\vec{Y}] / \overline{(\vec{r})}$
  for finitely many elements
  $r_1, \dots, r_k \in K\angles{\vec{X}}[\vec{Y}]$.
\end{definition}

\begin{lemma}
  \label{lemma-finite-type}
  A homomorphism of PD rings $(K, I_K, \gamma_K) \to (A, I, \gamma)$
  is of finite PD type
  if and only if
  there are elements $x_1, \dots, x_n \in I$
  and $y_1, \dots, y_m \in A$
  such that $A$ is generated as a $K$-algebra
  by $\gamma_k(x_i)$ and $y_i$,
  and the ideal $I$ is generated by
  the image of $I_K$ and $\gamma_k(x_i)$.
\end{lemma}

\begin{proof}
  The choice of elements $x_i \in I$ and $y_j \in A$
  corresponds to a homomorphism $f : K\angles{\vec{X}}[\vec{Y}] \to A$.
  The kernel of $f$ is always a PD saturated ideal.
  This $f$ is surjective if and only if
  $A$ is generated as a $K$-algebra by $\gamma_k(x_i)$ and $y_j$.
  And in this case,
  the induced map between
  the PD ideals of $K\angles{\vec{X}}[\vec{Y}]$ and $A$
  is surjective
  if and only if $I$ is generated by
  the image of $I_K$ and the $\gamma_k(x_i)$.
\end{proof}

\begin{lemma}
  \label{lemma-finite-presentation}
  Homomorphisms of finite PD presentation
  have the following properties.
  \begin{enumerate}[label=(\roman*)]
    \item
      If $R \to A \to A'$ are homomorphisms of PD rings
      with $R \to A$ and $R \to A'$
      of finite PD presentation,
      then $A \to A'$ is also of finite PD presentation.
    \item
      If $\rho$ is a PD saturated ideal of a PD ring $R$
      and $R \to R / \rho$ is of finite PD presentation,
      then there are finitely many elements $r_i \in R$
      such that $\rho = \overline{(\vec{r})}$.
  \end{enumerate}
\end{lemma}

\begin{proof}
  \begin{enumerate}[label=(\roman*)]
    \item
      For any PD ring $R$ and any element $r \in R$,
      we have
      $R[X]/\overline{(X - r)} = R[X]/(X - r) \cong R$,
      but we also have
      $R\angles{X}/\overline{(X - r)} \cong R$
      as long as $r$ lies in the PD ideal of $R$.
      Now,
      let $A = R\angles{\vec{X}}[\vec{Y}]/\overline{(\vec{r})}$,
      $A' = R\angles{\vec{X}'}[\vec{Y}']/\overline{(\vec{r}')}$,
      $f : A \to A'$,
      and choose $x_i$ in the PD ideal of
      $R\angles{\vec{X}'}[\vec{Y}']$
      and $y_i \in R\angles{\vec{X}'}[\vec{Y}']$
      such that $f(X_i) = x_i$ and $f(Y_i) = y_i$ in $A'$.
      Then we can calculate:
      \begin{align*}
        A' &= A' \angles{\vec{X}}[\vec{Y}]
        / \overline{(X_i - x_i, Y_i - y_i)} \\
        &= A' \angles{\vec{X}}[\vec{Y}]
        / \overline{(X_i - x_i, Y_i - y_i, \vec{r})} \\
        &= A \angles{\vec{X}'}[\vec{Y}']
        / \overline{(X_i - x_i, Y_i - y_i, \vec{r}')}
        \rlap{.}
      \end{align*}

    \item
      Let $R\angles{\vec{X}}[\vec{Y}] / \overline{(\vec{r})}
      \cong R / \rho$.
      By choosing elements in $I_R$ respectively $R$
      corresponding to the $X_i$ and $Y_i$,
      this isomorphism is induced by a PD map
      $f : R\angles{\vec{X}}[\vec{Y}] \to R$.
      But then we have $\rho = \overline{(f(r_i))}$.
      \qedhere
  \end{enumerate}
\end{proof}

\begin{lemma}
  \label{lemma-finite-type-presentation-composition}
  Homomorphisms of PD rings of finite type
  and
  homomorphisms of PD rings of finite presentation
  are both stable under composition.
\end{lemma}

\begin{proof}
  Let $(K, I, \gamma)$ be a PD ring.
  We have
  \[ \Big(
     K\angles{\vec{X}}[\vec{Y}] / \overline{(\vec{r})}
     \Big)
     \angles{\vec{X}'}[\vec{Y}'] / \overline{(\vec{r}')} =
     K\angles{\vec{X}, \vec{X}'}[\vec{Y}, \vec{Y}'] /
     \overline{(\vec{r}, \vec{r}'')}
     \rlap{,} \]
  where $r''_i$ is a lift of $r'_i$
  to $K\angles{\vec{X}, \vec{X}'}[\vec{Y}, \vec{Y}']$.
  Here,
  $\vec{r}$ and $\vec{r}'$
  can either be finite lists
  or arbitrary set-indexed families.
\end{proof}

\begin{lemma}
  \label{lemma-finite-type-presentation-base-change}
  Homomorphisms of PD rings of finite type
  and
  homomorphisms of PD rings of finite presentation
  are both stable under base change
  (pushout in the category of PD rings).
\end{lemma}

\begin{proof}
  The pushout of $K\angles{\vec{X}}[\vec{Y}]/\overline{(\vec{r})}$
  along $f : (K, I, \gamma) \to (K', I', \gamma')$
  is $K'\angles{\vec{X}}[\vec{Y}]/\overline{(f(\vec{r}))}$.
\end{proof}

\begin{lemma}
  \label{lemma-finite-type-presentation-local}
  Being of finite PD type
  and being of finite PD presentation
  are \emph{local} properties
  of homomorphisms of PD rings,
  in the sense of \citestacks{01SR},
  that is:
  Let $R = (R, I_R, \gamma_R)$ and $A = (A, I, \gamma)$
  be PD rings.
  \begin{enumerate}[label=(\alph*)]
    \item
      If $R \to A$
      is of finite PD type (presentation)
      and $r \in R$,
      then $R_r \to A_r$,
      with the induced PD structures,
      is of finite PD type (presentation).
    \item
      If $r \in R$, $a \in A$
      and $R_r \to A$ is of finite PD type (presentation),
      then $R \to A_a$ is of finite PD type (presentation).
    \item
      If $R \to A$ is a homomorphism of PD rings,
      $a_i \in A$ are elements such that $(a_1, \dots, a_n) = (1)$,
      and $R \to A_{a_i}$ is of finite PD type (presentation)
      for every $i$,
      then $R \to A$ is of finite PD type (presentation).
  \end{enumerate}
\end{lemma}

\begin{proof}
  \begin{enumerate}[label=(\alph*)]
    \item
      This is a special case of
      Lemma \ref{lemma-finite-type-presentation-base-change}.

    \item
      This follows from
      Lemma \ref{lemma-finite-type-presentation-composition}
      and the fact that $R_r \cong R[X] / \overline{(r X - 1)}$,
      and the same for $A$.

    \item
      Let the $R \to A_{a_i}$ be of finite PD type.
      For every $i$,
      choose elements $x_{i, j} a_i^{-m_{i, j}} \in I_{a_i}$
      and $y_{i, j} a_i^{-n_{i, j}} \in A_{a_i}$
      (with $x_{i, j} \in I$ and $y_{i, j} \in A$)
      according to Lemma \ref{lemma-finite-type}.
      Then $A_{a_i}$ is generated as an $R$-algebra
      by $\gamma_k(x_{i, j} a_i^{-m_{i, j}}) =
      \gamma_k(x_{i, j}) a_i^{-km_{i, j}}$
      and $y_{i, j} a_i^{-n_{i, j}}$.
      Let $\sum_i c_i a_i = 1$ in $A$
      and set $\widetilde{A} \defeq
      R[a_i, c_i, \gamma_k(x_{i, j}), y_{i, j}] \subseteq A$.
      Then we see that for every element $a \in A$
      and every $i$,
      we have $a_i^N a \in \widetilde{A}$
      for $N$ big enough,
      which means that we also have
      $a = (\sum_i c_i a_i)^N a \in \widetilde{A}$
      for $N$ big enough,
      so $\widetilde{A} = A$.
      We also have
      (by Lemma \ref{lemma-finite-type})
      that $I_{a_i}$ is generated as an ideal
      by the elements of $I_R$ and the
      $\gamma_k(x_{i, j}) a_i^{-km_{i, j}}$.
      But then if $\widetilde{I} \subseteq A$ is the ideal
      generated by $I_R$ and all $\gamma(x_{i, j})$,
      we have $\widetilde{I}_{a_i} = I_{a_i}$
      for all $i$,
      which implies $\widetilde{I} = I$.
      Thus, $R \to A$ is of finite PD type
      by the other direction of Lemma \ref{lemma-finite-type}.

      Now let the $R \to A_{a_i}$ be of finite PD presentation.
      We already know
      $A = R\angles{\vec{X}}[\vec{Y}] / \rho$
      for a PD saturated ideal $\rho$.
      For every $i$,
      choose a representative
      $\widetilde{a}_i \in R\angles{\vec{X}}[\vec{Y}]$
      of $a_i \in A$.
      We note that $\rho_{\widetilde{a}_i}$ is also saturated,
      and therefore
      $A_{a_i} = R\angles{\vec{X}}[\vec{Y}]_{\widetilde{a}_i} /
      \rho_{\widetilde{a}_i}$.
      Since both $A_{a_i}$
      and $R\angles{\vec{X}}[\vec{Y}]_{\widetilde{a}_i}$
      are of finite PD presentation over $R$,
      Lemma \ref{lemma-finite-presentation} tells us that
      $\rho_{\widetilde{a}_i} = \overline{(r_{i, j} a_i^{-n_{i, j}})}
      = \overline{(r_{i, j})}$,
      with finitely many $r_{i, j} \in \rho$.
      Let $c_i \in R\angles{\vec{X}}[\vec{Y}]$
      be elements with $1 - \sum_i c_i a_i \in \rho$.
      Then,
      setting $\widetilde{\rho} \defeq
      \overline{(1 - \sum_i c_i a_i, r_{i, j})}
      \subseteq R\angles{\vec{X}}[\vec{Y}]$,
      we have $\widetilde{\rho} \subseteq \rho$,
      $(\widetilde{a}_i) = (1)$
      in $R\angles{\vec{X}}[\vec{Y}]/\widetilde{\rho}$
      and $\widetilde{\rho}_{\widetilde{a}_i} = \rho_{\widetilde{a}_i}$
      for all $i$.
      This implies $\widetilde{\rho} = \rho$.
      \qedhere
  \end{enumerate}
\end{proof}

\begin{definition}
  A morphism of PD schemes
  $(T, \mathcal{I}, \gamma) \to (S, \mathcal{I}_S, \gamma_S)$
  is \emph{locally of finite PD type (presentation)}
  if there are affine open coverings
  $T = \bigcup_i \Spec A_i$
  and $S = \bigcup_i \Spec K_j$
  such that every $\Spec A_i$
  maps to some $\Spec K_{j_i}$
  where $K_{j_i} \to A_i$
  (with the induced structures of PD rings)
  is of finite PD type (presentation).
\end{definition}

The following good properties
of morphisms locally of finite PD type
and morphisms locally of finite PD presentation
now follow from Lemmas
\ref{lemma-finite-type-presentation-composition},
\ref{lemma-finite-type-presentation-base-change}
and \ref{lemma-finite-type-presentation-local}
as in \citestacks{01SQ}:
If $T \to S$ is locally of finite PD type (presentation),
then for any affine open $\Spec A \subseteq T$
mapping into an affine open $\Spec K \subseteq S$,
the homomorphism of PD rings $K \to A$
is of finite PD type (presentation).
In particular,
the morphisms locally of finite PD type (presentation)
between affine PD schemes
correspond to the morphisms
of finite PD type (presentation)
between PD rings.
And morphisms locally of finite PD type (presentation)
are stable under composition and base change.

\begin{definition}
  We say that the PD ideal $I$
  of a PD ring $(A, I, \gamma)$
  is \emph{PD-generated} by a subset $G \subseteq I$
  if $\overline{(G)} = I$,
  that is,
  if it is generated (as an ideal)
  by the elements $\gamma_n(g)$
  for $g \in G$ and $n \geq 1$.
  We say $I$ is \emph{finitely PD-generated}
  if it is PD-generated by a finite set $G$.
\end{definition}

We also globalize the notion of a PD ideal
being finitely PD-generated.
Note that by Lemma \ref{lemma-finite-presentation},
if $(A, I, \gamma)$ is a PD ring
and we equip $A/I$ with the trivial PD structure,
then $I$ is finitely PD-generated
if and only if
$A \to A/I$ is of finite PD presentation.

\begin{definition}
  The PD ideal sheaf $\mathcal{I}$
  of a PD scheme $(S, \mathcal{I}, \gamma)$
  is \emph{locally finitely PD-generated}
  if for every affine open $U \subseteq S$,
  the PD ideal $\mathcal{I}(U) \subseteq \mathcal{O}_S(U)$
  is finitely PD-generated,
  that is,
  if the closed embedding $V(\mathcal{I}) \to S$
  is locally of finite PD presentation,
  where $V(\mathcal{I})$ is regarded as a PD scheme
  with trivial PD structure.
\end{definition}

\subsection{Syntactic presentation of the big crystalline topos}

We can now finish the definition
of the precise variant of the crystalline topos
for which we will give a syntactic presentation.
The objects we want to allow in the site
are those PD thickenings $(T, \mathcal{I}, \gamma)$
over $S$ and $X$
where $T \to S$ is locally of finite PD presentation.
An essentially small dense subcategory is for example
the one where $T$ is additionally required
to be an open subscheme of an affine scheme
which is mapped into an affine open of $S$.
We also impose a finiteness assumption on $X$
without which the connection to compact models
motivating this definition
would be lost.

\begin{definition}
  \label{definition-Cris-fp}
  Let $(S, \mathcal{I}_S, \gamma_S)$ be a PD scheme
  and let $X$ be a scheme
  locally of finite presentation over $S_0 = V(\mathcal{I}_S)$.
  Then we set
  \[ (X / S)_{\Cris_{\fp}} \defeq (X / S)_{\Cris_\mathcal{C}}
     \rlap{,} \]
  where $\mathcal{C}$ is
  the class of PD thickenings $(T, \mathcal{I}, \gamma)$
  over $S$ and $X$
  for which $T \to S$ is locally of finite PD presentation.
\end{definition}

Now there is not much missing for the proof of
the classification result.
The following lemma will be used to show that
the axiom $(\nil)$ is a rigid-topology quotient.

\begin{lemma}
  \label{lemma-pd-generators-nil-ideal}
  Let $(A, \mathfrak{a}, \gamma)$ be a PD ring.
  If $\mathfrak{a}$ is PD-generated by nilpotent elements,
  then $\mathfrak{a}$ is a nil ideal.
\end{lemma}

\begin{proof}
  We need to show that $\mathfrak{a}$ is generated
  (as an ordinary ideal)
  by nilpotent elements.
  So if $a \in \mathfrak{a}$ is nilpotent,
  say $a^e = 0$,
  we need to show that $\gamma_n(a)$ is still nilpotent
  for every $n \geq 1$.
  It is most convenient to do the necessary calculation
  in $\QQ[X]$ first
  (with the unique PD structure $\gamma$ on any ideal):
  \[ (\gamma_n(X))^k = \frac{1}{(n!)^k} X^{kn} =
     \frac{(kn - e)!}{(n!)^k} X^e \gamma_{kn - e}(X)
     \rlap{.} \]
  One can check that $c \defeq \frac{(kn - e)!}{(n!)^k}$
  is an integer for $k$ big enough
  ($k \geq \lceil\frac{e}{n}\rceil (n + 1)$ suffices).
  But then we have the equation
  $(\gamma_n(X))^k = c X^e \gamma_{kn - e}(X)$
  also in the sub-PD-ring
  $\ZZ\angles{X} \subseteq \QQ[X]$,
  and then,
  via the unique PD morphism $\ZZ\angles{X} \to A$
  sending $X$ to $a$,
  we obtain $(\gamma_n(a))^k = 0$ for sufficiently large $k$.
\end{proof}

\begin{proposition}
  \label{proposition-pd-thickenings-presheaf-type}
  Let $(K, I_K, \gamma_K)$ be a PD ring
  with $I_K$ finitely PD-generated
  and let $R$ be a finitely presented $K/I_K$-algebra.
  Then the theory
  $\AlgQuot{K}{R} + \PD_{\gamma_K} + (\nil)$
  is of presheaf type
  and its compact models are those
  $(\gamma \curvearrowright I \triangleleft A \twoheadrightarrow B)$
  where $(A, I, \gamma)$ is of finite PD presentation
  over $(K, I_K, \gamma_K)$.
\end{proposition}

\begin{proof}
  We start with the cartesian theory
  $\Alg{K} + \Ideal_{I_K} + \PD_{\gamma_K}$.
  By Lemma \ref{lemma-compact-pd-algebras},
  the compact models are the PD rings
  of finite PD presentation over $(K, I_K, \gamma_K)$.
  It is an equivalence extension
  to add a sort $B$ with $K$-algebra structure
  for the quotient $A / I$,
  arriving at $\AlgQuot{K}{(K/I_K)} + \PD_{\gamma_K}$.
  Since $R$ is a finitely presented $K/I_K$-algebra,
  finitely many constant symbols
  and finitely many axioms in the empty context
  suffice to add an $R$-algebra structure to $B$,
  so by Proposition \ref{proposition-add-constant-symbol}
  and Corollary \ref{corollary-axiom-in-empty-context},
  $\AlgQuot{K}{R} + \PD_{\gamma_K}$
  is of presheaf type
  and a model
  $(\gamma \curvearrowright I \triangleleft A \twoheadrightarrow B)$
  is still compact if and only if
  $(A, I, \gamma)$ is of finite PD presentation over $K$.
  All that remains is to show that $(\nil)$
  is a rigid-topology quotient of this theory.

  Let a compact model $M =
  (\gamma \curvearrowright I \triangleleft A \twoheadrightarrow B)
  \in (\AlgQuot{K}{R} + \PD_{\gamma_K})\dashmod(\Set)_c$
  be given.
  Since $I_K$ is finitely PD-generated
  and $A$ is of finite PD presentation over $K$,
  the PD ideal $I$ is finitely PD-generated.
  Let $a \in I$ be one of the PD generators
  in a chosen finite family.
  The $J_{(\nil)}$-covering cosieve $S_a$
  of Theorem \ref{theorem-induced-topology}
  is the cosieve of all arrows $M \to M'$
  that send $a$ to an nilpotent element;
  it is generated by the countable family
  of the arrows
  $M \to M^n =
  (\gamma \curvearrowright I/\overline{(a^n)} \triangleleft
  A/\overline{(a^n)} \twoheadrightarrow B)$.
  In each of the (still compact) models $M^n$,
  the PD ideal $I/\overline{(a^n)}$
  is PD generated by the images of the PD generators of $I$,
  so we can apply this construction to the other PD generators in turn
  to cover $M$ by models where
  the PD ideal is generated by nilpotent elements
  and where thus,
  by Lemma \ref{lemma-pd-generators-nil-ideal},
  the axiom $(\nil)$ is fulfilled.
\end{proof}

\begin{theorem}
  \label{theorem-affine-crystalline-classifies}
  Let $(K, I_K, \gamma_K)$ be a PD ring
  with $I_K$ finitely PD-generated
  and let $R$ be a finitely presented $K/I_K$-algebra.
  Then the big crystalline topos
  \[ (\Spec R / \Spec K)_{\Cris_\fp} \]
  classifies the geometric theory
  \[ \AlgQuot{K}{R} + \PD_{\gamma_K} + (\nil) + (\loc)
     \rlap{.} \]

  The universal model is
  the short exact sequence
  associated to the structure sheaf
  \[ \mathcal{J} \hookrightarrow \mathcal{O}
     \twoheadrightarrow \mathcal{O}'
     \rlap{,} \]
  with the canonical $K$-algebra structure on $\mathcal{O}$
  and $R$-algebra structure on $\mathcal{O}'$.
\end{theorem}

\begin{proof}
  Set
  \[ \TT_0 \defeq \AlgQuot{K}{R} + \PD_{\gamma_K} + (\nil)
     \rlap{.} \]
  The underlying category of
  the site of affine objects for $(\Spec R / \Spec K)_{\Cris_\fp}$
  is dual to the subcategory of $\TT_0\dashmod(\Set)$
  where $(A, I, \gamma)$ is of finite PD presentation over $K$.
  By Proposition \ref{proposition-pd-thickenings-presheaf-type},
  this is exactly $\TT_0\dashmod(\Set)_c$,
  and the associated presheaf topos
  \[ [\TT_0\dashmod(\Set)_c, \Set] \]
  classifies $\TT_0$.
  The Zariski topology on this site,
  that is,
  the restriction of the one in
  Lemma \ref{lemma-affine-Zariski-topology-for-Cris}
  to our compact models,
  coincides with the topology induced by $(\nil)$
  via Theorem \ref{theorem-induced-topology}
  (just as in Corollary \ref{corollary-lazy-classification-result}).

  For the universal model,
  we note that the structure sheaf model
  $\mathcal{J} \hookrightarrow \mathcal{O}
  \twoheadrightarrow \mathcal{O}'$,
  when restricted to the site of affine objects,
  becomes simply the tautological model
  in $[\TT_0\dashmod(\Set)_c, \Set]$,
  which is the universal model of $\TT_0$.
  In particular,
  the presheaves of which the universal $\TT_0$-model consists
  are sheaves for the topology induced by $(\loc)$,
  implying that the same model,
  regarded as a model in the subtopos,
  is also the universal model of $\TT_0 + (\loc)$.
\end{proof}

\begin{corollary}
  \label{corollary-points-of-affine-crystallline}
  Let $K$, $R$ be as in
  Theorem \ref{theorem-affine-crystalline-classifies}.
  Then the category of points of the big crystalline topos
  $(\Spec R / \Spec K)_{\Cris_\fp}$
  is equivalent to
  \[ \Bigl( \AlgQuot{K}{R} + \PD_{\gamma_K} + (\nil) + (\loc)
     \Bigr)\dashmod(\Set)
     \rlap{,} \]
  the category of affine PD thickenings $(A, I, \gamma)$
  over $(K, I_K, \gamma_K)$
  (implying that $I$ is a nil ideal)
  together with an $R$-algebra structure on $A/I$,
  such that $A$ is a local ring.
\end{corollary}

\begin{proof}
  The points of a topos are by definition
  the geometric morphisms from $\Set$ to it.
  So this follows immediately from
  Theorem \ref{theorem-affine-crystalline-classifies}
  and the definition of classifying topos.
\end{proof}

\begin{remark}
  If the ring $R$ has non-zero characteristic,
  then there is a different way of dealing with the axiom $(\nil)$.
  Namely,
  one can see that if the PD ideal $I$ of $A$
  contains a non-zero integer,
  then $(\nil)$ is actually equivalent to
  \[ \top \turnstile{[]} \bigvee_{0 \neq n \in \ZZ} (c_n = 0)
     \rlap{,} \]
  which is an axiom in the empty context.
  So in Theorem \ref{theorem-affine-crystalline-classifies},
  instead of assuming that $I_K$ is finitely PD-generated,
  we can alternatively assume $R$ to have non-zero characteristic.
\end{remark}

\subsection{The non-affine case}

Now we apply the results of Section \ref{section-gluing}
to the big crystalline topos
to obtain a syntactic presentation of $(X/S)_\Crisfp$
in the case where $S$ and $X$ are not affine.
This will be an adaptation of
the treatment of the Zariski topos
in Section \ref{section-gluing}
to the crystalline topos.
and we will mostly be concerned with the aspects which differ
in these two situations
and how to handle them.

Since we only know a theory classified by $(X/S)_\Crisfp$
in the case where both $S$ and $X$ are affine,
we will want to cover both $S$ and $X$
by affine opens $S_i$ respectively $X_i$
such that for each $i$,
$X_i$ lies over $S_i$.
\[ \begin{tikzcd}
  S & X \ar[l] \\
  S_i \ar[u, hook] &
  X_i \ar[u, hook] \ar[l]
\end{tikzcd} \]
This is always possible,
as we can first choose an affine cover of $S$
by some $S_i$
and then cover the $X \times_S S_i$ by affine opens.
Also,
an open subscheme $S' \subseteq S$ is canonically a PD scheme
as $(S', \mathcal{I}_S|_{S'}, \gamma_S|_{S'})$,
and we note that if $X' \to S$ factors through $S'$,
then it automatically also factors through
$S'_0 = V(\mathcal{I}_S|_{S'})$.

When we turn to defining open subtoposes
corresponding to such pairs of open subschemes $(S', X')$,
there are two things to note.
Firstly,
it is no longer the case that $(S', X')$
is an object of the site $\Crisfp(X/S)$,
since $X'$ is not at all required to be isomorphic to $S'_0$,
and the PD scheme $(S', \mathcal{I}_S|_{S'}, \gamma_S|_{S'})$ itself
does not have to be a PD thickening.
But we can still consider the class of all objects
$(T, \mathcal{I}, \gamma) \in \Crisfp(X/S)$
such that $T \to S$ factors through $S'$
and $V(\mathcal{I}) \to X$ factors through $X'$,
that is,
the subterminal presheaf
\[ U_{(S', X')} \defeq \Hom_{X/S}(\hole, (S', X'))
   \;:\; \Crisfp(X/S)^\op \to \Set
   \rlap{,} \]
and it is easy to see that this is in fact a sheaf
for the Zariski topology on $\Crisfp(X/S)$.

Secondly, however,
this subterminal object $U_{(S', X')}$
depends in fact only on $X'$ and not on $S'$,
namely,
we have $U_{(S', X')} = U_{(S, X')}$.
This is because the closed embedding
$V(\mathcal{I}) \hookrightarrow T$
is a homeomorphism,
and whether or not $T \to S$ factors through
the open subscheme $S'$
can be tested on the level of points.
\[ \begin{tikzcd}
  S & X \ar[l] \\
  S' \ar[u, hook] & X' \ar[u, hook] \ar[l] \\
  T \ar[u, dashed] \ar[uu, bend left=45] &
  V(\mathcal{I}) \ar[u] \ar[l, hook]
\end{tikzcd} \]
Thus, we only have induced open subtoposes
$U_{X'} \defeq U_{(S, X')}$
for all open subschemes $X' \subseteq X$.

This leads to the slightly subtle situation
that a cover of $X$ by affine open $X_i$
with the property that every $X_i \to S$
factors through some affine open $S_i \subseteq S$
induces an open cover of $(X/S)_\Crisfp$
by subtoposes for which we know classified theories,
but if $S_i = \Spec K_i$ and $S'_i = \Spec K'_i$
are two different choices for the same $X_i = \Spec R_i$,
then we have no direct algebraic relation between $K_i$ and $K'_i$,
so we have no way to construct a diagonal extension
between the classified theories.
Therefore it seems more reasonable
to work with pairs $(S_i, X_i)$ of open subschemes from the beginning
and only drop the condition that the $S_i$ cover $S$.

Here is the analogue of
Lemma \ref{lemma-for-gluing-Zar}
that we need.

\begin{lemma}
  \label{lemma-for-gluing-Cris}
  \leavevmode
  \begin{enumerate}[label=(\roman*)]
    \item
      The mapping
      \[ X' \mapsto ((X/S)_\Crisfp)_{o(U_{X'})} \]
      from open subschemes of $X$ to open subtoposes of $(X/S)_\Zarfp$
      is monotone
      and preserves finite intersections
      and arbitrary unions.

    \item
      For open subschemes $S' \subseteq S$ and $X' \subseteq X$
      with $X' \to S'$,
      the open subtopos $((X/S)_\Crisfp)_{o(U_{X'})}$
      is equivalent to $(X'/S')_\Crisfp$.

    \item
      If $S = \Spec K$ and $X = \Spec R$ are affine,
      with $I_K$ finitely PD-generated
      and $R$ finitely presented over $K/I_K$,
      so that $(X/S)_\Crisfp$ classifies
      \[ \TT_{K, R} \;\defeq\;
         \AlgQuot{K}{R} + \PD_{\gamma_K} + (\nil) + (\loc)
         \rlap{,} \]
      and $h \in R$,
      then the open subtopos $((X/S)_\Crisfp)_{o(U_{D(h)})}$
      is presented by the closed geometric formula
      \[ \inv(c_h) \rlap{,} \]
      which is also equivalent to
      \[ \inv(c_g \oftype A) \land \inv(c_h \oftype B) \]
      for any $g \in K$ such that $g \divides h$ in $R$.
  \end{enumerate}
\end{lemma}

\begin{proof}
  \begin{enumerate}[label=(\roman*)]
    \item
      Monotonicity and finite intersections are clear.
      For unions,
      let $X'_i$ be a family of open subschemes of $X$
      and let $F \subseteq 1_{(X/S)_\Crisfp}$
      be a subterminal sheaf
      with $U_{X'_i} \leq F$ for all $i$.
      We want to show $U_{X'} \leq F$ for $X' \defeq \bigcup_i X'_i$.
      So let $T = (T, \mathcal{I}, \gamma) \in \Crisfp(X/S)$
      be given with $V(\mathcal{I}) \to X'$.
      By pulling back the $X'_i$,
      this induced an open cover of $V(\mathcal{I})$,
      or equivalently an open cover $T = \bigcup_i T_i$,
      such that $V(\mathcal{I}|_{T_i}) \to X_i$.
      This means $\abs{U_{X_i}(T_i)} = 1$,
      which implies $\abs{F(T_i)} = 1$,
      and since $F$ is a sheaf for the Zariski topology on $\Crisfp(X/S)$,
      we obtain $\abs{F(T)} = 1$.

    \item
      As explained above,
      $U_{X'}$ selects exactly those objects of $\Crisfp(X/S)$
      which belong to $\Crisfp(X'/S')$.
      The rest is exactly as in the case of the Zariski topos.

    \item
      We know that $(X/S)_\Crisfp$ classifies $\TT_{K, R}$
      and
      \[ ((X/S)_\Crisfp)_{o(U_{D(h)})} \simeq (D(h)/S)_\Crisfp \]
      classifies $\TT_{K, R_h}$,
      which is equivalent to $\TT_{K, R} + \inv(c_h)$.
      It follows that $\inv(c_h)$ presents this open subtopos,
      since one structure sheaf
      pulls back to the other,
      including the extra structure that makes up the universal models.
      For $g \in K$ with $g \divides h$ in $R$,
      the theory $\TT_{K, R}$ shows
      \[ \inv(c_h) \turnstile{[]} \inv(f(c_g)) \]
      and also,
      by Remark \ref{remark-loc-with-surj-and-nil},
      \[ \inv(f(c_g)) \doubleturnstile{[]} \inv(c_g) \rlap{.} \]
  \end{enumerate}
\end{proof}

As the final prerequisite,
we show that intersections of pairs of open subschemes
can be covered by appropriate pairs of standard opens.

\begin{lemma}
  \label{lemma-cover-by-pairs-of-standard-opens}
  Let $X \to S$ be a morphism of schemes,
  let $S_1, S_2 \subseteq S$
  and let $X_1, X_2 \subseteq X$
  be affine open subschemes,
  $S_i = \Spec K_i$,
  $X_i = \Spec R_i$,
  such that $X_i \to S_i$.
  Then there are families of elements
  $g_i^j \in K_i$, $h_i^j \in R_i$, $j \in J$, $i \in \{1, 2\}$,
  such that
  for every $j \in J$
  we have $g_i^j \divides h_i^j$ in $R_i$
  and
  \[ \Spec K_1 \supseteq D(g_1^j) = D(g_2^j) \subseteq \Spec K_2
     \rlap{,} \]
  \[ \Spec R_1 \supseteq D(h_1^j) = D(h_2^j) \subseteq \Spec R_2 \]
  as open subschemes of $S$ respectively $X$,
  and
  \[ X_1 \cap X_2 =
     \bigcup_{j \in J} D(g_1^j) =
     \bigcup_{j \in J} D(g_2^j)
     \rlap{.} \]
\end{lemma}

\begin{proof}
  As mentioned before,
  we can cover $S_1 \cap S_2$ by open subschemes
  simultaneously standard open in $S_1$ and $S_2$.
  The same is true for $X_1 \cap X_2$,
  and since the simultaneously standard opens
  even form a basis of opens of $X_1 \cap X_2$,
  we can also arrange that $D(h_1^j)$ maps to $D(g_1^j)$,
  and, equivalently, $D(h_2^j)$ maps to $D(g_2^j)$.
  This means that
  for each $i$ and $j$,
  we have the dashed arrow in
  \[ \begin{tikzcd}
    R_i \ar[r] & (R_i)_{h_i^j} \\
    K_i \ar[u] \ar[r] &
    (K_i)_{g_i^j} \ar[u, dashed]
    \rlap{,}
  \end{tikzcd} \]
  so $g_i^j$ is invertible in $(R_i)_{h_i^j}$,
  so $g_i^j$ divides some power of $h_i^j$,
  and replacing $h_i^j$ with this power,
  we obtain $g_i^j \divides h_i^j$ in $R_i$.
\end{proof}

Recall the notation $\AlgStr{R}_K(B)$
for the extension adding
to a sort $B$
an $R$-algebra structure
compatible with a previously given $K$-algebra structure,
from Definition \ref{definition-extensions-of-K-Alg}.
For a PD ring $(K, I_K, \gamma_K)$,
we also denote
$\gamma/\gamma_K$
the axioms
distinguishing the extension $\PD_{\gamma_K}$
of $\Alg{K} + \Ideal_{I_K}$
from $\PD$.

\begin{theorem}
  \label{theorem-gluing-Cris}
  Let $(S, \mathcal{I}_S, \gamma_S)$ be a PD scheme
  with $\mathcal{I}_S$ locally finitely PD-generated,
  and let $X$ be a scheme locally of finite presentation
  over $S_0 = V(\mathcal{I}_S)$.
  Let
  \[ \Spec K_i = S_i \subseteq S, \qquad
     \Spec R_i = X_i \subseteq X \]
  be open subschemes
  such that $X_i \to S_i$
  and $X = \bigcup_{i \in I} X_i$.
  For every $i \neq i' \in I$,
  let
  \[
    g_{i, i'}^j \in K_i, \quad
    g_{i', i}^j \in K_{i'}, \quad
    h_{i, i'}^j \in R_i, \quad
    h_{i', i}^j \in R_{i'},
  \]
  $j \in J_{\{i, i'\}}$,
  be families of elements as in
  Lemma \ref{lemma-cover-by-pairs-of-standard-opens},
  with corresponding ring isomorphisms
  \[ \varphi_{i, i'}^j = (\varphi_{i', i}^j)^{-1} :
     (K_i)_{g_{i, i'}^j} \to (K_{i'})_{g_{i', i}^j}, \qquad
     \widetilde{\varphi}_{i, i'}^j =
     (\widetilde{\varphi}_{i', i}^j)^{-1} :
     (R_i)_{h_{i, i'}^j} \to (R_{i'})_{h_{i', i}^j}
     \rlap{.} \]

  Then $(X/S)_\Crisfp$ classifies the geometric theory
  \begin{align*}
    \TT_{X/S} \;\defeq\;
    {}&\AlgQuot{\ZZ}{\ZZ} + \PD + (\nil) + (\loc)
    + \angles{p_i}_{i \in I}
    + (\bigvee_{i \in I} p_i)
    \\
    &+ \Bigl( \Bigl( \AlgStr{K_i}(A) + \AlgStr{R_i}_{K_i}(B)
                     + \gamma/\gamma_{K_i}
              \Bigr)/p_i
       \Bigr)_{i \in I}
    \\
    &+ \Bigl(
         x \in \widetilde{c}_{h_{i, i'}^j} \land \inv(x)
         \turnstile{x \oftype B}
         p_{i'}
       \Bigr)_{i \neq i' \in I, j \in J_{\{i, i'\}}}
    \\
    &+ \Bigl(
         p_i \land p_{i'}
         \turnstile{[]}
         \bigvee_{j \in J_{\{i, i'\}}} \ex{x}{B}
           (x \in \widetilde{c}_{h_{i, i'}^j} \land \inv(x))
       \Bigr)_{i \neq i' \in I, j \in J_{\{i, i'\}}}
    \\
    &+ \Bigl(
         x \in \widetilde{c}_{g_{i', i}^j} \land
         \inv(x) \land
         y \in \widetilde{c}_\lambda \land
         z \in \widetilde{c}_{\lambda'}
         \turnstile{x, y, z \oftype A}
         x^n y = z
       \Bigr)_{i \neq i' \in I, j \in J_{\{i, i'\}}, \lambda \in K_i}
    \\
    &+ \Bigl(
         x \in \widetilde{c}_{h_{i', i}^j} \land
         \inv(x) \land
         y \in \widetilde{c}_\mu \land
         z \in \widetilde{c}_{\mu'}
         \turnstile{x, y, z \oftype B}
         x^n y = z
       \Bigr)_{i \neq i' \in I, j \in J_{\{i, i'\}}, \mu \in R_i}
       \rlap{,}
  \end{align*}
  where in the last two families of axioms,
  $\lambda' \in K_{i'}$ and $n \in \NN$
  are chosen for each $\lambda \in K_i$
  such that
  $\varphi_{i, i'}^j(\lambda) = (g_{i', i}^j)^{-n} \lambda'$,
  and
  $\mu' \in R_{i'}$ and $n \in \NN$
  are chosen for each $\mu \in R_i$
  such that
  $\widetilde{\varphi}_{i, i'}^j(\mu) = (h_{i', i}^j)^{-n} \mu'$,
\end{theorem}

\begin{proof}
  Since $\mathcal{I}_S$ is locally finitely PD-generated,
  the PD ideals $I_{K_i} \defeq \mathcal{I}_S(S_i) \subseteq K_i$
  with PD structure $\gamma_{K_i} \defeq \gamma_S(S_i)$
  are finitely PD-generated,
  and since $X$ is locally of finite presentation over $S_0$,
  each $R_i$ is a finitely presented $K_i/I_{K_i}$-algebra.
  So $(X_i/S_i)_\Crisfp$ classifies
  \[ \TT_{K_i, R_i} =
     \AlgQuot{K_i}{R_i} + \PD_{\gamma_{K_i}} + (\nil) + (\loc)
     \rlap{.} \]

  By Lemma \ref{lemma-for-gluing-Cris},
  these toposes
  \[ \E_i \defeq ((X/S)_\Crisfp)_{o(U_{X_i})}
     \simeq (X_i/S_i)_\Crisfp \]
  form an open cover of $(X/S)_\Crisfp$,
  and the universal models are extensions of
  the restrictions of the short exact sequence
  $\mathcal{J} \hookrightarrow \mathcal{O}
  \twoheadrightarrow \mathcal{O}'$
  surrounding the structure sheaf.
  So we can use the base theory
  and base model
  \begin{align*}
    \TT_0 &\defeq \AlgQuot{\ZZ}{\ZZ} + \PD + (\nil) + (\loc)
    \rlap{,} \\
    M_0 &\defeq
    (\mathcal{J} \hookrightarrow \mathcal{O}
    \twoheadrightarrow \mathcal{O}')
    \rlap{,}
  \end{align*}
  where $\TT_0$ is of course equivalent to
  $\Ring + \Ideal + \PD + (\nil) + (\loc)$,
  but we want to have the sort $B = A/I$
  in the base theory.
  And we have presentations of the $\E_i$
  over $(\TT_0, M_0)$
  with
  \[ \EE_i \defeq \AlgStr{K_i}(A) + \AlgStr{R_i}_{K_i}(B) +
      \gamma/\gamma_{K_i}
      \rlap{.} \]
  Closed geometric formulas for the intersections
  $\E_i \cap \E_{i'} \subseteq \E_i$
  are given by
  \[ \phi_{i, i'} \;\defeq\;
     \bigvee_{j \in J_{\{i, i'\}}} \inv(c_{h_{i, i'}^j})
     \rlap{.} \]

  The diagonal quotient extensions $\QQ_{\{i, i'\}}$
  have to make both the $K_{i'}$-algebra structure on $A$
  and the $R_{i'}$-algebra structure on $B$
  definable in terms of $\TT_0 + \EE_i + \phi_{i, i'}$.
  Since the formula $\phi_{i, i'}$ also implies
  \[ \bigvee_{j \in J_{\{i, i'\}}} \inv(c_{g_{i, i'}^j})
     \rlap{,} \]
  we can proceed as in Theorem \ref{theorem-gluing-Zar}
  for both $K_i$ and $R_i$.
  That is,
  for every $\lambda \in K_i$
  and $\mu \in R_i$,
  we write
  \[ \varphi_{i, i'}^j(\lambda) = (g_{i', i}^j)^{-n} \lambda' \]
  for some $\lambda' \in K_{i'}$, $n \in \NN$,
  respectively
  \[ \widetilde{\varphi}_{i, i'}^j(\mu) = (h_{i', i}^j)^{-n} \mu' \]
  for some $\mu' \in R_{i'}$, $n \in \NN$.
  Then $\mathcal{O}_{((X_i \cap X_{i'})/(S_i \cap S_{i'}))_\Crisfp}$
  satisfies
  \begin{align*}
    \QQ_{i, i'} \quad&\defeq\quad
    \Bigl\{\quad
      \inv(c_{g_{i', i}^j})
      \turnstile{[]}
      (c_{g_{i', i}^j})^n c_\lambda = c_{\lambda'}
    \quad\mid\quad
      j \in J_{\{i, i'\}},
      \lambda \in K_i
    \quad\Bigr\}
    \rlap{,} \\
    \widetilde{\QQ}_{i, i'} \quad&\defeq\quad
    \Bigl\{\quad
      \inv(c_{h_{i', i}^j})
      \turnstile{[]}
      (c_{h_{i', i}^j})^n c_\mu = c_{\mu'}
    \quad\mid\quad
      j \in J_{\{i, i'\}},
      \mu \in R_i
    \quad\Bigr\}
    \rlap{,}
  \end{align*}
  and
  \[ \QQ_{\{i, i'\}} \defeq \QQ_{i, i'} + \QQ_{i', i} +
     \widetilde{\QQ}_{i, i'} + \widetilde{\QQ}_{i', i} \]
  is a diagonal quotient extension
  of $\EE_i + \phi_{i, i'}$ and $\EE_{i'} + \phi_{i', i}$
  over $\TT_0$.

  Applying Corollary \ref{corollary-localic-gluing}
  to these data
  and simplifying the axioms slightly
  yields the theory $\TT_{X/S}$ in the statement.
\end{proof}

%
%

%
%

\newpage
\bibliography{dissertation.bib}

\end{document}